\documentclass[10pt]{amsart}
\usepackage[a4paper, bottom = 2.3cm, top= 2.5cm, left= 2.1 cm, right= 2. cm]{geometry}
\usepackage[T1]{fontenc}
\usepackage{graphicx}
\usepackage{indentfirst,csquotes}
\usepackage{amssymb,amsthm,amsmath}
\usepackage{xcolor,paralist,hyperref,titlesec,fancyhdr,etoolbox}
\newtheorem{theorem}{Theorem}[]
\newtheorem{definition}[theorem]{Definition}

\newtheorem{lemma}[theorem]{Lemma}
\newtheorem{proposition}[theorem]{Proposition}

\newtheorem{remark}[theorem]{Remark}

\newtheorem*{hyp}{Hypotheses}

\numberwithin{equation}{section}
\usepackage{xfrac}
\usepackage{cleveref}
\hypersetup{
    colorlinks=true,
    linkcolor=blue,
    filecolor=magenta,      
    urlcolor=cyan,
    citecolor=green!40!black,
    pdftitle={Generalized  Doubly Parabolic Keller-Segel System with Fractional Diffusion},
    pdfauthor={Bronzi, Anne C., and De Andrade, Crystianne L.},
    pdfsubject={},
    pdfkeywords={Keller-Segel model, Chemotaxis, Parabolic-parabolic system, Mild solutions,  Anomalous diffusion, Fractional Laplacian.} 
}
\usepackage[sorting=nyt,  
    autocite=inline, 
    labelnumber=true,
    style = numeric-comp, 
    natbib=true,		    
    backend=biber,		    
    sortcites=true,		    
    maxbibnames=99,
    maxcitenames=1,
    uniquename=init,
    uniquelist=false,
    giveninits=true,  
    useprefix=true,    
    isbn=false,
    url=false,
    abbreviate=true,
    doi=true]{biblatex}
\usepackage{enumitem} 
\setlist[enumerate]{label=\arabic*.}

\RequirePackage{xspace}
\newcommand{\ie}{i.\,e.,\xspace}
\usepackage{multicol}
\usepackage{tabularx}
\usepackage{comment}

\titleformat{\section}{\normalfont\centering\MakeUppercase}{\thesection}{0.2cm}{}
\titleformat{\subsection}{\normalfont\centering\MakeUppercase}{\thesubsection}{0.2cm}{}

\hypersetup{ colorlinks=true, linkcolor=black, filecolor=black, urlcolor=black }

\begin{document}
\thispagestyle{empty}
\title{Generalized  Doubly Parabolic Keller-Segel System with Fractional Diffusion} 
\author{Anne Bronzi$^1$ and Crystianne De Andrade$^2$}

{\let\thefootnote\relax
\footnotemark 
\footnotetext{$^1$\href{mailto:acbronzi@unicamp.br}{acbronzi@unicamp.br} 
}
\footnotemark 
\footnotetext{$^2$\href{mailto:c264386@dac.unicamp.br}{c264386@dac.unicamp.br}
}
}

\begin{abstract}
The Keller-Segel model is a system of partial differential equations that describes the movement of cells or organisms in response to chemical signals, a phenomenon known as chemotaxis.    
In this study, we analyze a doubly parabolic Keller-Segel system in the whole space $\mathbb{R}^d$, $d\geq 2$, where both cellular and chemical diffusion are governed by fractional Laplacians with distinct exponents.   
This system generalizes the classical Keller-Segel model by introducing superdiffusion, a form of anomalous diffusion. This extension accounts for nonlocal diffusive effects observed in experimental settings, particularly in environments with sparse targets.   
We establish results on the local well-posedness of mild solutions for this generalized system and global well-posedness under smallness assumptions on the initial conditions in $L^p(\mathbb{R}^d)$. Furthermore, we characterize the asymptotic behavior of the solution.

\vspace{0.35cm} \noindent\textbf{key words.} Keller-Segel model, chemotaxis, parabolic-parabolic system, mild solutions,  anomalous diffusion, Fractional Laplacian. \vspace{0.35cm}

\end{abstract} 

\maketitle
\bigskip



\section{Introduction}
The foundation of mathematical modeling for chemotaxis in cellular systems can be traced back to the work of Keller and Segel (1970) \citep{Keller1970Initiation} and Patlak (1953) \citep{Patlak-1953}, with the Keller-Segel model standing out as the most extensively studied one. This model is represented by a system of partial differential equations that describes the chemically driven movement of cell (or organism) density, $\rho=\rho (x,t)$, towards regions with higher concentrations of a chemical signal, known as a chemoattractant, with density $c=c(x,t)$. 
In this paper we investigate the existence, uniqueness, and properties  of nonnegative solutions for the following generalized Keller-Segel model in $\mathbb R^d$, $d \geq 2$, with nonlocal diffusion terms with exponents $\alpha \in (1,2]$ and $\beta \in (1,d]$:
\begin{equation}
    \label{eq-work}
    \left\{\begin{aligned}
    &\partial_{t} \rho =- \Lambda^{\alpha} \rho-\chi\nabla \cdot \left(\rho \nabla c\right)  & x \in \mathbb{R}^d,  & \quad t>0, \\
    &\tau \partial_{t} c =- \Lambda^{\beta} c+ \rho-\gamma c  \qquad & x \in \mathbb{R}^d,  & \quad t>0, \; 
    \end{aligned}\right.
\end{equation}
where $\chi, \gamma$, and $\tau$ are nonnegative constants. 

In this system, the right-hand side of the first equation incorporates two key components: cell diffusion and chemotactic flux. The diffusion term captures the random microscopic movements of cells that collectively result in observable macroscopic motion. The chemotactic flux, modeled as an advective term, reflects the directed movement of cells toward higher concentrations of the chemical signal, \ie $\rho$ migrates in the direction of the chemical gradient, with $\chi$ representing the chemotactic sensitivity (assumed to be $1$ for simplicity in this study). This flux drives aggregation due to the chemical's presence, a process that is mathematically interpreted as finite-time blowup.  
The second equation models the chemoattractant's behavior, encompassing its diffusion, production, and consumption (or degradation). This signal is emitted by cells, diffused throughout the environment, and degraded at a rate proportional to its local concentration, with the nonnegative constant $\gamma$ describing this degradation rate \citep{KS-Review-2020, Bournaveas-2010-one-dimensional-fractional, book-base, PDE-Models}. 
 
The diffusion on both equations is modeled by the fractional Laplacian, $\Lambda^{\alpha}=(-\Delta)^{\alpha / 2}$, which can be defined in Fourier variables as 
\begin{equation}
    \label{eq-definition-fractional-Laplacian-fourier}
    \widehat{(\Lambda^{\alpha} f)}(\xi)=|\xi|^\alpha \widehat{f}(\xi), \; \; \forall \xi \in \mathbb R^d, \text{and } \; \alpha \in(0,2],
\end{equation}
for $f$ sufficiently smooth, where the Fourier transform of $f$ is given by $\displaystyle \mathcal{F} f(\xi)=\hat{f}(\xi)=\frac{1}{(2\pi)^{d/2}}\int_{\mathbb{R}^d} f(x) \mathrm{e}^{-\mathrm{i} x \cdot \xi} \mathrm{d} x.$

The classical Keller-Segel system ($\alpha=\beta=2$), both parabolic-parabolic and its parabolic-elliptic counterpart $(\tau=0)$, has been extensively mathematically studied. 
Since the work of \citet*{Keller1970Initiation}, this model has drawn the attention of mathematicians due to its qualitative behavior such as blowup, pattern formation, stabilization, etc.  
Briefly, it is known for the classical Keller-Segel model in $\mathbb{R}^{d}$, that in one dimension, this system has a globally smooth and unique solution that remains uniformly bounded for all $t \geq 0$; 
in higher dimensions the Keller-Segel system is critical in $L^{d/2}(\mathbb{R}^d)$, meaning that a small initial condition in $L^{d/2}(\mathbb{R}^d)$ ensures global well-posedness in time, whereas a large mass leads to blow-up \cite{KS-Review-2020,Horstmann03from1970, book-base}. 
Specifically, in two dimensions, for the parabolic-elliptic Keller-Segel system ($\tau=0$) with $\gamma=0$, solutions exist globally in time if the total mass $m_0=\|\rho_0 \|_{L^1}$ satisfies $m_0 < 8 \pi/\chi$; otherwise, if $m_0 > 8 \pi/\chi$, the solution becomes unbounded in finite time $T>0$, \ie $\\|\rho(\cdot, t)\|_{L^{\infty}} \rightarrow\infty$ as $t\rightarrow T$ \citep{Perthame19}.
For $d \geq 2$, there exists a unique local-in-time mild solution $\rho \in C([0, T), L^p(\mathbb{R}^d))$ for any initial data $\rho_0 \in L^p(\mathbb{R}^d)$ with $p > d/2$. Additionally, there is a constant $C$, small enough, such that when $\left(\chi m_2(0)\right)^{(d-2)/d} \leq C \,\chi  m_0$, where $m_2(t)$ is the moment of second order, the solution blows up \citep{Biler1995, book-base}.

For the classical parabolic-parabolic Keller-Segel model in two dimensions, \citet{Mizoguchi-2012} established that when the total initial mass satisfies $m_0<8 \pi/\chi$, and the initial data $\rho_0 \in L^1(\mathbb{R}^2) \cap L^{\infty}(\mathbb{R}^2)$ and $c_0 \in L^1(\mathbb{R}^2) \cap H^1(\mathbb{R}^2)$,  the solution exists globally in time. 
In the critical case $m_0 = 8 \pi/\chi$, global existence or finite-time blowup depends on additional properties of the initial data.
\citet{Mizoguchi-2012} also pointed out that the meaning of “critical mass”, which separates global existence and blow-up, differs between the parabolic-parabolic and parabolic-elliptic systems. 
In this context, $8 \pi/\chi$ is the critical mass in the sense that all solutions with $m_0<8 \pi/\chi$ exist globally in time, while there exists a solution with $m_0>8 \pi/\chi$ that blows up in finite time. 
\citet*{Large-mass-Biler-Corrias} constructed a forward self-similar nonnegative solutions and demonstrated that, under certain conditions, such solutions can exist globally even when the total mass exceeds $8 \pi/\chi$, offering a stark contrast to the behavior observed in the parabolic-elliptic case with $\gamma=0$.
In higher dimensions ($d \geq 3$), \citet{Perthame-2006,Perthame-2008} proved global existence for $\rho_0$ and $\nabla c_0$ sufficiently small in $L^p(\mathbb{R}^{d})$, with $p > d/2$, and $L^d(\mathbb{R}^{d})$, respectively. 

\subsection{Keller-Segel model with Fractional Diffusion}

The use of a fractional diffusion ($\Lambda^{\alpha}=(-\Delta)^{\alpha / 2}$, $0<\alpha<2$) in place of
the classical one ($-\Delta=\Lambda^{2}$) in Keller-Segel equations has been supported by significant theoretical and empirical evidence since the 1990s (see \cite{Jan-2017} and the references therein). 
%
In a biological setting, organisms often adopt Lévy walk (superdiffusion) as an efficient search strategy, particularly when seeking scarce or rare resources such as chemoattractants or food. The Lévy walk significant probability for long positional jumps corresponds, in this context, to the fact that organisms maintain motion in a single direction for much longer periods than in typical random walks, as noted in \citep{Estrada-Rodriguez-2018, Escudero-2006-Fractional-model, Jan-2017}.
%
These works and the references therein provide examples of organisms exhibiting this type of behavior in chemotactic responses.

To illustrate some of the available results, we first mention some studies for the parabolic-elliptic ($\tau=0$) version of the system \eqref{eq-work} in the one-dimensional case.
\citet{Escudero-2006-Fractional-model} proved the boundedness of solutions for $1<\alpha<2$. 
Then, \citet{Bournaveas-2010-one-dimensional-fractional} studied the one-dimensional case with $\alpha \in(0,2]$.  
They proved global existence for $\alpha=1$ under the assumption of small initial data in $L^{1}\left(\mathbb{R}\right)$, and for $\alpha < 1$, showed that solution may exist globally or may blowup depending on the size of the  $L^{1/\alpha}\left(\mathbb{R}\right)$ norm of the initial data.  
Since their strategy fails when $\alpha=1$, they did not describe the solution's behavior for large initial mass in this case. 
Then, upon numerical evidence, \citet{Bournaveas-2010-one-dimensional-fractional} conjectured a blowup of solutions for certain large data. However, \citet{Burczak-2016-fractional}  proved, at least in the periodic setting, that in this critical case ($\alpha=1$) the system is global-in-time for any initial data. It is important to point out that all the previous works considered $\beta=2$. 

In higher dimensions ($d \geq 2$), \citet{Biler-fractional-9} proved,  for $\alpha \in (1,2]$ and $\beta \in (1,d]$ with $\gamma=0$, the existence and uniqueness of a local-in-time mild solution $\rho \in C\hspace{-0.1cm}\left([0, T], L^p(\mathbb{R}^{d})\right)$ for initial condition $\rho_{0} \in L^{p}(\mathbb{R}^{d})$, where
$T=T\hspace{-0.1cm}\left(\left\|\rho_0\right\|_{L^p}\hspace{-0.05cm}\right)$. In the same manner, they demonstrated a unique global-in-time mild solution $\rho \in \hspace{-0.1cm}\left([0, \infty), L^p(\mathbb{R}^{d})\right)$ for $\rho_0 \in L^{d /(\alpha+\beta-2)}(\mathbb{R}^{d})$, provided that $\left\|\rho_0\right\|_{L^{d /(\alpha+\beta-2)}}$ is small enough. 

We also mention the works of \citet*{Biler-2018-fractional} and \citet{Huang-Liu-fractional-2016} with $\alpha \in (0,2)$ and $\beta=2$ for the study of existence and nonexistence of global-in-time solutions.
In the work of \citet{Biler-2018-fractional}, it is mentioned that, when $\gamma=0$ and $\beta=2$, the parabolic-elliptic system in $d\geq 2$ classifies as 
subcritical ($\alpha \in(1,2)$) and supercritical ($\alpha \in(0,1]$) case.  
The authors highlighted that several results on local and global-in-time solutions in the subcritical case
(such as those in 
\cite[Theorem 2.2]{Biler-fractional-15}, \cite[Theorem 1.1]{Biler-fractional-16}, \cite[Theorem 2.1]{Biler-fractional-9},  \cite[Theorem 2]{Small-data-parabolic-parabolic-and-parabolic-elliptic} and \cite[Section 7]{Biler-2018-fractional}) in different functional spaces (Lebesgue, Besov, Morrey) are, in a general sense, analogous to those in which $\alpha=2$.
%
Furthermore, \citet{Biler-fractional-16} showed results on local-in-time and blow-up solutions in $d=2$ for initial data in critical Besov spaces $\dot{B}_{2, r}^{1-\alpha}(\mathbb{R}^2)$ with $r \in[1, \infty]$. 

The doubly parabolic case with fractional operators where $\alpha=\beta$ has been addressed by \citet*{Biler-fractional-16, doubly-fractional-Gan-2011, doubly-fractional-signal-dependent-2019, Xi-Wang-2018}. 
%
\citet{Biler-fractional-16} considered the model in $\mathbb{R}^2$ with $\alpha \in (1, \; 2)$ in Besov spaces and established the local existence of a unique solution.  \citet{doubly-fractional-Gan-2011} considered the model in $\mathbb{R}^d$ with $\alpha \in (1, \; 2]$ in Fourier-Herz spaces, proving local well-posedness and a global well-posedness result for a small initial data.
%
Furthermore, \citet*{doubly-fractional-signal-dependent-2019} studied a doubly parabolic model on $\mathbb{R}^d$, $d \geq 2$, with signal-dependent sensitivity and a source term, in Sobolev space, considering $\alpha \in (0, \; 2)$. They showed the existence, uniqueness, and temporal decay of the global classical solution under smallness assumptions on the initial data.
%
Similarly, \citet*{Xi-Wang-2018}, examining a doubly parabolic system with fractional Laplacian $\alpha \in (4/3, \; 2)$ and $d\leq 2$ in Sobolev space, proved the existence and the uniqueness of global classical solution under the assumption of a small enough initial data, and showed the asymptotic decay behaviors of $\rho$ and $\nabla c$. 

To the best of our knowledge, the question of global and local well-posedness for the fractional parabolic-parabolic Keller-Segel system \eqref{eq-work} has remained an open question prior to this study, even for the standard chemotactic term, $\beta=2$. 
As mentioned, the majority of existing literature focuses on the parabolic-elliptic Keller-Segel system, leaving relatively few results addressing the parabolic-parabolic case, particularly in the scenario where $\alpha=\beta$. This indicates the need for further exploration to deepen our understanding of the system's behavior in this specific setting.  

Therefore, the goal of this paper is to address this gap, for that we extend the methodology of \citet{Biler-fractional-9}, which employed the contraction mapping principle to study the parabolic-elliptic Keller-Segel system, to analyze the parabolic-parabolic case.  
We prove that, for initial conditions $\rho_0 \in L^{p} (\mathbb{R}^{d})$ and $\nabla c_0 \in L^r (\mathbb{R}^{d})$, for certain values of $p$ and $r$, solutions to system \eqref{eq-work} remain locally bounded in time. Furthermore, we establish global-in-time existence of solutions under a smallness assumption on the initial data.  

Additionally, regarding the findings in \cite{Perthame-2008} concerning local well-posedness, our results (in this case where $\alpha=\beta=2$ and $d \geq 3$) not only recover but also extend them.

\subsection{Plan of the paper} 

In \cref{sec:Preliminaries}, we introduce the type of solution to system \eqref{eq-work} we work with, introduce basic notation, and present preliminary results essential for proving the existence of solutions.

\Cref{Existence-decay-estimates} begins with the statements of our main results, along with some remarks. \Cref{subsec:Overall-idea} outlines the overall approach to proving the existence of solutions. The proofs of our main results are then provided in \cref{subsec:Local-existence,subsec:Global-existence}, focusing on local well-posedness and global well-posedness, respectively.
Furthermore, additional results on local and global well-posedness are presented at the end of \cref{subsec:Local-existence,subsec:Global-existence}, respectively.

Finally, in \Cref{Parameters-for-Lebesgue-Spaces}, we focus on the parameters defining the Lebesgue spaces used in \Cref{Theorem:local-solutions,Theorem:Global-solutions}. 
We establish the existence of these parameters and show that the constraints they satisfy are exactly the conditions necessary to apply H\"older's inequality and other key estimates introduced in \Cref{sec:Preliminaries} and \Cref{subsec:Overall-idea}, allowing us to use the Fixed Point Theorem and prove estimates for the solution.

\section{Preliminaries}
\label{sec:Preliminaries}

In this section, we define the specific type of solution under consideration, introduce notation, and establish foundational results for investigating the existence, uniqueness, and properties of the solutions to \eqref{eq-work}.

We start by fixing some notation: for $u$ and $v$, measurable functions from $\mathbb{R}^{d}$ to $\mathbb{R}^{m}$, we denote by $u*v$ the convolution of $u$ and $v$, which is the measurable function from $\mathbb{R}^{d}$ to $\mathbb{R}$ defined as 
    \begin{equation*}
        u*v(x)=\int_{\mathbb{R}^{d}} u(x-y)\cdot v(y) \mathrm{~d}y,
    \end{equation*}
where the dot stands for the inner product in $\mathbb R^m$.

The constants, which may change from line to line, will be denoted by the same letter $C$, and, unless stated otherwise, are independent of $\rho$, $c$, $x$, and $t$. In the case of dependence on a parameter, we will use the notation $C = C(\cdot)$ to emphasize the dependence of $C$ on the parameter “$\cdot$”. 

Let us rewrite equation \eqref{eq-work} in an integral form using the Duhamel's principle: 
\begin{equation}
    \label{mild-solution-rho}
    \rho(x, t)=\left(K^{\alpha}_{t}* \rho_{0}\right)(x)-\int_{0}^{t} \left(\nabla K^{\alpha}_{t-s} *\left[\rho(s) \nabla c(s)\right]\right)(x) \mathrm{~d} s
\end{equation}
and
\begin{equation}
    \label{def-eq-c} 
    \nabla c(x, t)=e^{-\frac{\gamma}{\tau} t} K^{\beta}_{\frac{t}{\tau}} * \nabla c_{0}(x)+\int_{0}^{t} \frac{1}{\tau} e^{\gamma\left(\frac{s-t}{\tau}\right)}  \nabla K^{\beta}_{\frac{t-s}{\tau}}* \rho \mathrm{~d} s,
\end{equation}
where $K^{\alpha}_{t}$ is defined as
\begin{equation}
    \label{definition-kt}
    K^{\alpha}_{t}(x)=\frac{1}{(2 \pi)^{d/2}}\int_{\mathbb{R}^{d}} \mathrm{e}^{-t|\xi|^{\alpha}}\mathrm{e}^{i x \cdot \xi}  \mathrm{d} \xi, \mbox{ for all } x\in \mathbb R^d, t>0.
\end{equation}

Let $\mathcal F^{-1}$ denote the inverse Fourier transform defined by
\begin{equation*}\mathcal{F}^{-1}f(x)=f^{\vee}(x)=\frac{1}{(2\pi)^{d/2}}\int_{\mathbb{R}^d} f(\xi) \mathrm{e}^{\mathrm{i} x \cdot \xi} \mathrm{d} \xi,
\end{equation*}
for $f$ sufficiently smooth and $x\in \mathbb R^d$. 

Notice that  \begin{equation*}
        K^{\alpha}_{t}(x)=\mathcal{F}^{-1}\left(\mathrm{e}^{-t(2\pi|\xi|)^{\alpha}}\right) 
        =t^{-\frac{d}{\alpha}} K^{\alpha}\left(\frac{x}{t^{1/\alpha}}\right),
    \end{equation*}
where the kernel $K^{\alpha}$ is  
    \begin{equation}
        \label{definition-k}
        K^{\alpha}(x)=\frac{1}{(2 \pi)^{d/2}} \int_{\mathbb{R}^{d}} \mathrm{e}^{i x \cdot \xi} \mathrm{e}^{-|\xi|^{\alpha}} \mathrm{d} \xi.
    \end{equation}

\begin{definition}[Mild solution to Keller–Segel system] \label{Definition:Mild-solution-Keller–Segel} We say that $(\rho, c)$ is a mild solution of \eqref{eq-work} with initial data $(\rho_0,\nabla c_0)$ if it satisfies the integral formulation \eqref{mild-solution-rho} and \eqref{def-eq-c}. 
\end{definition}
\begin{remark} 
    We specify $\nabla c$ in \eqref{def-eq-c} rather than $c$, since $c$ itself does not explicitly appear in the first equation of the parabolic-parabolic system \eqref{eq-work}. 
    As a result, in the existence theorem presented in \Cref{Existence-decay-estimates}, both the initial conditions and the corresponding results are formulated in terms of $\nabla c$. \label{specify-nabla-c}
\end{remark}

Next, we present important properties and estimates regarding the kernel functions $K^{\alpha}_{t}$ and $K^{\alpha}$. These are fundamental results for establishing the existence and properties of solutions to system \eqref{eq-work}. 

\begin{lemma}[\citet{Miao-2008-Well-posedness}]\label{Properties-k} 
    Consider the kernel functions $K^{\alpha}_t $ and $K^{\alpha}$. Then, 
    \begin{enumerate}[topsep=2pt, itemsep=5pt, left=14pt, labelsep=2pt,label=(\alph*)]
        \item $K^{\alpha} \in  C_0^{\infty}(\mathbb{R}^d)$; \label{Properties-k-1}
        \item $K^{\alpha}$, $ \nabla K^{\alpha} \in L^p(\mathbb{R}^d)$, for any $1 \leq p \leq \infty$; \label{norm-k-Lp}  
        \item \hypertarget{Properties-norm-k}{$K^{\alpha}_t$,} $ \nabla  K^{\alpha}_t \in L^p(\mathbb{R}^d)$, for $0<t<\infty$ and any $1 \leq p \leq \infty$;  for $1 \leq q \leq p \leq \infty$, $t>0$, and $\varphi \in L^{q}(\mathbb{R}^{d})$, we have \label{norm-k}
    \end{enumerate}
    \begin{gather}
        \left\|K^{\alpha}_t *\varphi\right\|_{L^{p}} \leq C t^{-\frac{d}{\alpha}\left(\frac{1}{q}-\frac{1}{p}\right)}\|\varphi\|_{L^{q}}, \label{estimativa-1} \\
        \left\|\nabla K^{\alpha}_t *\varphi\right\|_{L^{p}} \leq C t^{-\frac{d}{\alpha}\left(\frac{1}{q}-\frac{1}{p}\right)-\frac{1}{\alpha}}\|\varphi\|_{L^{q}}. \label{estimativa-2}
    \end{gather}
\end{lemma}
\begin{remark} \label{norm-k-sobolev-space}
Note that for $\varphi \in W^{k,q}\left(\mathbb{R}^{d}\right)$, where $k>0$ is an integer, the estimates
\begin{gather*}
    \left\|K^{\alpha}_t *\varphi\right\|_{W^{k,p}} \leq C t^{-\frac{d}{\alpha}\left(\frac{1}{q}-\frac{1}{p}\right)}\|\varphi\|_{W^{k,q}}, \\
    \left\|\nabla K^{\alpha}_t *\varphi\right\|_{W^{k,p}} \leq C t^{-\frac{d}{\alpha}\left(\frac{1}{q}-\frac{1}{p}\right)-\frac{1}{\alpha}}\|\varphi\|_{W^{k,q}},
\end{gather*}
with $t>0$ and $1 \leq q \leq p \leq \infty$, follow 
from \Cref{Properties-k}\ref{norm-k} by replacing $\varphi$ with $D^\nu \varphi$, where $\nu$ is a multi-index such that $|\nu| \leq k$.
\end{remark}
\begin{lemma} \label{integral-k2}
Let $u \in L^1(\mathbb{R}^d)$, and consider the kernel functions $K^{\alpha}_t $ and $K^{\alpha}$ defined by \eqref{definition-kt} and \eqref{definition-k}, respectively. Then, it follows that 
\begin{equation*}
    \int_{\mathbb{R}^d} K^{\alpha}_{t}* u(x) \mathrm{~d} x=\int_{\mathbb{R}^d} u(x) \mathrm{~d} x \hspace{1cm} \text{ and } \hspace{1cm} \int_{\mathbb{R}^d} \nabla K^{\alpha}_{t} * u(x) \mathrm{~d} x=0. 
\end{equation*}
\end{lemma} 
\begin{proof}
By the Fourier Inversion Theorem, if a function $f$ and its Fourier transform $\widehat{f}$ both belong to $L^1(\mathbb{R}^d)$, then $f$ coincides almost everywhere with a continuous function $f_0$, and 
$(\widehat{f})^{\vee}=\left(f^{\vee}\right)^{\wedge}=f_0$ holds. 
Thus, since $K^{\alpha}$ and $\widehat{K^{\alpha}}$ lie in $L^1(\mathbb{R}^d)$, it follows that $\int_{\mathbb{R}^d} K^{\alpha}(x) \mathrm{~d} x =\widehat{K^{\alpha}}(0)=1$. 
Next, considering $ T(x) = t^{-1} x \; (t > 0)$, the  Fourier transform satisfies $(f \circ T) \widehat{\; \;}(\xi)=t^d \widehat{f}(t \xi)$.
Applying this to $K^{\alpha}_{t}(x) = t^{-\frac{d}{\alpha}} K^{\alpha}\left(\frac{x}{t^{1/\alpha}}\right)$, we deduce that  
\begin{equation}
    \label{integral-k(a)}
    \int_{\mathbb{R}^d} K^{\alpha}_{t}(x) \mathrm{~d} x  =\widehat{K^{\alpha}_{t}}(0)=\widehat{K^{\alpha}}(0)=1.
\end{equation}
Additionally, using the property $\left(\nabla f\right) \widehat{\;\;}(\xi) = i \xi \widehat{f}(\xi)$ for $f \in C^1(\mathbb{R}^d)$ with $\nabla f \in L^1(\mathbb{R}^d)$, we obtain  
\begin{equation}
    \label{integral-k(b)}
    \int_{\mathbb{R}^d} \nabla K^{\alpha}(x)   \mathrm{~d} x=0.
\end{equation}

Now, let $F(x,y)=K^{\alpha}_{t}(x-y)u(y)$. From \Cref{Properties-k}\ref{norm-k}, we have for a.e. $y \in \mathbb{R}^d$  that $\displaystyle \int_{\mathbb{R}^d}|F(x, y)| \mathrm{d} x=\|K^{\alpha}_{t}\|_{L^1}|u(y)|<\infty$ and, further, $\displaystyle \int_{\mathbb{R}^d} \int_{\mathbb{R}^d}|F(x, y)| \mathrm{d} x \mathrm{d} y=\|K^{\alpha}_{t}\|_{L^1}\|u\|_{L^1}<\infty,$ 
since $u \in L^1(\mathbb{R}^d)$. 
Then, from Tonelli's theorem, we conclude that $F \in L^1\left(\mathbb{R}^d \times \mathbb{R}^d\right)$. 
Consequently, applying Fubini's theorem and using \eqref{integral-k(a)}, we obtain
\begin{equation*}        \int_{\mathbb{R}^d}\int_{\mathbb{R}^d} K^{\alpha}_{t}(x-y) u(x) \mathrm{d}x \mathrm{d} y  = \int_{\mathbb{R}^d} u(x) \left(\int_{\mathbb{R}^d} K^{\alpha}_{t}(x-y)   \mathrm{d}x\right) \mathrm{d} y
= \int_{\mathbb{R}^d} u(y) \mathrm{d} y. 
\end{equation*}
Similarly, using \eqref{integral-k(b)}, we prove that
\begin{align*}
\int_{\mathbb{R}^d} \nabla K^{\alpha}_{t} * v(x) \mathrm{d} x = \int_{\mathbb{R}^d} v(y) \left(\int_{\mathbb{R}^d} \nabla K^{\alpha}_{t}(x-y)   \mathrm{d}x\right) \mathrm{d} y=0.     
\end{align*}
\end{proof}

\section{Main results}
\label{Existence-decay-estimates}
In this section, we state and prove the local and global well-posedness of the solution to the parabolic-parabolic system \eqref{eq-work} and present some properties of the solution. In what follows, we state our main result as \cref{Theorem:local-solutions,Theorem:Global-solutions}. The proofs are deferred to \cref{subsec:Local-existence,subsec:Global-existence}, respectively.

\begin{hyp} Consider the following assumptions on the parameters of the system \eqref{eq-work}, $d, \alpha$ and $\beta$, and on the parameters of the space of solutions, $p$ and $r$:
\begin{enumerate}[topsep=2pt, itemsep=2pt, left=0pt, labelsep=5pt,label=(H\arabic*)]
    \item $d\geq 2$, $\alpha \in (1,2]$ and $\beta \in (1,d]$; \label{(LGA1)} 
    \item parameters $p$ and $r$ satisfy   \label{(LA2)}
    \begin{enumerate}
        \item [\hypertarget{(LA2)(a)}{(a)}] $\max \left\{ \frac{2d}{d+\beta-1},  \frac{d}{\alpha+\beta-2}\right\}< p \leq \frac{d}{\beta -1}$   and $\max  \left\{p,  \frac{p}{p-1},  \frac{d}{\alpha-1}\right\}< r  <\frac{pd}{d-p(\beta-1)}$, 
         \end{enumerate}
         or
         \begin{enumerate}
        \item [\hypertarget{(LA2)(b)}{(b)}]$p > \frac{d}{\beta -1}$ and $ r>\max  \left\{p, \; \frac{p}{p-1}, \; \frac{d}{\alpha-1}\right\}$,
    \end{enumerate}
        and the equality $r=\max  \left\{p, \frac{p}{p-1}\right\}$ is admissible in both cases if $\max  \left\{p, \frac{p}{p-1}\right\}>\frac{d}{\alpha-1}$; 
    \item  if $2\beta \left(\alpha-1\right)\geq\alpha$, the parameters $p$ and $r$ satisfy $\displaystyle p<\frac{d\alpha}{2\beta \left(\alpha-1\right)-\alpha}$ and \ref{(LA2)}. \label{(GA5)}
\end{enumerate}
\end{hyp}

\begin{theorem}[Local well-posedness] \label{Theorem:local-solutions}
Assume that \ref{(LGA1)} and \ref{(LA2)} are in force, and consider the Banach space $\mathbf{X}$ defined as 
\begin{equation*}
    \mathbf{X}=\{u \in C\hspace{-0.1cm}\left([0, T], L^{p}(\mathbb{R}^{d})\right): \|u\|_{\mathbf{X}} \equiv \sup_{t \in[0, T]}\|u(t)\|_{L^{p}} < \infty \}.
\end{equation*}

\noindent
\hypertarget{thm:local(a)}{\textbf{Case (a):}}
Then, for every initial condition $\rho_0 \in L^1(\mathbb{R}^{d}) \cap L^{p}(\mathbb{R}^{d})$ and $\nabla c_0 \in L^{r}(\mathbb{R}^{d})$, there exist $T=T\left(\|\rho_{0}\|_{L^{p}}, \; \left\|\nabla c_{0}\right\|_{L^{r}}\right)$ and a unique local mild solution $(\rho, \nabla c)$ to system \eqref{eq-work} in $[0,T]$, such that $\rho \in C\hspace{-0.1cm}\left([0, T], L^{p}(\mathbb{R}^{d})\right)$, $\nabla c \in C\hspace{-0.1cm}\left([0, T], L^{r}(\mathbb{R}^{d})\right)$, and
\begin{equation}
    \label{local-estimate-c0}
    \sup_{t \in [0, T]} t^{1-\frac{1}{\alpha}\left(\frac{d}{r}+1 \right)}\left\|\nabla c(\cdot, t)\right\|_{L^{r}} <C\left(T,\|\rho_{0}\|_{L^{p}}, \left\|\nabla c_{0}\right\|_{L^{r}} \right).
\end{equation} 

\noindent Moreover, 
\begin{enumerate}[label=(\roman*)]
    \item  $\rho \in  L^{\infty}\left((0,T);L^1 (\mathbb{R}^{d}) \cap L^{p}(\mathbb{R}^{d})\right)$, with
    \begin{equation}
        \label{local-estimate-rho}
        \|\rho(\cdot, t)\|_{L^{p}} \leq C\|\rho_{0}\|_{L^{p}}, 
    \end{equation}
    and the total mass is conserved;\label{thm:local(i)}
    \item  if $ c_0 \in W^{1,r}(\mathbb{R}^{d})$, then $c \in C\hspace{-0.1cm}\left(\left[0, T\right], W^{1,r}(\mathbb{R}^{d})\right)$, and 
    \begin{equation}
        \label{local-estimate-c-c}
        \sup_{t \in [0, T]} t^{1-\frac{1}{\alpha}\left(\frac{d}{r}+1 \right)}\left\| c(\cdot, t)\right\|_{L^{r}} <C\left(T,\; \|\rho_{0}\|_{L^{p}}, \;\left\| c_{0}\right\|_{L^{r}} \right); 
    \end{equation} \label{thm:local(iii)}
    \item  if $\displaystyle \int_{\mathbb{R}^{d}} c_{0}(x) \mathrm{d}x < \infty$, then 
    \begin{equation}
        \label{c_decay1}
        \int_{\mathbb{R}^{d}} c\left(x,t\right) \mathrm{d}x =\left(\int_{\mathbb{R}^{d}}\rho_0(x) \mathrm{d}x\right) \left(\frac{1-e^{-\frac{\gamma}{\tau} t}}{\gamma}\right)+\left(\int_{\mathbb{R}^{d}} c_0 (x) \mathrm{d}x\right)  e^{-\frac{\gamma}{\tau} t}
    \end{equation}
    and, consequently, 
    \begin{equation}
        \label{c_decay2}
        \int_{\mathbb{R}^{d}} c\left(x,t\right) \mathrm{d}x \xrightarrow[\; t\rightarrow \infty \;]{}\frac{\int_{\mathbb{R}^{d}}\rho_0(x) \mathrm{d}x}{\gamma},
    \end{equation}
    which includes $\gamma=0$ by taking $\gamma \rightarrow 0$.  \label{thm:local(iv)}
\end{enumerate} \vspace{0.25cm}

\noindent Assume $\beta \in (1,2]$. We also obtain \vspace{0.1cm}
\begin{enumerate}[label=(\roman*),resume]
    \item $\rho \in L^1\left((0, T) ; L^q(\mathbb{R}^{d})\right)$ for all $q \geq 1$,  and $\rho \in L_{\mathrm{loc}}^{\infty}\left((0, T] ; L^q(\mathbb{R}^{d})\right)$ for all $q>p$. \label{thm:local(v)}
    \item $\nabla c \in C\hspace{-0.1cm}\left([0, T], L^{q}(\mathbb{R}^{d})\right)$ for $q>\frac{d}{\alpha-1}$,  and 
    \begin{equation}
        \label{local-estimate-c0-3}
        \sup_{t \in [0, T]} t^{1-\frac{1}{\alpha}\left(\frac{d}{q}+1 \right)}\left\|\nabla c(\cdot, t)\right\|_{L^{q}} <C\left(T,\|\rho_{0}\|_{L^{p}}, \left\|\nabla c_{0}\right\|_{L^{r}} \right).
    \end{equation} \label{thm:local(vi)}
\end{enumerate} 

\noindent\hypertarget{thm:local(b)}{\textbf{Case (b):}} Let $\alpha \leq \beta$ and the initial condition be such that $\rho_0 \in L^1(\mathbb{R}^{d}) \cap L^{p}(\mathbb{R}^{d})$ and $\nabla c_0 \in L^{\wp}(\mathbb{R}^{d})$, where 
\begin{equation}
    \label{codition-r-0-Theorem}
    \wp \in \left[\frac{d}{\alpha-1}, r \right). 
\end{equation}
\begin{enumerate}[left=40pt, itemsep=7pt, topsep=2pt]
    \item [\hypertarget{thm:local(b)-1}{\textbf{(b.1)}}] If $\alpha=\beta$ and $\wp=\frac{d}{\alpha-1}$, there exist $\epsilon>0$ and $T=T\left(\|\rho_{0}\|_{L^{p}}, \left\|\nabla c_{0}\right\|_{L^{\wp}},\epsilon \right)$ such that, if \, $\left\|\nabla c_{0}\right\|_{L^{\wp}}<\epsilon$, there exists a unique local mild solution $(\rho, \nabla c)$ to system \eqref{eq-work},  with $\rho \in C\hspace{-0.1cm}\left([0, T], L^{p}(\mathbb{R}^{d})\right)$ and $\nabla c \in C\hspace{-0.1cm}\left((0, T], L^{r}(\mathbb{R}^{d})\right)$. Moreover, 
\begin{equation}
    \label{local-estimate-c0-2}
    \sup_{t \in (0, T]} t^{1-\frac{1}{\alpha}\left(\frac{d}{r}+1 \right)}\left\|\nabla c(\cdot, t)\right\|_{L^{r}} <C\left(T,\; \|\rho_{0}\|_{L^{p}}, \;\left\|\nabla c_{0}\right\|_{L^{\wp}} \right),
\end{equation} 
and \ref{thm:local(i)}, \ref{thm:local(iii)}, \ref{thm:local(iv)}, and \ref{thm:local(v)} are still satisfied. Additionally, \ref{thm:local(vi)} is also satisfied by replacing $\left\|\nabla c_{0}\right\|_{L^{r}}$ with $\left\|\nabla c_{0}\right\|_{L^{\wp}}$ in \eqref{local-estimate-c0-3}.
    \item [\hypertarget{thm:local(b)-2}{\textbf{(b.2)}}] Otherwise, the results in \hyperlink{thm:local(b)-1}{\textbf{(b.1)}} remain valid without the need for the smallness condition on $\left\|\nabla c_{0}\right\|_{L^{\wp}}$, and $T=T\left(\|\rho_{0}\|_{L^{p}}, \left\|\nabla c_{0}\right\|_{L^{\wp}} \right)$. 
\end{enumerate}
\end{theorem} 
\begin{remark} \label{Remark-For-Lucilla-Corrias-Benoıt-Perthame}
    For $d\geq 3$ and $\alpha=\beta=2$, \Cref{Theorem:local-solutions} recovers and extends the result established by \citet{Perthame-2008}. Specifically, in Theorem 2.1, \citet{Perthame-2008} established
    the existence of a local solution for the classical parabolic-parabolic Keller-Segel model with initial data $\rho_{0}\in L^p(\mathbb R^d)$ and $\nabla c_{0}\in L^{\wp}(\mathbb R^d)$ such that $p>d/2$, and $\wp=d$. In contrast, in \Cref{Theorem:local-solutions} $\wp$ can vary within the range $[d,r]$, where $r$ satisfies $d<r<\frac{pq}{d-p}$ for $p \in \left(\frac{d}{2}, d\right)$, or $r>p$ for $p>d$.
\end{remark}
\begin{remark} \label{Remark-new-theorem}
For \hyperlink{thm:local(b)}{\textbf{Case (b)}}, as shown in \Cref{lemma-condition-initial-local}, the conditions on $\alpha$, $\beta$ and $\wp$ in \hyperlink{thm:local(b)-1}{\textbf{(b.1)}} imply that
\begin{equation}
    \label{codition-r-0-Theorem2}
    \frac{1}{\alpha}\left(\frac{d}{r}+1\right)+\frac{d}{\beta}\left(\frac{1}{\wp }-\frac{1}{r}\right) = 1, 
\end{equation}
and, in \hyperlink{thm:local(b)-2}{\textbf{(b.2)}}, that 
\begin{equation}
    \label{codition-r-0-Theorem2-2}
    \frac{1}{\alpha}\left(\frac{d}{r}+1\right)+\frac{d}{\beta}\left(\frac{1}{\wp }-\frac{1}{r}\right)<1.    
\end{equation}
As we point out in the proof of \Cref{Theorem:local-solutions}, this difference is what justifies the fact that in \hyperlink{thm:local(b)-1}{\textbf{(b.1)}}, one more condition is necessary for the existence of solution. 
\end{remark}
\begin{theorem}[Global well-posedness] \label{Theorem:Global-solutions}
    Assume that \ref{(LGA1)}, \ref{(LA2)} and \ref{(GA5)} are in force, and consider the Banach space $\mathbf{X}$ defined as 
    \begin{equation*}
        \mathbf{X}=\{u \in C\hspace{-0.1cm}\left([0, \infty), L^{p}(\mathbb{R}^{d})\right): \|u\|_{\mathbf{X}} \equiv \underset{t>0}{sup} \; t^{\sigma} \left\|u (t) \right\|_{L^{p}} < \infty \},
    \end{equation*}
    where 
    \begin{equation}
        \label{definition-sigma}
        \sigma=2-\frac{1}{\alpha}\left(\frac{d}{r}+1 \right)-\frac{d}{\beta}\left(\frac{1}{p}-\frac{1}{r}\right)-\frac{1}{\beta}.
    \end{equation}
    Fix  \begin{equation}
        \label{definition-p1-p2}
        p_1=\frac{pd}{\alpha \sigma p+d} \qquad \text{and} \qquad p_{2}=\frac{dr\alpha}{\beta \left(r(\alpha-1)-d\right)+d\alpha}.
    \end{equation}
   There  exists $\epsilon>0$ such that, if 
    \begin{equation}
        \label{condition-global}
        \|\rho_{0}\|_{L^{p_1}} +\left\|\nabla c_{0}\right\|_{L^{p_2}}< \epsilon,
    \end{equation}
    then there exists a unique global mild solution $(\rho, \nabla c)$ to system \eqref{eq-work} such that $\rho \in C\hspace{-0.1cm}\left((0, \infty), L^{p}(\mathbb{R}^{d})\right)\cap \mathbf{X}$ and  $\nabla c \in C\hspace{-0.1cm}\left((0, \infty), L^{r}(\mathbb{R}^{d})\right)$, with initial condition $\rho_0 \in L^{p_1}(\mathbb{R}^{d})$ and $\nabla c_0 \in L^{p_2}(\mathbb{R}^{d})$. 
    Moreover,    
    \begin{equation}
        \label{global-estimate-rho}
        \sup_{t \geq 0} t^{\sigma}\|\rho(\cdot, t)\|_{L^{p}} \leq C\|\rho_{0}\|_{L^{p_1}}, 
    \end{equation}
    and    
    \begin{equation}
        \label{global-estimate-c}
        \sup_{t \geq 0} t^{1-\frac{1}{\alpha}\left(\frac{d}{r}+1 \right)}\left\|\nabla c(\cdot, t)\right\|_{L^{r}} <C \left\|\nabla c_{0}\right\|_{L^{p_2}}+ C  \left\|\rho_0 \right\|_{L^{p_1}}. 
    \end{equation}

    Furthermore,  if, in addition to \ref{(GA5)} (or equivalently, in this case, \ref{(LA2)}) $p$ satisfies
    \begin{equation}
        \label{condition-2-global-p-1}
        p \leq \frac{2d}{(\alpha-1)+2(\beta-1)}
    \end{equation}
    or $p$ and $r$ satisfy 
    \begin{equation}
        \label{condition-2-global-p-2}
        \frac{2d}{(\alpha-1)+2(\beta-1)} < p < \frac{\alpha d }{\max\left\{2\beta \left(\alpha-1\right)-\alpha,  \; \alpha(\alpha-2)+\beta \right\}}  
    \end{equation}
    and 
    \begin{equation}
        \label{condition-2-global-r-1}
         r \leq \frac{(2\beta-\alpha)pd}{[\beta(\alpha-1)+\alpha(\beta-1)]p-\alpha d}, 
    \end{equation}
    then $\rho \in L^{\infty}\left((0, \infty), L^{p_1}(\mathbb{R}^{d})\right)$.
\end{theorem}

\subsection{Overall idea of the proofs of the main theorems}
\label{subsec:Overall-idea}

The proofs of \cref{Theorem:local-solutions} and \cref{Theorem:Global-solutions} rely on reformulating system \eqref{eq-work} using its integral representation \eqref{mild-solution-rho} and \eqref{def-eq-c} to transform the problem of finding solutions to \eqref{eq-work} into a fixed point problem in a suitable Banach space $\mathbf{X}$. Then, by applying the following abstract result, which is based on \citep[Lemma 20]{Wavelets-paraproducts-Navier-Stokes}, we obtain a unique solution to the problem. 

\begin{lemma} \label{lemma-contraction}
Let $(\mathbf{X}, \| \cdot \|_{\mathbf{X}})$ be a Banach space. Assume that $\mathcal{A}:\mathbf{X} \times \mathbf{X} \rightarrow \mathbf{X}$ is a bounded bilinear form satisfying 
 \begin{equation}
    \label{definition-C-A}
    \|\mathcal{A}(u,v)\|_{\mathbf{X}} \leq C_{\mathcal{A}}\|u\|_{\mathbf{X}}\|v\|_{\mathbf{X}} \quad \text{for all } \; u, \; v \in \mathbf{X}, \text{ and a constant } \; C_\mathcal{A}>0,
\end{equation}
and $\mathcal{L}:\mathbf{X} \rightarrow \mathbf{X}$ is a bounded linear form satisfying
\begin{equation}
    \label{definition-C-L}
    \|\mathcal{L}(u)\|_{\mathbf{X}} \leq C_\mathcal{L}\|u\|_{\mathbf{X}} \quad \text{for all } \; u \in \mathbf{X}, \text{ and a constant } \;  C_\mathcal{L}>0.
\end{equation}
Then, if $\; 0<\delta<\frac{1-2C_\mathcal{L}}{4C_\mathcal{A}}$ and $\; u_1 \in \mathbf{X}$ is such that $\|u_1\|_{\mathbf{X}}<\delta$, the equation 
\[u = u_1+\mathcal{A}(u, u)+\mathcal{L}(u)\]
has a solution in $\mathbf{X}$ such that $\|u\|_{\mathbf{X}}\leq 2\delta$. This solution is the only one in the ball $\bar{B}(0,2\delta)$.  
\end{lemma}
\begin{proof} 
Starting from $u_1$, we define recursively the sequence \[u_{n+1}=u_{1}+\mathcal{A}\left(u_{n}, u_{n}\right)+\mathcal{L}\left(u_{n}\right), \mbox{ for all } n\in\mathbb N.\]

First, let us prove that $\left\|u_{n}\right\|_{\mathbf{X}} \leq 2 \delta$, for all $n\in \mathbb N$. By hypothesis, this is valid for $n=1$.
By induction, we assume that $\left\|u_{n}\right\|_{\mathbf{X}} \leq 2 \delta$. Then 
\begin{equation*}
    \left\|u_{n+1}\right\|_{\mathbf{X}} \leq\left\|u_{1}\right\|_{\mathbf{X}}+C_{\mathcal{A}}\left\|u_{n}\right\|_{\mathbf{X}}^{2} +C_{\mathcal{L}}\left\|u_{n}\right\|_{\mathbf{X}} 
    \leq \delta+4 C_{\mathcal{A}} \delta^{2}+2C_{\mathcal{L}} \delta 
    =\delta+\delta\left(4 C_{\mathcal{A}} \delta+2C_{\mathcal{L}} \right) 
    \leq 2 \delta,
\end{equation*}
since $4\delta C_\mathcal{A}  +2C_\mathcal{L}<1$. 

Next, observe that
\begin{equation*}
    \begin{split}
        \|u_{n+1}-u_{n}\|_{\mathbf{X}} &=\left\| \mathcal{A}\left(u_{n}, u_{n}\right)-\mathcal{A}\left(u_{n-1}, u_{n-1}\right)+\mathcal{L}\left(u_{n}\right)-\mathcal{L}\left(u_{n-1}\right)\right\|_{\mathbf{X}} \\
         &=\left\| \mathcal{A}\left(u_{n}-u_{n-1}, u_{n}\right)+\mathcal{A}\left(u_{n-1}, u_{n}-u_{n-1}\right)+\mathcal{L}\left(u_{n}\right)-\mathcal{L}\left(u_{n-1}\right)\right\|_{\mathbf{X}} \\ 
        &\leq \left\| \mathcal{A}\left(u_{n}-u_{n-1}, u_{n}\right)\right\|_{\mathbf{X}}+\left\|\mathcal{A}\left(u_{n-1}, u_{n}-u_{n-1}\right)\right\|_{\mathbf{X}}+\left\|\mathcal{L}\left(u_{n}\right)-\mathcal{L}\left(u_{n-1}\right)\right\|_{\mathbf{X}} \\
        &\leq C_\mathcal{A}\|u_{n} \|_{\mathbf{X}} \|u_{n}-u_{n-1} \|_{\mathbf{X}} + C_\mathcal{A}\|u_{n-1} \|_{\mathbf{X}} \|u_{n}-u_{n-1} \|_{\mathbf{X}} +C_\mathcal{L} \|u_{n}-u_{n-1} \|_{\mathbf{X}} \\
        &\leq \left(4\delta C_\mathcal{A}  +C_\mathcal{L}\right) \|u_{n}-u_{n-1} \|_{\mathbf{X}} \\
        &= C \|u_{n}-u_{n-1} \|_{\mathbf{X}}, 
    \end{split}
\end{equation*}
where $C=4\delta C_\mathcal{A}  +C_\mathcal{L}<1$. Hence $\left\|u_{n+1}-u_{n}\right\|_{\mathbf{X}} \leq C^{n}\left\|u_{2}-u_{1}\right\|_{\mathbf{X}}$. 

Therefore, $u_{n}$ converges to a limit $u$, which is the required solution.  To prove the uniqueness of
$u$ in $\bar{B}(0,2 \delta)$, assume there exists another solution $v$ in $\bar{B}(0,2 \delta)$. By applying the same calculations as before, we obtain $\|u-v\|\leq C\|u-v\|$, which implies $u=v$.
\end{proof}

The strategies for the proofs of \cref{Theorem:local-solutions} and \cref{Theorem:Global-solutions} are as follows.

\subsubsection*{Step 1.} 
Building upon \Cref{lemma-contraction} and \Cref{Definition:Mild-solution-Keller–Segel}, we define the components required for the fixed-point formulation. 
We start by defining $u_1$, which incorporates the initial condition $\rho_0$, as 
\begin{equation}
    \label{def-u0}
    u_1(x,t)=K^{\alpha}_{t}* \rho_{0}(x).
\end{equation}
Next, we define the bilinear form $\mathcal{A}:\mathbf{X} \times \mathbf{X} \rightarrow \mathbf{X}$ as
\begin{equation}
    \label{def-B}
    \mathcal{A}(u,v)(t)=-\int_{0}^{t} \nabla K^{\alpha}_{t-s}*\hspace{-0.08cm} \left(\hspace{-0.08cm} u(s) \int_{0}^{s} \frac{1}{\tau} e^{\gamma\left(\frac{w-s}{\tau}\right)} \nabla K^{\beta}_{\frac{s-w}{\tau}} * v(w) \mathrm{d} w \hspace{-0.08cm}\right) \mathrm{d} s.
\end{equation}
The linear form  $\mathcal{L}:\mathbf{X} \rightarrow \mathbf{X}$, which captures the influence of the initial concentration gradient, is defined as 
\begin{equation}
    \label{def-L}
    \mathcal{L}(u)(t)=-\int_{0}^{t} \nabla K^{\alpha}_{t-s} *\left[u(s) e^{-\frac{\gamma}{\tau} s} K^{\beta}_{\frac{s}{\tau}}* \nabla c_{0}\right] \mathrm{d} s.
\end{equation}

\subsubsection*{Step 2.} \label{step}
We prove that $\mathcal{A}$ and $\mathcal{L}$ satisfy the hypotheses \eqref{definition-C-A} and \eqref{definition-C-L}  of \Cref{lemma-contraction}, respectively. 
For that, we use \cref{Properties-k}\ref{norm-k} along with the following integral estimate: 
\begin{lemma} \label{Lemma-Estimativa-Integral}
    Let $b+1 > 0$, $a+1>0$. Then, the following inequality holds for all $t>0$
    \begin{equation}
        \label{Estimativa-Integral}
        \int_{0}^{t}(t-s)^{a} s^{b} \mathrm{d} s \leq  Ct^{a+b+1},    
    \end{equation}
    where $C$ is a positive constant independent of $t$.
\end{lemma}
\begin{proof}
We have 
\begin{align*}
    \int_{0}^{t}(t-s)^{a} s^{b} \mathrm{d} s 
    &=t^{a+b}\int_{0}^{t} \left(1-\frac{s}{t}\right)^{a} \left(\frac{s}{t}\right)^{b} \mathrm{d} s  \\
    &=t^{a+b+1}\int_{0}^{1} \left(1-z\right)^{(1+a)-1} z^{(1+b)-1}  \mathrm{d} z  \\
    &=\frac{\Gamma(a+1)\Gamma(b+1)}{\Gamma(a+b+2)}t^{a+b+1}.
\end{align*}
\end{proof}

\subsubsection*{Step 3.} 
After obtaining estimates for $\|u_1\|_{\mathbf{X}}$, $C_\mathcal{A}$ and $C_\mathcal{L}$, we impose the following conditions
\begin{equation}
    \label{Final-hypotheses}
    0<\delta<\frac{1-2C_\mathcal{L}}{4C_\mathcal{A}} \qquad \text{ and } \qquad \|u_1\|_{\mathbf{X}}<\delta. 
\end{equation}
Under these assumptions, \Cref{lemma-contraction} ensures that equation $\rho = u_1+\mathcal{A}(\rho, \rho)+\mathcal{L}(\rho)$ admits a unique solution in $\mathbf{X}$ such that $\|\rho\|_{\mathbf{X}}\leq 2\delta$.  This concludes the proof.  

As we shall see below, in \textbf{Step 2}, $C_\mathcal{L}$ depends on $\left\|\nabla c_{0}\right\|_{L^{q*}}$ and $\|u_1\|_{\mathbf{X}}<C\|\rho_{0}\|_{L^{p*}}$, for suitable values of $p*$ and $q*$. 
Consequently, the conditions in \eqref{Final-hypotheses} translate into constraints on the initial data $\|\rho_{0}\|_{L^{p*}}$, $\left\|\nabla c_{0}\right\|_{L^{q*}}$, and/or the maximum time $T$ for which the solution is guaranteed to exist. 

\begin{remark}\label{Comments-about-the-Proof}
In \cref{Theorem:local-solutions} and \cref{Theorem:Global-solutions}, specific restrictions are imposed on the parameters of the system and on the parameters related to the space of initial data and solutions (cf. \ref{(LGA1)}, \ref{(LA2)}, \ref{(GA5)}). 
These restrictions are essential because, throughout the proofs of these theorems, certain conditions must be met to ensure that $\mathcal{A}$ and $\mathcal{L}$ are bounded and that $u_1$ satisfies the hypotheses of \Cref{lemma-contraction}, and that the mild solution resides within certain function spaces.
In \cref{Parameters-for-Lebesgue-Spaces}, we present the proof of the existence of $p$ and $r$ satisfying such restrictions and show that they imply those conditions.  
\end{remark}

\begin{remark} The main difference between the proofs of local (\cref{Theorem:local-solutions}) and global (\cref{Theorem:Global-solutions}) well-posedness is the time independence of $C_\mathcal{A}$ and $C_\mathcal{L}$ defined in \eqref{definition-C-A} and \eqref{definition-C-L}, respectively, within the context of the global solution. This arises due to the restrictions on the spaces where the initial condition lies and the subsequent solution exists.
\end{remark}

The next results show that assuming higher spatial regularity in the initial data leads to solutions exhibiting higher spatial regularity. 

\begin{theorem}[Higher Regularity of Local-in-time Solutions]
\label{Theorem:Higher-Regularity-local}
Assume we are under the same hypotheses of \Cref{Theorem:local-solutions}, with $\mathbf X$ redefined as the following Banach space
\begin{equation*}
    \mathbf{X} = \{ u \in C\left([0, T], W^{N,p}(\mathbb{R}^{d})\right) : \|u\|_{\mathbf{X}} \equiv \sup_{t \in[0, T]}\|u(t)\|_{W^{N,p}} < \infty \},
\end{equation*}
with $N$ a positive integer, and assume the initial data satisfy the additional regularity conditions:
\begin{equation*}
    \rho_0 \in W^{N, p}(\mathbb{R}^{d}) \cap L^1(\mathbb{R}^{d}), \quad \text{and} \quad \nabla c_0 \in W^{N, q_0}(\mathbb{R}^{d}),
\end{equation*}
where $q_0 = r$ as in \hyperlink{thm:local(a)}{\textbf{Case (a)}} or $q_0 = \wp$ in as \hyperlink{thm:local(b)}{\textbf{Case (b)}} of \Cref{Theorem:local-solutions}.

Then, we obtain the version of \Cref{Theorem:local-solutions} for higher regularity, \ie there exists $T>0$, depending on the initial conditions, and a unique local mild solution $(\rho, \nabla c)$ to system \eqref{eq-work} in $[0, T]$, such that
\begin{equation*}
    \rho \in C([0, T], W^{N, p}(\mathbb{R}^{d})), \quad \nabla c \in C([0, T], W^{N, r}(\mathbb{R}^{d})).
\end{equation*}
Moreover, the estimates from \hyperlink{thm:local(a)}{\textbf{Case (a)}}, where $q_0 = r$, and \hyperlink{thm:local(b)}{\textbf{Case (b)}}, where $q_0 = \wp$, hold with $L^{v}(\mathbb{R}^{d})$ replaced by $W^{N,v}(\mathbb{R}^{d})$ for all $v \in \{p, r, q, \wp\}$, and $W^{1,r}(\mathbb{R}^{d})$ replaced by $W^{N+1,r}(\mathbb{R}^{d})$ in \ref{thm:local(iii)}. 
\end{theorem}

\begin{proof}
We replace $C([0, T], L^{p}(\mathbb{R}^{d}))$ with $C([0, T], W^{N, p}(\mathbb{R}^{d}))$ in the fixed-point framework used in \Cref{Theorem:local-solutions}. By using  \Cref{norm-k-sobolev-space} and the fact that in the Sobolev spaces, assuming $\varphi \in W^{N, p}(\mathbb{R}^{d})$ and $ \psi \in W^{N, r}(\mathbb{R}^{d})$, where $1/p+1/r=1/q \leq 1$, we have $ \varphi \psi \in W^{N, q}(\mathbb{R}^{d})$ and 
\begin{equation*}
    \left\| \varphi \psi\right\|_{W^{N,q}} \leq C \left\| \varphi \right\|_{W^{N,p}} \left\|  \psi \right\|_{W^{N,r}}
\end{equation*}
we can easily adapt the proof of \Cref{Theorem:local-solutions} to this higher regularity setting. 
\end{proof}
\begin{theorem}[Higher Regularity of Global-in-time Solutions]
\label{Theorem:Higher-Regularity-global}
Under the hypotheses of \Cref{Theorem:Global-solutions}, redefine the Banach space $\mathbf{X}$ to be
\begin{equation*}
    \mathbf{X}=\{u \in C\hspace{-0.1cm}\left([0, \infty), W^{N,p}(\mathbb{R}^{d})\right): \|u\|_{\mathbf{X}} \equiv \underset{t>0}{sup} \; t^{\sigma} \left\|u (t) \right\|_{W^{N,p}} < \infty \}.
\end{equation*}
Moreover, assume the initial data satisfy the additional regularity conditions:
\begin{equation*}
    \rho_0 \in W^{N, p_1}(\mathbb{R}^{d}), \quad \text{and} \quad \nabla c_0 \in W^{N, p_2}(\mathbb{R}^{d}),
\end{equation*}

Then, the conclusion of \Cref{Theorem:Global-solutions} remains valid in this higher-regularity setting, \ie there exists a unique global-in-time mild solution $(\rho, \nabla c)$ satisfying
\begin{equation*}
    \rho \in C\big([0, \infty), W^{N, p}(\mathbb{R}^{d})\big), \quad \nabla c \in C\big([0, \infty), W^{N, r}(\mathbb{R}^{d})\big),
\end{equation*}
with the corresponding decay estimates \eqref{global-estimate-rho} and \eqref{global-estimate-c} where the norms in $L^{v}(\mathbb{R}^{d})$ are replaced by $W^{N,v}(\mathbb{R}^{d})$ for all $v \in \{p, r, q, p_1,p_2\}$.
\end{theorem}
\begin{proof}
Again, the proof follows by adapting the arguments of \Cref{Theorem:local-solutions} to this higher-regularity setting. This relies on \Cref{norm-k-sobolev-space} and the Sobolev algebra property (for $\varphi \in W^{N, p}(\mathbb{R}^{d})$ and $\psi \in W^{N, r}(\mathbb{R}^{d})$ satisfying $1/p + 1/r = 1/q \leq 1$, we have  $ \varphi \psi \in W^{N, q}(\mathbb{R}^{d})$ and  $\left\| \varphi \psi\right\|_{W^{N,q}} \leq C \left\| \varphi \right\|_{W^{N,p}} \left\|  \psi \right\|_{W^{N,r}}$). 
\end{proof}

\subsection{Local well-posedness} 
\label{subsec:Local-existence} 

\begin{proof}[Proof of \Cref{Theorem:local-solutions}]
As described in \Cref{subsec:Overall-idea}, we establish the existence and uniqueness of local solution to system \eqref{eq-work} by applying \textbf{Step 1, 2} and \textbf{3} with $\mathbf X$ being the Banach space $C\left([0, T], L^{p}(\mathbb{R}^{d})\right)$ equipped with the usual norm.  We then verify that the hypotheses of  \Cref{lemma-contraction} are satisfied. 
For that, we employ \Cref{Properties-k}\ref{norm-k}, \cref{Lemma-Estimativa-Integral,lemma-conditions-geral-global,lemma-condition-initial-local},  and for $T>0$ (to be chosen later depending on the initial data), we obtain the following estimates: 
\begin{align} 
    &\|u_1\|_{\mathbf{X}} \leq C_1\|\rho_{0}\|_{\mathbf{X}}, \label{Estimate-u0-local}\\
    &\|\mathcal{A}(u,v)\|_{\mathbf{X}} \leq C_2 T^{1-\frac{1}{\alpha}\left(\frac{d}{r}+1\right)} \left\| u\right\|_{\mathbf{X}} \|v\|_{\mathbf{X}}, \; \text{for } \; \gamma > 0, \label{Estimate-B-local} \\
    &\|\mathcal{A}(u,v)\|_{\mathbf{X}} \leq C_2 T^{2-\frac{1}{\alpha}\left(\frac{d}{r}+1\right)-\frac{d}{\beta}\left(\frac{1}{p}-\frac{1}{r}\right)-\frac{1}{\beta}} \left\| u\right\|_{\mathbf{X}} \|v\|_{\mathbf{X}}, \; \text{for } \; \gamma= 0 \label{Estimate-B-GAMMA-0} \\
    & \left\|\mathcal{L}(u) \right\|_{\mathbf{X}} \leq C_3 T^{1-\frac{1}{\alpha}\left(\frac{d}{r}+1\right)-\frac{d}{\beta}\left(\frac{1}{\wp}-\frac{1}{r}\right)}\left\|\nabla c_{0}\right\|_{L^{\wp}} \left\|u\right\|_{\mathbf{X}}, \quad \text{for } \; \gamma \geq 0, \label{Estimate-L-local} 
\end{align}
where $p$ and $r$ are as presented in \ref{(LA2)}, $\wp$ equals $r$ for \hyperlink{thm:local(a)}{\textbf{Case (a)}} and, for \hyperlink{thm:local(b)}{\textbf{Case (b)}}, $\wp$ is given by \eqref{codition-r-0-Theorem}. The constants $C_1$, $C_2$, and $C_3$ depend on specific parameters, as follows: $C_1=C_1(\alpha, p,d )$, $C_2=C_2(\alpha, \beta, p, r, d)$, $C_3=C_3(\alpha, \beta, p, r, \wp, d)$. 

Since the derivations of these estimates follow closely the steps that are going to be used below to prove \eqref{local-estimate-c0} and \cref{thm:local(i)}, we omit their proofs. 
However, we emphasize that to establish \eqref{Estimate-B-local} and \eqref{Estimate-B-GAMMA-0}, the parameters $p$ and $r$ must satisfy the conditions stated in \eqref{condition-geral}, which is guaranteed by \cref{lemma-conditions-geral-global}. Similarly, to prove \eqref{Estimate-L-local}, $r$ must fulfill the conditions \eqref{condition-geral(a)} and \eqref{condition-geral(c)}, while $\wp$ must meet the requirements specified in \eqref{condicao-L-unica}, which is guaranteed by \cref{lemma-conditions-geral-global,lemma-condition-initial-local}.  

From \eqref{Estimate-u0-local} and \eqref{Estimate-B-local}, we can set the constants $C_{\mathcal{A}}=C_2T^{\vartheta_2}$ and $C_{\mathcal{L}}=C_3\left\|\nabla c_{0}\right\|_{L^{q}} T^{\vartheta_1}$, respectively, where $q=r$ for \hyperlink{thm:local(a)}{\textbf{Case (a)}} and $q=\wp$ for \hyperlink{thm:local(b)}{\textbf{Case (b)}}, 
$C_2$ and $C_3$ are constants depending on $\alpha$, $\beta$, $q$, $r$ and $p$, $T>0$, and $\vartheta_2\geq \vartheta_1\geq 0$ are  given by 
\begin{equation*}
    \vartheta_1= 1-\frac{1}{\alpha}\left(\frac{d}{r}+1\right)-\frac{d}{\beta}\left(\frac{1}{\wp}-\frac{1}{r}\right) 
\end{equation*}
and
\begin{equation*}
    \vartheta_2 =\begin{cases}
        1-\frac{1}{\alpha}\left(\frac{d}{r}+1\right), & \; \text{if } \; \gamma > 0, \\[8pt]
        2-\frac{1}{\alpha}\left(\frac{d}{r}+1\right)-\frac{d}{\beta}\left(\frac{1}{p}-\frac{1}{r}\right)-\frac{1}{\beta}, & \; \text{if } \; \gamma= 0 
    \end{cases}
\end{equation*}
Additionally, if condition \eqref{codition-r-0-Theorem2} is satisfied, then $\vartheta_2>0$ and $\vartheta_1=0$; otherwise, $\vartheta_2\geq \vartheta_1>0$. 

Next, set $\delta=C_1\|\rho_{0}\|_{\mathbf{X}}>0$ and notice that, by \eqref{Estimate-u0-local}, $\|u_1\|_{\mathbf{X}} \leq \delta$.
Thus, to fall into the hypotheses of \Cref{lemma-contraction}, $F(T)$ defined as  
\begin{equation*}
    F(T)=\frac{1-2C_\mathcal{L}}{4C_\mathcal{A}}=\frac{1-2C_3\left\|\nabla c_{0}\right\|_{L^{q}} T^{\vartheta_1}}{4C_2T^{\vartheta_2}} 
\end{equation*}
must satisfy $\delta <F(T)$, \ie $F(T)>C_1\|\rho_{0}\|_{L^{p}}$. 
%
Note first that, since it is required that $F(T)>0$, for $\vartheta_1>0$ we have 
\begin{equation*}
    T<T_0 \equiv \left(\frac{1}{2C_3\left\|\nabla c_{0}\right\|_{L^{q}}}\right)^{1/\vartheta_1}.
\end{equation*}
Alternatively, for $\vartheta_1=0$, $F(T)>0$ implies that $1-2C_3\left\|\nabla c_{0}\right\|_{L^{q}}>0$. Then, it must exist $\epsilon>0$ such that this condition is met if $\left\|\nabla c_{0}\right\|_{L^{q}}<\epsilon$. In that case, $T_0=\infty$.  
%
In addition, since $T>0$, $F(T)$ is a continuous function on $T$ and, as
\begin{equation*}
    \displaystyle \lim_{ T \to 0} F(T) = \infty \quad \text{ and } \quad \displaystyle \lim_{ T \to T_0} F(T) = 0,
\end{equation*}
by the intermediate value theorem, for any nontrivial $\rho_0 \in L^{p}(\mathbb{R}^{d})$ 
there is $T^{*} \in (0,T_0)$ such that $F(T^{*})>C_1\|\rho_{0}\|_{L^{p}}$.  
Hence, choosing this value of $T$, we prove the local existence of solutions by applying \Cref{lemma-contraction}. Notice that $T^{*}$ depends on the values of $\|\rho_{0}\|_{L^{p}}$ and $\left\|\nabla c_{0}\right\|_{L^{q}}$. Moreover, as a conclusion of \Cref{lemma-contraction}, we obtain (\ref{local-estimate-rho}). Then, as  $\rho$ is given by \eqref{mild-solution-rho},  $\rho \in C\hspace{-0.1cm}\left([0, T], L^{p}(\mathbb{R}^{d})\right)$ (renaming $T^{*}$ by $T$). 

Throughout this proof, consider the parameter $\wp$, by abuse of notation, where $\nabla c_0 \in L^{\wp}(\mathbb{R}^{d})$, such that $\wp\leq r$, with $r$ defined as in \ref{(LA2)}. Note that for any $\alpha$ and $\beta$,  setting $\wp= r$ places us in \hyperlink{thm:local(a)}{\textbf{Case (a)}}, while \hyperlink{thm:local(b)}{\textbf{Case (b)}} corresponds to $\wp \in \left[d/(\alpha-1), r \right)$ for $\alpha \leq \beta$.

To prove the estimate of $\nabla c$, that is, \eqref{local-estimate-c0} for \hyperlink{thm:local(a)}{\textbf{Case (a)}}, and \eqref{local-estimate-c0-2} for \hyperlink{thm:local(b)}{\textbf{Case (b)}}, we consider the integral formulation \eqref{def-eq-c}. Then,
\begin{equation} 
    \label{nabla-c-1}
    \|\nabla c(\cdot, t)\|_{L^{r}} \leq  e^{-\frac{\gamma}{\tau} t}\left\|K^{\beta}_{\frac{t}{\tau}} * \nabla c_{0}\right\|_{L^{r}}+\int_{0}^{t} \frac{1}{\tau} e^{\frac{\gamma}{\tau} (s-t)}\left\|\nabla K^{\beta}_{\frac{t-s}{\tau}} * \rho(s)\right\|_{L^{r}} \mathrm{d} s, 
\end{equation}
and applying  \eqref{estimativa-1} and  \eqref{estimativa-2} to \eqref{nabla-c-1}, we get 
\begin{equation*}
    \begin{split}
        \|\nabla c(\cdot, t)\|_{L^{r}}    
        &\leq C e^{-\frac{\gamma}{\tau}t}\left(\frac{t}{\tau}\right)^{-\frac{d}{\beta}\left(\frac{1}{\wp}-\frac{1}{r}\right)}\left\|\nabla c_{0}\right\|_{L^{\wp}} +C \int_{0}^{t} \frac{1}{\tau} e^{\frac{\gamma}{\tau} (s-t)}\left(\frac{t-s}{\tau}\right)^{-\frac{d}{\beta}\left(\frac{1}{p}-\frac{1}{r}\right)-\frac{1}{\beta}} \|\rho(s)\|_{L^{p}} \mathrm{d} s \\
        &\leq C e^{-\frac{\gamma}{\tau}t}\left(\frac{t}{\tau}\right)^{-\frac{d}{\beta}\left(\frac{1}{\wp}-\frac{1}{r}\right)}\left\|\nabla c_{0}\right\|_{L^{\wp}} +C  \left(\int_{0}^{t} \frac{1}{\tau} e^{\frac{\gamma}{\tau} (s-t)}\left(\frac{t-s}{\tau}\right)^{-\frac{d}{\beta}\left(\frac{1}{p}-\frac{1}{r}\right)-\frac{1}{\beta}}  \mathrm{d} s\right) \left\|\rho \right\|_{\mathbf{X}}.
    \end{split}
\end{equation*}

Notice that hypothesis \ref{(LA2)} implies $\frac{d}{\beta}\left(\frac{1}{p}-\frac{1}{r}\right) <1$ (see \Cref{lemma-conditions-geral-global}). Hence, if $\gamma>0$, the integral in the last term above satisfies
\begin{equation*}
    \int_{0}^{t} \frac{1}{\tau} e^{\frac{\gamma}{\tau} (s-t)}\left(\frac{t-s}{\tau}\right)^{-\frac{d}{\beta}\left(\frac{1}{p}-\frac{1}{r}\right)-\frac{1}{\beta}}  \mathrm{d} s \leq C \Gamma \left(1-\frac{d}{\beta}\left(\frac{1}{p}-\frac{1}{r}\right)-\frac{1}{\beta}\right) < \infty,
\end{equation*}
and we obtain 
\[\|\nabla c(\cdot, t)\|_{L^{r}} \leq C t^{-\frac{d}{\beta}\left(\frac{1}{\wp}-\frac{1}{r}\right)}\left\|\nabla c_{0}\right\|_{L^{\wp}}+ C  \left\|\rho \right\|_{\mathbf{X}}.\]
Therefore, 
\begin{equation*}
    \begin{split}
        t^{1-\frac{1}{\alpha}\left(\frac{d}{r}+1 \right)} \|\nabla c(\cdot, t)\|_{L^{r}} &\leq C t^{1-\frac{1}{\alpha}\left(\frac{d}{r}+1 \right)-\frac{d}{\beta}\left(\frac{1}{\wp}-\frac{1}{r}\right)}\left\|\nabla c_{0}\right\|_{L^{\wp}}+ C t^{1-\frac{1}{\alpha}\left(\frac{d}{r}+1 \right)}  \left\|\rho \right\|_{\mathbf{X}}\\ 
        &\leq C T^{1-\frac{1}{\alpha}\left(\frac{d}{r}+1 \right)-\frac{d}{\beta}\left(\frac{1}{\wp}-\frac{1}{r}\right)}\left\|\nabla c_{0}\right\|_{L^{\wp}}+2 C T^{1-\frac{1}{\alpha}\left(\frac{d}{r}+1 \right)}  \|\rho_{0}\|_{L^{p}} \\
        & <C\left(T,\; \|\rho_{0}\|_{L^{p}}, \;\left\|\nabla c_{0}\right\|_{L^{\wp}} \right),
    \end{split}
\end{equation*}
since $\frac{1}{\alpha}\left(\frac{d}{r}+1 \right)+\frac{d}{\beta}\left(\frac{1}{\wp}-\frac{1}{r}\right)\leq 1$, which follows from the definition of $r$ and by setting $\wp= r$ (\hyperlink{thm:local(a)}{\textbf{Case (a)}}) or $\wp \in \left[d/(\alpha-1), r \right)$ (\hyperlink{thm:local(b)}{\textbf{Case (b)}}) (see \Cref{lemma-condition-initial-local}). 
On the other hand, if $\gamma=0$, we have
\begin{equation*}
    \begin{split}
        \int_{0}^{t} \frac{1}{\tau} \left(t-s\right)^{-\frac{d}{\beta}\left(\frac{1}{p}-\frac{1}{r}\right)-\frac{1}{\beta}}  \mathrm{d} s & \leq C t^{1-\frac{d}{\beta}\left(\frac{1}{p}-\frac{1}{r}\right)-\frac{1}{\beta}},
    \end{split}
\end{equation*}
and, consequently, 
\[\|\nabla c(\cdot, t)\|_{L^{r}} \leq C t^{-\frac{d}{\beta}\left(\frac{1}{\wp}-\frac{1}{r}\right)}\left\|\nabla c_{0}\right\|_{L^{\wp}}+ C t^{1-\frac{d}{\beta}\left(\frac{1}{p}-\frac{1}{r}\right)-\frac{1}{\beta}}  \left\|\rho \right\|_{\mathbf{X}}.\] 

Thus, similarly to above, using that $t\leq T$ and $\ \left\|\rho \right\|_{\mathbf{X}}\leq \left\|\rho_{0}\right\|_{L^{p}}$, we obtain 
\begin{equation*}
t^{1-\frac{1}{\alpha}\left(\frac{d}{r}+1 \right)} \|\nabla c(\cdot, t)\|_{L^{r}}\leq C\left(T,\|\rho_{0}\|_{L^{p}}, \left\|\nabla c_{0}\right\|_{L^{\wp}} \right)
\end{equation*}
since $\frac{1}{\alpha}\left(\frac{d}{r}+1 \right)+\frac{d}{\beta}\left(\frac{1}{\wp}-\frac{1}{r}\right)\leq 1$ (see \Cref{lemma-condition-initial-local}).
%
Thus, considering $\wp=r$ or $\wp$ as defined in \eqref{codition-r-0-Theorem}, estimates \eqref{local-estimate-c0} and \eqref{local-estimate-c0-2}, respectively, hold. 

\textbf{\textit{Proof of \ref{thm:local(i)}:}} Starting from \eqref{mild-solution-rho}, and using
\Cref{Properties-k}\ref{norm-k} and H\"older's inequality, we obtain for $t>0$ and $1 \leq \left( \frac{r}{p+r}\right) p \leq p_1<p$, 
\begin{equation}
   \label{rho-p1}
    \|\rho(t)\|_{L^{p_1}}
    \leq \|\rho_{0}\|_{L^{p_1}} +  C \int_{0}^{t}(t-s)^{-\frac{d}{\alpha}\left(\frac{1}{p}+\frac{1}{r}-\frac{1}{p_1}\right)-\frac{1}{\alpha}} \left\|\rho(s)\right\|_{L^{p}} \left\|\nabla c(s) \right\|_{L^{r}} \mathrm{d} s.
\end{equation}

Next, we can apply \Cref{Lemma-Estimativa-Integral} and use estimate \eqref{local-estimate-c0} for \hyperlink{thm:local(a)}{\textbf{Case (a)}} and \eqref{local-estimate-c0-2} for \hyperlink{thm:local(b)}{\textbf{Case (b)}}, since 
\begin{equation*}
    \frac{d}{\alpha}\left(\frac{1}{p}+\frac{1}{r}-\frac{1}{p_1}\right)+\frac{1}{\alpha} <1, \quad  \frac{1}{\alpha}\left(\frac{d}{r}+1\right)>0, \quad \text{ and } \quad \frac{d}{\alpha}\left(\frac{1}{p_1}-\frac{1}{p}\right)>0,
\end{equation*}
to obtain
\begin{equation*}
    \begin{split}
         &\int_{0}^{t} (t-s)^{-\frac{d}{\alpha}\left(\frac{1}{p}+\frac{1}{r}-\frac{1}{p_1}\right)-\frac{1}{\alpha}}   \left\|\rho(s)\right\|_{L^{p}} \left\|\nabla c(s) \right\|_{L^{r}}  \mathrm{d} s \\
        &\hspace{3cm} \leq \int_{0}^{t}(t-s)^{-\frac{d}{\alpha}\left(\frac{1}{p}+\frac{1}{r}-\frac{1}{p_1}\right)-\frac{1}{\alpha}} s^{-1+\frac{1}{\alpha}\left(\frac{d}{r}+1\right)} \left\|\rho(s)\right\|_{L^{p}} s^{1-\frac{1}{\alpha}\left(\frac{d}{r}+1\right)} \left\|\nabla c(s) \right\|_{L^{r}} \mathrm{d} s \\
        &\hspace{3cm}  \leq C \left(\int_{0}^{t}(t-s)^{-\frac{d}{\alpha}\left(\frac{1}{p}+\frac{1}{r}-\frac{1}{p_1}\right)-\frac{1}{\alpha}}  s^{-1+\frac{1}{\alpha}\left(\frac{d}{r}+1\right)}\mathrm{d} s\right)   \left\|\rho_0 \right\|_{L^{p}} \left( \left\|\rho_0 \right\|_{L^{p}}+\left\| \nabla c_0 \right\|_{L^{\wp}}  \right)    \\
         &\hspace{3cm} \leq  C T^{\frac{d}{\alpha}\left(\frac{1}{p_1}-\frac{1}{p}\right)}   \left\|\rho_0 \right\|_{L^{p}} \left( \left\|\rho_0 \right\|_{L^{p}}+\left\| \nabla c_0 \right\|_{L^{\wp}}  \right), 
    \end{split}
\end{equation*}
where $\wp=r$ for \hyperlink{thm:local(a)}{\textbf{Case (a)}}, $\wp$ is in the range defined by \eqref{codition-r-0-Theorem} for \hyperlink{thm:local(b)}{\textbf{Case (b)}}, and the constant $C$ depends on $T$. Thus, from \eqref{rho-p1}, we have  
\begin{equation}
    \label{rho-p1-2}
    \|\rho(t)\|_{L^{p_1}} <C\left(T,\|\rho_{0}\|_{L^{p}}, \|\rho_{0}\|_{L^{p_1}},  \left\|\nabla c_{0}\right\|_{L^{\wp}} \right).
\end{equation}

Note that if $1/p+1/r=1$, we can take $p_1=1$. Otherwise, we have $1/p+1/r<1$, and we can recalculate \eqref{rho-p1} using the new information given in \eqref{rho-p1-2}. 
Specifically, we can recalculate  \eqref{rho-p1} for $p_2\geq 1$ such that $\left( \frac{r}{p_1+r}\right) p_{1} \leq p_{2}<p_1$ using \eqref{rho-p1-2}, and continue iteratively for  $p_n \geq 1$ such that $\left( \frac{r}{p_{n-1}+r}\right) p_{n-1} \leq p_{n}<p_{n-1}$, using 
\begin{equation*}
    \|\rho(t)\|_{L^{p_n}} <C\left(T, \|\rho_{0}\|_{L^{p}}, \|\rho_{0}\|_{L^{p_1}}, \|\rho_{0}\|_{L^{p_2}},  \cdots ,  \|\rho_{0}\|_{L^{p_n}}, \left\|\nabla c_{0}\right\|_{L^{\wp}} \right),
\end{equation*}
until $\left( \frac{r}{p_{n}+r}\right) p_{n} \leq 1$, at which point we can choose $p_{n+1}=1$.  
Then, we conclude that $\rho \in  L^{\infty}((0,T);L^1(\mathbb{R}^d)\cap L^{p}(\mathbb{R}^d))$. 

Next, to prove the mass conservation, note that $(\rho \nabla c) (\cdot, t) \in L^{1}(\mathbb{R}^d)$ for every $t \in [0,T]$. Indeed, with the above calculation, we can choose $p_1$ so that $\rho \in  L^{p_1}(\mathbb{R}^d)$ and $1/p_1+1/r=1$, then $\left\|(\rho \nabla c)(\cdot, t) \right\|_{L^{1}} \leq C \left\|\rho (\cdot, t) \right\|_{L^{p_1}}\left\|\nabla c (\cdot, t) \right\|_{L^{r}} <\infty$. 
Then, by integrating both sides of equation \eqref{mild-solution-rho} with respect to $x$ and applying \Cref{integral-k2}, since $\rho_0 \in L^1(\mathbb{R}^d)$,  $\rho(\cdot,t) \in L^1(\mathbb{R}^d)$ and $\rho\nabla c(\cdot,t) \in L^1(\mathbb{R}^d)$ for every $t \in [0,T]$, we establish that $\displaystyle \int_{\mathbb{R}^d} \rho (x, t) \mathrm{d} x=\int_{\mathbb{R}^d} \rho_0(x) \mathrm{d} x$ for every $t \in [0,T]$. Therefore, we conclude the proof of \ref{thm:local(i)}. 

\textbf{\textit{Proof of \ref{thm:local(iii)}:}} 
Analogously to the
way we proved the behavior of $ \nabla c$ and obtained estimates \eqref{local-estimate-c0} and (\ref{local-estimate-c0-2}), we now prove the estimate of $ c$. 
From integral formulation, as $\wp\leq r$, with $r$ satisfying \ref{(LA2)},  we obtain 
\begin{equation*}
    \begin{split}
        \| c(\cdot, t)\|_{L^{r}} &\leq  e^{-\frac{\gamma}{\tau} t}\left\|K^{\beta}_{\frac{t}{\tau}} *  c_{0}\right\|_{L^{r}}+\int_{0}^{t} \frac{1}{\tau} e^{\frac{\gamma}{\tau} (s-t)}\left\| K^{\beta}_{\frac{t-s}{\tau}} * \rho(s)\right\|_{L^{r}} \mathrm{d} s \\
        &\leq C e^{-\frac{\gamma}{\tau}t}\left(\frac{t}{\tau}\right)^{-\frac{d}{\beta}\left(\frac{1}{\wp}-\frac{1}{r}\right)}\left\| c_{0}\right\|_{L^{\wp}}+C \int_{0}^{t} \frac{1}{\tau} e^{\frac{\gamma}{\tau} (s-t)}\left(\frac{t-s}{\tau}\right)^{-\frac{d}{\beta}\left(\frac{1}{p}-\frac{1}{r}\right)} \|\rho(s)\|_{L^{p}} \mathrm{d} s \\
        &\leq C e^{-\frac{\gamma}{\tau}t}\left(\frac{t}{\tau}\right)^{-\frac{d}{\beta}\left(\frac{1}{\wp}-\frac{1}{r}\right)}\left\| c_{0}\right\|_{L^{\wp}}+C  \left(\int_{0}^{t} \frac{1}{\tau} e^{\frac{\gamma}{\tau} (s-t)}\left(\frac{t-s}{\tau}\right)^{-\frac{d}{\beta}\left(\frac{1}{p}-\frac{1}{r}\right)}  \mathrm{d} s\right) \left\|\rho \right\|_{\mathbf{X}}.
    \end{split}
\end{equation*}

Then, if $\gamma>0$, the integral in the last term above satisfies
\begin{equation*}
    \int_{0}^{t} \frac{1}{\tau} e^{\frac{\gamma}{\tau} (s-t)}\left(\frac{t-s}{\tau}\right)^{-\frac{d}{\beta}\left(\frac{1}{p}-\frac{1}{r}\right)}  \mathrm{d} s \leq C \Gamma \left(1-\frac{d}{\beta}\left(\frac{1}{p}-\frac{1}{r}\right)\right),
\end{equation*}
since $ \frac{d}{\beta}\left(\frac{1}{p}-\frac{1}{r}\right)+\frac{1}{\beta} <1$, and we obtain
\[\| c(\cdot, t)\|_{L^{r}} \leq C t^{-\frac{d}{\beta}\left(\frac{1}{\wp}-\frac{1}{r}\right)}\left\| c_{0}\right\|_{L^{\wp}}+ C  \left\|\rho \right\|_{\mathbf{X}}.\] 
Hence,
\begin{equation*}
    \begin{split}
        t^{1-\frac{1}{\alpha}\left(\frac{d}{r}+1 \right)} \| c(\cdot, t)\|_{L^{r}} &\leq C t^{1-\frac{1}{\alpha}\left(\frac{d}{r}+1 \right)-\frac{d}{\beta}\left(\frac{1}{\wp}-\frac{1}{r}\right)}\left\| c_{0}\right\|_{L^{\wp}}+ C t^{1-\frac{1}{\alpha}\left(\frac{d}{r}+1 \right)}  \left\|\rho \right\|_{\mathbf{X}}\\
        &\leq C T^{1-\frac{1}{\alpha}\left(\frac{d}{r}+1 \right)-\frac{d}{\beta}\left(\frac{1}{\wp}-\frac{1}{r}\right)}\left\| c_{0}\right\|_{L^{\wp}}+C T^{1-\frac{1}{\alpha}\left(\frac{d}{r}+1 \right)}  \|\rho_{0}\|_{L^{p}} \\
        & <C\left(T,\; \|\rho_{0}\|_{L^{p}}, \;\left\| c_{0}\right\|_{L^{\wp}} \right).
    \end{split}
\end{equation*}

On the other hand, if $\gamma=0$, we have
\begin{equation*}
    \begin{split}
        \int_{0}^{t} \frac{1}{\tau} \left(t-s\right)^{-\frac{d}{\beta}\left(\frac{1}{p}-\frac{1}{r}\right)}  \mathrm{d} s & \leq C t^{1-\frac{d}{\beta}\left(\frac{1}{p}-\frac{1}{r}\right)},
    \end{split} 
\end{equation*}
since $ \frac{d}{\beta}\left(\frac{1}{p}-\frac{1}{r}\right)+\frac{1}{\beta} <1$, and, consequently,
\[\| c(\cdot, t)\|_{L^{r}} \leq C t^{-\frac{d}{\beta}\left(\frac{1}{\wp}-\frac{1}{r}\right)}\left\| c_{0}\right\|_{L^{\wp}}+ C t^{1-\frac{d}{\beta}\left(\frac{1}{p}-\frac{1}{r}\right)}  \left\|\rho \right\|_{\mathbf{X}}.\] 
Therefore, 
\begin{equation*}
    \begin{split}
        t^{1-\frac{1}{\alpha}\left(\frac{d}{r}+1 \right)} \| c(\cdot, t)\|_{L^{r}} & \leq C t^{1-\frac{1}{\alpha}\left(\frac{d}{r}+1 \right)-\frac{d}{\beta}\left(\frac{1}{\wp}-\frac{1}{r}\right)}\left\| c_{0}\right\|_{L^{\wp}}+ C t^{2-\frac{1}{\alpha}\left(\frac{d}{r}+1 \right)-\frac{d}{\beta}\left(\frac{1}{p}-\frac{1}{r}\right)}  \left\|\rho \right\|_{\mathbf{X}} \\ 
        &\leq C T^{1-\frac{1}{\alpha}\left(\frac{d}{r}+1 \right)-\frac{d}{\beta}\left(\frac{1}{\wp}-\frac{1}{r}\right)}\left\| c_{0}\right\|_{L^{\wp}}+ C T^{2-\frac{1}{\alpha}\left(\frac{d}{r}+1 \right)-\frac{d}{\beta}\left(\frac{1}{p}-\frac{1}{r}\right)}  \|\rho_{0}\|_{L^{p}} \\
        & <C\left(T,\; \|\rho_{0}\|_{L^{p}}, \;\left\| c_{0}\right\|_{L^{\wp}} \right),        
    \end{split}
\end{equation*}
where again $\wp=r$ for \hyperlink{thm:local(a)}{\textbf{Case (a)}} and $\wp$ is in the range defined by \eqref{codition-r-0-Theorem} for \hyperlink{thm:local(b)}{\textbf{Case (b)}}.

\textbf{\textit{Proof of \ref{thm:local(iv)}:}}  Note that, from equation \eqref{def-eq-c}, we can obtain the correspondent equation for $c$. 
Moreover, since from \ref{thm:local(i)} we have $\rho(\cdot,t) \in L^1(\mathbb{R}^d)$ for almost every $t \in [0,T]$, we can apply \Cref{integral-k2}.  
Then, proceeding the same way as before, we establish 
\begin{equation*}
    \begin{split}
        \int_{\mathbb{R}^d}c(x, t)\mathrm{~d} x &=e^{-\frac{\gamma}{\tau} t} \int_{\mathbb{R}^d}K^{\beta}_{\frac{t}{\tau}} * c_{0}(x)\mathrm{~d} x+\int_{\mathbb{R}^d}\int_{0}^{t} \frac{1}{\tau} e^{\gamma\left(\frac{s-t}{\tau}\right)}  K^{\beta}_{\frac{t-s}{\tau}}* \rho(x,s) \mathrm{~d} s \mathrm{~d} x \\
        &=e^{-\frac{\gamma}{\tau} t} \int_{\mathbb{R}^d} c_{0}(x)\mathrm{~d} x+\int_{0}^{t} \frac{1}{\tau} e^{\gamma\left(\frac{s-t}{\tau}\right)}  \int_{\mathbb{R}^d} \rho(x,s)  \mathrm{~d} x \mathrm{~d} s \\
        &=e^{-\frac{\gamma}{\tau} t} \int_{\mathbb{R}^d} c_{0}(x)\mathrm{~d} x+\int_{0}^{t} \frac{1}{\tau} e^{\gamma\left(\frac{s-t}{\tau}\right)}  \int_{\mathbb{R}^d} \rho_0(x)  \mathrm{~d} x \mathrm{~d} s. 
    \end{split}
\end{equation*}
Thus, we obtain \eqref{c_decay1}, and \eqref{c_decay2} follows.

\textbf{\textit{Proof of \ref{thm:local(v)}:}} Let $q_0>p$ be such that 
\begin{equation}
    \label{relation-p-q-0}
    \frac{1}{p}+\frac{1}{r}-\frac{\alpha-1}{d}<\frac{1}{q_0}<\frac{1}{p}.
\end{equation}

Starting from \eqref{mild-solution-rho}, we apply H\"older's inequality and use estimate \eqref{estimativa-2} to obtain that
\begin{equation*}
    \begin{split}
        \|\rho(t)\|_{L^{q_0}}
        &\leq \|K^{\alpha}_{t}* \rho_{0}\|_{L^{q_0}} + \left\| \int_{0}^{t}  \nabla K^{\alpha}_{t-s} *\left[\rho(s) \nabla c(s) \right] \mathrm{d} s \right\|_{L^{q_0}}  \\ 
        &\leq C t^{-\frac{d}{\alpha}\left(\frac{1}{p}-\frac{1}{q_0}\right)} \|\rho_{0}\|_{L^{p}} +  \int_{0}^{t}  \left\|\nabla K^{\alpha}_{t-s} *\left[\rho(s) \nabla c(s) \right] \right\|_{L^{q_0}} \mathrm{d} s   \\ 
         &\leq C t^{-\frac{d}{\alpha}\left(\frac{1}{p}-\frac{1}{q_0}\right)} \|\rho_{0}\|_{L^{p}} +  C \int_{0}^{t}(t-s)^{-\frac{d}{\alpha}\left(\frac{1}{p}+\frac{1}{r}-\frac{1}{q_0}\right)-\frac{1}{\alpha}} \left\|\rho(s)\right\|_{L^{p}} \left\|\nabla c(s) \right\|_{L^{r}} \mathrm{d} s   \\ 
         &\leq C t^{-\frac{d}{\alpha}\left(\frac{1}{p}-\frac{1}{q_0}\right)} \|\rho_{0}\|_{L^{p}} \\&\quad \qquad +  C\left(T, \|\rho_{0}\|_{L^{p}},\left\|\nabla c_{0}\right\|_{L^{\wp}} \right) \int_{0}^{t}(t-s)^{-\frac{d}{\alpha}\left(\frac{1}{p}+\frac{1}{r}-\frac{1}{q_0}\right)-\frac{1}{\alpha}} s^{-1+\frac{1}{\alpha}\left(\frac{d}{r}+1\right)} \mathrm{d} s. 
    \end{split}
\end{equation*}
Since, from \eqref{relation-p-q-0}, we have $\frac{d}{\alpha}\left(\frac{1}{p}+\frac{1}{r}-\frac{1}{q_0}\right)+\frac{1}{\alpha}<1$, then we obtain 
\begin{equation}
    \label{rho-1}
    \|\rho(t)\|_{L^{q_0}} \leq  C\left(T, \; \|\rho_{0}\|_{L^{p}}, \;\left\|\nabla c_{0}\right\|_{L^{\wp}} \right) t^{-\frac{d}{\alpha}\left(\frac{1}{p}-\frac{1}{q_0}\right)}.
\end{equation}

Now, let $q_1>q_0$ be such that
\begin{equation}
    \label{relation-p-q-1}
    \frac{1}{q_0}+\frac{1}{r}-\frac{\alpha-1}{d}<\frac{1}{q_1}<\frac{1}{q_0}.  
\end{equation}

Note that, from \ref{(LA2)}, we have $\frac{1}{p}-\frac{1}{r}<\frac{\beta-1}{d}$, 
which, given that $\beta \in (1,2]$, implies
\begin{equation}
    \label{relation-p-r-bonded}
    \frac{d}{\alpha}\left(\frac{1}{p}-\frac{1}{r}\right)-\frac{1}{\alpha} <\frac{d}{\alpha}\left(\frac{\beta-1}{d}\right)-\frac{1}{\alpha}=\frac{\beta-2}{\alpha} \leq 0.
\end{equation}

Then, using \eqref{rho-1} and
estimate \eqref{local-estimate-c0} for \hyperlink{thm:local(a)}{\textbf{Case (a)}} and \eqref{local-estimate-c0-2} for \hyperlink{thm:local(b)}{\textbf{Case (b)}}, we obtain
\begin{equation*}
    \begin{split}
        &\|\rho(t)\|_{L^{q_1}}
         \leq \|K^{\alpha}_{t}* \rho_{0}\|_{L^{q_1}} + \left\| \int_{0}^{t}  \nabla K^{\alpha}_{t-s} *\left[\rho(s) \nabla c(s) \right] \mathrm{d} s \right\|_{L^{q_1}}\\ 
        & \;\qquad\leq C t^{-\frac{d}{\alpha}\left(\frac{1}{p}-\frac{1}{q_1}\right)} \|\rho_{0}\|_{L^{p}} +  C \int_{0}^{t}(t-s)^{-\frac{d}{\alpha}\left(\frac{1}{q_0}+\frac{1}{r}-\frac{1}{q_1}\right)-\frac{1}{\alpha}} \left\|\rho(s)\right\|_{L^{q_0}} \left\|\nabla c(s) \right\|_{L^{r}} \mathrm{d} s \\  
        &\;\qquad\leq C t^{-\frac{d}{\alpha}\left(\frac{1}{p}-\frac{1}{q_1}\right)} \|\rho_{0}\|_{L^{p}} \\
        & \;\qquad\qquad + C\left(T, \|\rho_{0}\|_{L^{p}},\left\|\nabla c_{0}\right\|_{L^{\wp}} \right) \int_{0}^{t}(t-s)^{-\frac{d}{\alpha}\left(\frac{1}{q_0}+\frac{1}{r}-\frac{1}{q_1}\right)-\frac{1}{\alpha}} s^{-\frac{d}{\alpha}\left(\frac{1}{p}-\frac{1}{r}-\frac{1}{q_0}\right)-1+\frac{1}{\alpha}} \mathrm{d} s \\
        &\;\qquad\leq C\left(T, \|\rho_{0}\|_{L^{p}},\left\|\nabla c_{0}\right\|_{L^{\wp}} \right) t^{-\frac{d}{\alpha}\left(\frac{1}{p}-\frac{1}{q_1}\right)},
    \end{split}
\end{equation*}
since, from \eqref{relation-p-q-1} and \eqref{relation-p-r-bonded}, we have $\frac{d}{\alpha}\left(\frac{1}{q_0}+\frac{1}{r}-\frac{1}{q_1}\right)+\frac{1}{\alpha}<1$ and $\frac{d}{\alpha}\left(\frac{1}{p}-\frac{1}{r}-\frac{1}{q_0}\right)-\frac{1}{\alpha} <0$, respectively. 

By induction, let $q_{n+1}>q_n$ be such that
\begin{equation}
    \label{relation-p-q-n}
    \frac{1}{q_{n}}+\frac{1}{r}-\frac{\alpha-1}{d}<\frac{1}{q_{n+1}}<\frac{1}{q_{n}},
\end{equation}
and assume
\begin{equation}
    \label{rho-2}
    \|\rho(t)\|_{L^{q_n}}
    \leq C\left(T, \; \|\rho_{0}\|_{L^{p}}, \;\left\|\nabla c_{0}\right\|_{L^{\wp}} \right) t^{-\frac{d}{\alpha}\left(\frac{1}{p}-\frac{1}{q_n}\right)}. 
\end{equation}

Then, for $t>0$, with estimates \eqref{estimativa-2} and \eqref{rho-2}, we obtain
\begin{align*}
    &\|\rho(t)\|_{L^{q_{n+1}}} 
    \leq \|K^{\alpha}_{t}* \rho_{0}\|_{L^{q_{n+1}}} + \left\| \int_{0}^{t}  \nabla K^{\alpha}_{t-s} *\left[\rho(s) \nabla c(s) \right] \mathrm{d} s \right\|_{L^{q_{n+1}}} \\ 
    & \;\leq C t^{-\frac{d}{\alpha}\left(\frac{1}{p}-\frac{1}{q_{n+1}}\right)} \|\rho_{0}\|_{L^{p}} \\ 
    & \hspace{0.6cm} + C \int_{0}^{t}(t-s)^{-\frac{d}{\alpha}\left(\frac{1}{q_n}+\frac{1}{r}-\frac{1}{q_{n+1}}\right)-\frac{1}{\alpha}} \left\|\rho(s)\right\|_{L^{q_n}} s^{-1+\frac{1}{\alpha}\left(\frac{d}{r}+1\right)} s^{1-\frac{1}{\alpha}\left(\frac{d}{r}+1\right)} \left\|\nabla c(s) \right\|_{L^{r}} \mathrm{d} s \\
    & \;\leq C t^{-\frac{d}{\alpha}\left(\frac{1}{p}-\frac{1}{q_{n+1}}\right)} \|\rho_{0}\|_{L^{p}} \\
    & \hspace{0.6cm} + C\left(T, \; \|\rho_{0}\|_{L^{p}}, \;\left\|\nabla c_{0}\right\|_{L^{\wp}} \right) \int_{0}^{t}(t-s)^{-\frac{d}{\alpha}\left(\frac{1}{q_n}+\frac{1}{r}-\frac{1}{q_{n+1}}\right)-\frac{1}{\alpha}} s^{-\frac{d}{\alpha}\left(\frac{1}{p}-\frac{1}{r}-\frac{1}{q_n}\right)-1+\frac{1}{\alpha}}  \mathrm{d} s.
\end{align*}
Thus, we derive
\begin{equation}
    \label{proof(iv)}
    \|\rho(t)\|_{L^{q_{n+1}}} \leq C t^{-\frac{d}{\alpha}\left(\frac{1}{p}-\frac{1}{q_{n+1}}\right)} \|\rho_{0}\|_{L^{p}} + C\left(T, \|\rho_{0}\|_{L^{p}}, \left\|\nabla c_{0}\right\|_{L^{\wp}} \right) t^{-\frac{d}{\alpha}\left(\frac{1}{p}-\frac{1}{q_{n+1}}\right)},
\end{equation}
since, from \eqref{relation-p-q-n} and \eqref{relation-p-r-bonded}, we have $\frac{d}{\alpha}\left(\frac{1}{q_n}+\frac{1}{r}-\frac{1}{q_{n+1}}\right)+\frac{1}{\alpha}<1$ and $\frac{d}{\alpha}\left(\frac{1}{p}-\frac{1}{r}-\frac{1}{q_n}\right)$ $-\frac{1}{\alpha}<0$, respectively. 

Notice that, for any $q>p$, there exists $n$ such that $p<q_{n}<q<q_{n+1} \leq \infty$. Then, using H\"older's inequality, we have that $L^{q_{n}}(\mathbb{R}^{d}) \cap L^{q_{n+1}}(\mathbb{R}^{d}) \subset L^q(\mathbb{R}^{d})$ and, from \eqref{proof(iv)}, we infer that 
\begin{equation}
    \label{proof(iv)-2}
    \|\rho(t)\|_{L^{q}}
    \leq C\left(T, \; \|\rho_{0}\|_{L^{p}}, \;\left\|\nabla c_{0}\right\|_{L^{\wp}} \right) t^{-\frac{d}{\alpha}\left(\frac{1}{p}-\frac{1}{q}\right)}.
\end{equation}
Thus, we conclude that $\rho \in L_{\mathrm{loc}}^{\infty}\left((0, T] ; L^q(\mathbb{R}^{d})\right)$ for all $q>p$. 

Now, observe that $0<\frac{d}{\alpha}\left(\frac{1}{p}-\frac{1}{q_{n}}\right)<1$ for all $n \in \mathbb{N}$. The lower bound holds because $q_n>p$ for all $n \in \mathbb{N}$, while the upper bound follows from the fact that, since $p$ satisfies \ref{(LA2)} and $\beta \in (1,2]$, we have $\frac{d}{\alpha}\frac{1}{p}<1+\frac{\beta-2}{\alpha}\leq 1$. Consequently, from \eqref{proof(iv)-2}, we see that $ \|\rho(t)\|_{L^{q}}$ is bonded by a function of $t$ that is integrable over $[0,T]$. From this, we conclude that $\rho \in L^1\left((0, T) ; L^q(\mathbb{R}^{d})\right)$ for all $q > p$. 
Additionally, it is straightforward to verify that $\rho \in L^1\left((0, T) ; L^q(\mathbb{R}^{d})\right)$ for all $1\leq q \leq p$.
Therefore, we obtain $\rho \in L^1\left((0, T) ; L^q(\mathbb{R}^{d})\right)$ for all $q \geq 1$.

\textbf{\textit{Proof of \ref{thm:local(vi)}:}}
We now prove the estimate of $\nabla c$, as before, by considering the integral formulation \eqref{def-eq-c} and applying estimates \eqref{estimativa-1} and  \eqref{estimativa-2} 
to obtain 
\begin{equation*}
    \begin{split}
        \|\nabla c(\cdot, t)\|_{L^{q}} &\leq  e^{-\frac{\gamma}{\tau} t}\left\|K^{\beta}_{\frac{t}{\tau}} * \nabla c_{0}\right\|_{L^{q}}+\int_{0}^{t} \frac{1}{\tau} e^{\frac{\gamma}{\tau} (s-t)}\left\|\nabla K^{\beta}_{\frac{t-s}{\tau}} * \rho(s)\right\|_{L^{q}} \mathrm{d} s \\
        &\leq C e^{-\frac{\gamma}{\tau}t}\left(\frac{t}{\tau}\right)^{-\frac{d}{\beta}\left(\frac{1}{\wp}-\frac{1}{r}\right)}\left\|\nabla c_{0}\right\|_{L^{\wp}}+C \int_{0}^{t} \frac{1}{\tau} e^{\frac{\gamma}{\tau} (s-t)}\left(\frac{t-s}{\tau}\right)^{-\frac{1}{\beta}} \|\rho(s)\|_{L^{q}} \mathrm{d} s.\\
    \end{split}
\end{equation*}

From \textbf{\textit{\ref{thm:local(v)}}}, we know that $\rho \in L^1\left((0, T) ; L^q(\mathbb{R}^{d})\right)$ and $\|\rho\|_{L^{\infty}\left([t_1,t_2]; L^q(\mathbb{R}^{d})\right)} < \infty$ for $ t_1, \; t_2 \in (0,T]$, as $\rho \in L_{\mathrm{loc}}^{\infty}\left((0, T] ; L^q(\mathbb{R}^{d})\right)$.
Then, taking into account the fact that $\beta>1$, we obtain
\begin{equation*}
    \begin{split}
        &\int_{0}^{t} \frac{1}{\tau} e^{\frac{\gamma}{\tau} (s-t)}\left(\frac{t-s}{\tau}\right)^{-\frac{1}{\beta}} \|\rho(s)\|_{L^{q}} \mathrm{d} s \\
        &\hspace{1cm} =\int_{0}^{t_1} \frac{1}{\tau} e^{\frac{\gamma}{\tau} (s-t)}\left(\frac{t-s}{\tau}\right)^{-\frac{1}{\beta}} \|\rho(s)\|_{L^{q}} \mathrm{d} s+\int_{t_1}^{t} \frac{1}{\tau} e^{\frac{\gamma}{\tau} (s-t)}\left(\frac{t-s}{\tau}\right)^{-\frac{1}{\beta}} \|\rho(s)\|_{L^{q}} \mathrm{d} s\\
        &\hspace{1cm}\leq C(T) \|\rho\|_{L^1\left((0, T) ; L^q(\mathbb{R}^{d})\right)} +C(T) \|\rho\|_{L^{\infty}\left([t_1,t_2]; L^q(\mathbb{R}^{d})\right)}. 
    \end{split}
\end{equation*}

Therefore, \eqref{local-estimate-c0-3} follows from setting $\wp=r$. For \hyperlink{thm:local(b)}{\textbf{Case (b)}}, $\wp$ can be set in the range defined by \eqref{codition-r-0-Theorem}.
\end{proof}

As in \Cref{Theorem:local-solutions}, we apply the fixed point theorem, using the same steps outlined in \Cref{subsec:Overall-idea}, to establish local well-posedness of system \eqref{eq-work} in the weighted $L^{\infty}(\mathbb{R}^d)$ space defined as 
\begin{equation*}
    L_{\alpha+d}^{\infty}(\mathbb{R}^d)=\left\{\varphi \in L^{\infty}(\mathbb{R}^d):\|\varphi\|_{L_{\alpha+d}^{\infty}} \equiv \operatorname{ess} \sup _{x \in \mathbb{R}^d}(1+|x|)^{\alpha+d}|\varphi(x)|<\infty\right\},
\end{equation*}
where $\alpha$ is the order of the fractional differential operator in \eqref{eq-work}. 
Note that, since $\alpha>0$, the weighted space $L_{\alpha+d}^{\infty}(\mathbb{R}^d)$ satisfies $L_{\alpha+d}^{\infty}(\mathbb{R}^d) \subset L^1(\mathbb{R}^d) \cap L^{\infty}(\mathbb{R}^d)$. 

To proceed, we present weighted estimates involving the kernel function $K^{\alpha}_t(x)$:
\begin{lemma}{\citep[Lemma 3.3]{Biler-fractional-9}} \label{lemma-norm-k-weighted-spaces}
    Let $\varphi \in L_{\alpha+d}^{\infty}(\mathbb{R}^d)$ and $t>0$. There exists $C>0$ independent of $\varphi$ and $t$ such that
    \begin{equation}
        \label{estimativa-1-weighted-space}
        \left\|K^{\alpha}_t *\varphi\right\|_{L_{\alpha+d}^{\infty}} \leq C(1+t)\left\|\varphi\right\|_{L_{\alpha+d}^{\infty}}, 
    \end{equation}
    \begin{equation}
        \label{estimativa-2-weighted-space}
        \left\|\nabla K^{\alpha}_t *\varphi\right\|_{L_{\alpha+d}^{\infty}} \leq C t^{-1 / \alpha}\left\|\varphi\right\|_{L_{\alpha+d}^{\infty}}+C t^{1-1 / \alpha}\left\|\varphi\right\|_{L^1}.
    \end{equation}
\end{lemma}

Additionally, consider the following result:
\begin{lemma}{\citep[Lemma 6]{inequality-2005}} \label{lemma-inequality-2005} 
    Let $\gamma, \; \vartheta$ be multi-indices, $|\vartheta|<|\gamma|+\alpha \max(j,1)$, $j = 0, \; 1, \; 2, \;  \ldots ,$ $1 \leq p \leq \infty$, and $\alpha \in (0,2]$. Then,
    \begin{equation}
        \label{estimativa-k-lemma-inequality-2005}
        \|x^{\vartheta} D_{t}^{j} D^{\gamma}K^{\alpha}_t \|_{L^{p}}= Ct^{\frac{|\vartheta|-|\gamma|}{\alpha}-j-\frac{d(p-1)}{\alpha p}} 
    \end{equation}
    for $C$ a constant depending only on $\alpha, \; \gamma, \; \vartheta, \; j, \; p,$ and the space dimension $d$.
\end{lemma}

With these estimates, we can construct a mild solution to system \eqref{eq-work} in the weighted space $L_{\alpha+d}^{\infty}(\mathbb{R}^d)$ and establish estimates for the solution.
\begin{proposition} \label{Proposition-local}
    Let $d\geq 2$, $\alpha \in (1,2]$ and $\beta \in (1,d]$.
    Then, for every initial condition $\rho_0 \in L_{\alpha+d}^{\infty}(\mathbb{R}^d)$ and $\nabla c_0 \in L^{\infty}(\mathbb{R}^d)$, there exist $T=T(\|\rho_{0}\|_{L_{\alpha+d}^{\infty}}, \; \left\|\nabla c_{0}\right\|_{L^{\infty}})$ and a unique local mild solution $(\rho, \nabla c)$ to system \eqref{eq-work} in $[0,T]$ 
    such that $\rho \in C\hspace{-0.1cm}\left([0, T], L_{\alpha+d}^{\infty}(\mathbb{R}^{d})\right)$ and $\nabla c \in C\hspace{-0.1cm}\left(\left[0, T\right], L^{\infty}(\mathbb{R}^{d})\right)$. 
    Moreover, 
    \begin{enumerate}[label=\textbf{(\roman*)}] 
        \item if $\; \nabla c_0 \in L^{r}(\mathbb{R}^{d})$, then $\nabla c \in C\hspace{-0.1cm}\left(\left[0, T\right], L^{r}(\mathbb{R}^{d})\right)$, for $1\leq r \leq \infty$; \label{thm:local-Weighted-space(i)}
        \item if $\; c_0 \in W^{1,r}(\mathbb{R}^{d})$, then $c \in C\hspace{-0.1cm}\left(\left[0, T\right], W^{1,r}(\mathbb{R}^{d})\right)$, for $1\leq r \leq \infty$; \label{thm:local-Weighted-space(ii)}
        \item $\displaystyle \int_{\mathbb{R}^d}|x|^\vartheta \rho (x, t) \mathrm{d} x<\infty$ for any $0\leq \vartheta<\alpha$. \label{thm:local-Weighted-space(iii)}
    \end{enumerate}
\end{proposition}
\begin{proof}
As before, we prove the existence and uniqueness of local solution to system \eqref{eq-work} by employing the fixed point theorem through \Cref{lemma-contraction}, considering the Banach space $\mathbf{X}=C\hspace{-0.1cm}\left([0, T], L_{\alpha+d}^{\infty}(\mathbb{R}^{d})\right)$ equipped with the usual norm, where $T$ depends on $\|\rho_{0}\|_{L_{\alpha+d}^{\infty}}$ and $\left\|\nabla c_{0}\right\|_{L^{\infty}}$.  
We omit the detailed steps here, as they closely follow the arguments in \cref{Theorem:local-solutions}. Moreover, \ref{thm:local-Weighted-space(i)} and \ref{thm:local-Weighted-space(ii)} 
are established using similar calculations, and their details are omitted for brevity. 

Next, we focus on proving \ref{thm:local-Weighted-space(iii)}. For that, note that $|x|^\vartheta \leq \left(1+|x-y|\right)^\vartheta \left(1+|y| \right)^\vartheta $ and $\left(1+|x|\right)^\vartheta \leq 1+\sum_{i=1}^{d} |x_i|^\vartheta$. 
Moreover, from \cref{lemma-inequality-2005}, since $\vartheta < \alpha$, we have $K^{\alpha}_t$ and $ x_i^\vartheta  K^{\alpha}_t \in L^{1}(\mathbb{R}^{d})$, for $1\leq i \leq d$ and $t>0$.  
Then, as $\rho_{0} \in L_{\alpha+d}^{\infty}(\mathbb{R}^{d})$, we obtain  
\begin{align*}
    \int_{\mathbb{R}^d} |x|^\vartheta (K^{\alpha}_t* &\rho_{0})(x) \mathrm{~d} x = \int_{\mathbb{R}^d}\int_{\mathbb{R}^d} |x|^\vartheta K^{\alpha}_t(x-y) \rho_{0}(y)\mathrm{~d} y \mathrm{~d} x  \\ 
    &\leq \int_{\mathbb{R}^d}\int_{\mathbb{R}^d} \left(1+|x-y|\right)^\vartheta \left(1+|y| \right)^\vartheta  K^{\alpha}_t(x-y) \rho_{0}(y)\mathrm{~d} y \mathrm{~d} x  \\ 
    &\leq \left(\int_{\mathbb{R}^d}\int_{\mathbb{R}^d} \left(1+|x-y|\right)^\vartheta K^{\alpha}_t(x-y)\left(1+|y| \right)^{-d-\alpha+\vartheta} \mathrm{~d} y \mathrm{~d} x\right) \left\|\rho_{0}\right\|_{L_{\alpha+d}^{\infty}}  \\ 
    &\leq \left(\left\|\left(1+|\cdot |\right)^\vartheta K^{\alpha}_t \right\|_{L^{1}}\left\| \left(1+|\cdot| \right)^{-d-\alpha+\vartheta}\right\|_{L^{1}} \right) \left\|\rho_{0}\right\|_{L_{\alpha+d}^{\infty}}  \\ 
    &= C \left\|\left(1+|\cdot |\right)^\vartheta K^{\alpha}_t \right\|_{L^{1}} \left\|\rho_{0}\right\|_{L_{\alpha+d}^{\infty}} \\
    &\leq   C \left\|\left(1+\sum_{i=1}^{d} |x_i|^\vartheta\right)  K^{\alpha}_t \right\|_{L^{1}} \left\|\rho_{0}\right\|_{L_{\alpha+d}^{\infty}}  \\
    & \leq  C\left( \left\|K^{\alpha}_t \right\|_{L^{1}}+ \sum_{i=1}^{d} \| x_i^\vartheta K^{\alpha}_t \|_{L^{1}} \right)\left\|\rho_{0}\right\|_{L_{\alpha+d}^{\infty}}  \\
    &\leq C \left(1+t^{\frac{\vartheta}{\alpha}} \right)\left\|\rho_{0}\right\|_{L_{\alpha+d}^{\infty}}, 
\end{align*}
where the last line is due to application of \Cref{lemma-inequality-2005}, as $\vartheta <\alpha$.

Moreover, from \Cref{lemma-inequality-2005}, as $\vartheta <\alpha+1$, we see that  $\nabla K^{\alpha}_t$ and  $ x_i^\vartheta  \nabla K^{\alpha}_t \in L^{1}(\mathbb{R}^{d})$ for $1\leq i \leq d$ and $t>0$. 
Then, similarly to the previous calculation, as $(\rho\nabla c)(\cdot,t) \in L_{\alpha+d}^{\infty}(\mathbb{R}^{d})$ for $t\geq 0$, we obtain 
\begin{equation*}
    \begin{split}
        \int_{\mathbb{R}^d}|x|^\vartheta \nabla K^{\alpha}_{t-s} *[\rho(s) \nabla c(s)] \mathrm{d} x &\leq C \left( \left\|\nabla K^{\alpha}_t \right\|_{L^{1}}+ \sum_{i=1}^{d} \| x_i^\vartheta \nabla K^{\alpha}_t \|_{L^{1}} \right) \left\|\rho(s) \nabla c(s)\right\|_{L_{\alpha+d}^{\infty}}\\ 
        &\leq C (t^{-\frac{1}{\alpha}}+t^{\frac{\vartheta-1}{\alpha}})  \left\|\rho(s) \right\|_{L_{\alpha+d}^{\infty}} \left\| \nabla c(s)\right\|_{L^{\infty}}. 
    \end{split}
\end{equation*}
Therefore, from \eqref{mild-solution-rho} and prior estimates, the result follows. 
\end{proof}

We can also obtain local well-posedness of system \eqref{eq-work} in the Banach space $E_{\alpha+d}(\mathbb{R}^d)$ defined as 
\begin{equation*}
    E_{\alpha+d}(\mathbb{R}^d)=\left\{\varphi\in W^{1,\infty}_{loc}(\mathbb{R}^d):\|\varphi\|_{E_{\alpha+d}} \equiv \|\varphi\|_{L_{\alpha+d}^{\infty}}+\|\nabla \varphi\|_{L_{\alpha+d}^{\infty}} <\infty\right\},
\end{equation*}
where again $\alpha$ is the order of the fractional differential operator in \eqref{eq-work}. 

For that, we consider the following estimates:
\begin{lemma}{\citep[Lemma 3.2]{brandolese2008far}} \label{lemma-norm-k-weighted-spaces-2}
    Assume $\varphi \in E_{\alpha+d}(\mathbb{R}^d)$ and $t>0$. There exists a constant $C>0$, independent of $\varphi$ and $t$, such that
    \begin{equation}
        \label{estimativa-1-weighted-space-2}
        \left\|K^{\alpha}_t *\varphi\right\|_{E_{\alpha+d}} \leq C(1+t)\left\|\varphi\right\|_{E_{\alpha+d}}, 
    \end{equation}
    \begin{equation}
        \label{estimativa-2-weighted-space-2}
        \left\|\nabla K^{\alpha}_t *\varphi\right\|_{E_{\alpha+d}} \leq C t^{-1 / \alpha}\left\|\varphi\right\|_{E_{\alpha+d}}+C t^{1-1 / \alpha}\left\|\varphi\right\|_{L^1}.
    \end{equation}
\end{lemma}
Then, we can establish the following proposition:
\begin{proposition} \label{Proposition-space-E-alpha+d}
    Let $d\geq 2$, $\alpha \in (1,2]$ and $\beta \in (1,d]$. 
    Then, for every initial condition $\rho_0 \in E_{\alpha+d}(\mathbb{R}^d)$ and $\nabla c_0 \in L^{\infty}(\mathbb{R}^d)$, there exist $T=T\left(\|\rho_{0}\|_{E_{\alpha+d}},\left\|\nabla c_{0}\right\|_{L^{\infty}}\right)$ and a unique local mild solution $(\rho, \nabla c)$ to system \eqref{eq-work} in $[0,T]$, such that $\rho \in C\hspace{-0.1cm}\left([0, T], E_{\alpha+d}(\mathbb{R}^{d})\right)$ and $\nabla c \in C\hspace{-0.1cm}\left(\left[0, T\right], L^{\infty}(\mathbb{R}^{d})\right)$. 
    Moreover, if $\nabla c_0 \in L^{r}(\mathbb{R}^{d})$ for any $1\leq r \leq \infty$, then $\nabla c \in C\hspace{-0.1cm}\left(\left[0, T\right], L^{r}(\mathbb{R}^{d})\right)$.
    Furthermore, if $ c_0 \in W^{1,r}(\mathbb{R}^{d})$, then $c \in C\hspace{-0.1cm}\left(\left[0, T\right], W^{1,r}(\mathbb{R}^{d})\right)$. 
\end{proposition}
\begin{proof}
    The proof follows the same steps as in \Cref{Theorem:local-solutions}.
\end{proof}

\subsection{Global well-posedness} \label{subsec:Global-existence} 

\begin{proof}[Proof of \Cref{Theorem:Global-solutions}]
To establish the existence and uniqueness of a global solution to system \eqref{eq-work},  by applying \textbf{Step 1, 2} and \textbf{3} with $\mathbf X$ being the Banach space given by $\mathbf X=\{u \in C\hspace{-0.1cm}\left((0, \infty), L^{p}(\mathbb{R}^{d})\right): \|u\|_{\mathbf{X}} < \infty \}$, with the norm defined by $\|u\|_{\mathbf{X}} = \underset{t>0}{\sup} \; t^{\sigma}\|u(t)\|_{L^{p}}$, where $\sigma$ is given by \eqref{definition-sigma} with $p$ and $r$ satisfying \ref{(LA2)} and \ref{(GA5)}.  
To verify that the premises in \Cref{lemma-contraction} are satisfied, along with $p$ and $r$ satisfying \ref{(LA2)} and \ref{(GA5)}, we set $p_1$ and $p_2$ as in \eqref{definition-p1-p2}, and use \hyperlink{Properties-norm-k}{Lemmas \ref{Properties-k}\ref{norm-k}}, \ref{Lemma-Estimativa-Integral}, \ref{lemma-conditions-geral-global}, \ref{Lemma-difinition-p1}, and \ref{Lemma-difinition-p2}, for $t>0$, to obtain  
\begin{align}
    &\|u_1\|_{\mathbf{X}}
    \leq C_{1}\left\|\rho_{0}\right\|_{L^{p_1}}, \label{Estimate-u0-global}\\
    &\|\mathcal{A}(u,v)\|_{\mathbf{X}} \leq C_{2} \left\| u\right\|_{\mathbf{X}} \|v\|_{\mathbf{X}},  \label{Estimate-B-global} \\
    & \left\|\mathcal{L}(u)\right\|_{\mathbf{X}} \leq C_{3} \left\|\nabla c_{0}\right\|_{L^{p_{2}}} \left\|u\right\|_{\mathbf{X}}, \label{Estimate-L-global}
\end{align}
where $C_{1}$, $C_{2}$ and $C_{3}$ are time-independent constants. 

Indeed, since $p_1=\frac{pd}{\alpha \sigma p+d}$ is such that $1 \leq p_1 < p$, as shown in \Cref{Lemma-difinition-p1}, we can apply \eqref{estimativa-1} to \eqref{def-u0}. Then,   we see that \eqref{Estimate-u0-global} is derived from
\begin{equation*}
    \|u_1\|_{L^{p}} = \left\|K^{\alpha}_{t}(x) * \rho_{0}\right\|_{L^{p}} 
    \leq C_1 t^{-\frac{d}{\alpha}\left(\frac{1}{p_1}-\frac{1}{p}\right)}\left\|\rho_{0}\right\|_{L^{p_1}} 
    = C_1t^{-\sigma}\left\|\rho_{0}\right\|_{L^{p_1}}   
\end{equation*}
and the definition of $\|\cdot\|_{\mathbf{X}}$.

To prove \eqref{Estimate-B-global} and \eqref{Estimate-L-global}, we point out, as established in \Cref{lemma-conditions-geral-global},  that parameters $p$ and $r$ satisfying \ref{(LA2)} and \ref{(GA5)} ensure $\frac{1}{p}+\frac{1}{r}\leq 1$. As a result, there exists $q$, with $1 \leq q \leq p \leq r \leq \infty$, such that $\frac{1}{q}=\frac{1}{p}+\frac{1}{r}$. 

Hence, to derive \eqref{Estimate-B-global}, we select this $q$, and apply estimate \eqref{estimativa-2} to \eqref{def-B} for $t>0$, to obtain
\begin{equation*}
    \begin{split}
    &\|\mathcal{A}(u,v)\|_{L^{p}}  \\&\leq C \int_{0}^{t}\hspace{-0.05cm}(t-s)^{-\frac{d}{\alpha}\left(\frac{1}{q}-\frac{1}{p}\right)-\frac{1}{\alpha}}\left\| u(s)\right\|_{L^{p}}\hspace{-0.05cm}\int_{0}^{s} \frac{e^{\frac{\gamma}{\tau} (w-s)}}{\tau}\left(\frac{s-w}{\tau}\right)^{-\frac{d}{\beta}\left(\frac{1}{p}-\frac{1}{r}\right)-\frac{1}{\beta}} \|v(w)\|_{L^{p}} \mathrm{d} w \mathrm{d} s \\ 
    & \leq C \int_{0}^{t}(t-s)^{-\frac{1}{\alpha}\left(\frac{d}{r}+1 \right)}  s^{-\sigma}  \int_{0}^{s} \frac{e^{\frac{\gamma}{\tau} (w-s)}}{\tau} \left(\frac{s-w}{\tau}\right)^{-\frac{d}{\beta}\left(\frac{1}{p}-\frac{1}{r}\right)-\frac{1}{\beta}}  w^{-\sigma} \mathrm{d} w \mathrm{d} s\left\| u\right\|_{\mathbf{X}} \|v\|_{\mathbf{X}}.
    \end{split}
\end{equation*}

Next, we can apply \Cref{Lemma-Estimativa-Integral} to estimate the inner integral above, since, by \Cref{lemma-conditions-geral-global}, $\sigma<1$ and $\frac{d}{\beta}\left(\frac{1}{p}-\frac{1}{r}\right)+\frac{1}{\beta}<1$. Then, 
\begin{equation*}
    \begin{split}
        \int_{0}^{s} \frac{1}{\tau} e^{\frac{\gamma}{\tau} (w-s)}\left(\frac{s-w}{\tau}\right)^{-\frac{d}{\beta}\left(\frac{1}{p}-\frac{1}{r}\right)-\frac{1}{\beta}}  w^{-\sigma} \mathrm{d} w & \leq C \int_{0}^{s} \left(s-w\right)^{-\frac{d}{\beta}\left(\frac{1}{p}-\frac{1}{r}\right)-\frac{1}{\beta}}  w^{-\sigma} \mathrm{d} w\\
        & \leq C s^{1-\sigma-\frac{d}{\beta}\left(\frac{1}{p}-\frac{1}{r}\right)-\frac{1}{\beta}},
    \end{split}    
\end{equation*}
and, since, in view of \Cref{lemma-conditions-geral-global}, we have $\frac{1}{\alpha}\left(\frac{d}{r}+1 \right)<1$ and $\sigma+\frac{1}{2} \left[\frac{d}{\beta}\left(\frac{1}{p}-\frac{1}{r}\right)+\frac{1}{\beta}\right]<1$, we can apply  \Cref{Lemma-Estimativa-Integral} again and use the above estimate to obtain
\begin{equation*}
    \begin{split}
        &\int_{0}^{t}(t-s)^{-\frac{1}{\alpha}\left(\frac{d}{r}+1 \right)}  s^{-\sigma}  \int_{0}^{s} \frac{1}{\tau} e^{\frac{\gamma}{\tau} (w-s)}\left(\frac{s-w}{\tau}\right)^{-\frac{d}{\beta}\left(\frac{1}{p}-\frac{1}{r}\right)-\frac{1}{\beta}}  w^{-\sigma} \mathrm{d} w  \mathrm{~d} s  \\
        & \hspace{5cm} \leq C \int_{0}^{t}(t-s)^{-\frac{1}{\alpha}\left(\frac{d}{r}+1 \right)}  s^{-\sigma} s^{1-\sigma-\frac{d}{\beta}\left(\frac{1}{p}-\frac{1}{r}\right)-\frac{1}{\beta}}  \mathrm{d} s \\
        & \hspace{5cm} \leq C \int_{0}^{t}(t-s)^{-\frac{1}{\alpha}\left(\frac{d}{r}+1 \right)}  s^{1-2\sigma-\frac{d}{\beta}\left(\frac{1}{p}-\frac{1}{r}\right)-\frac{1}{\beta}}  \mathrm{d} s \\
        &\hspace{5cm} \leq C t^{-\sigma}. 
    \end{split}    
\end{equation*}
Therefore, \eqref{Estimate-B-global} follows from the definition of $\sigma$ and $\|\cdot\|_{\mathbf{X}}$.

Finally, to obtain \eqref{Estimate-L-global}, we take $q$ such that $\frac{1}{q}=\frac{1}{p}+\frac{1}{r}$, and  apply estimates \eqref{estimativa-1} and \eqref{estimativa-2} to \eqref{def-L} for $t>0$. Since, from \Cref{Lemma-difinition-p2}, $1 \leq p_2 \leq r$, we obtain 
\begin{equation*}
    \begin{split}
        \left\|\mathcal{L}(u)\right\|_{L^{p}}
        &\leq C\int_{0}^{t} (t-s)^{-\frac{d}{\alpha}\left(\frac{1}{q}-\frac{1}{p}\right)-\frac{1}{\alpha}} \left\|u(s)\right\|_{L^{p}}   e^{-\frac{\gamma}{\tau} s}\left\|K^{\beta}_{\frac{s}{\tau}}(x) * \nabla c_{0}(x)\right\|_{L^{r}} \mathrm{d} s \\
        &\leq C\int_{0}^{t} (t-s)^{-\frac{1}{\alpha}\left(\frac{d}{r}+1 \right)} \left\|u(s)\right\|_{L^{p}}   e^{-\frac{\gamma}{\tau}s}\left(\frac{s}{\tau}\right)^{-\frac{d}{\beta}\left(\frac{1}{p_2}-\frac{1}{r}\right)}\left\|\nabla c_{0}\right\|_{L^{p_2}} \mathrm{d} s \\ 
        &\leq C\left( \int_{0}^{t} (t-s)^{-\frac{1}{\alpha}\left(\frac{d}{r}+1 \right)} s^{-\sigma-\frac{d}{\beta}\left(\frac{1}{p_2}-\frac{1}{r}\right) } e^{-\frac{\gamma}{\tau}s}\mathrm{d} s\right) \left\|\nabla c_{0}\right\|_{L^{p_{2}}} \left\|u\right\|_{\mathbf{X}}.
    \end{split}
\end{equation*}

Now, we can apply \Cref{Lemma-Estimativa-Integral} to estimate the integral above, since, by \Cref{lemma-conditions-geral-global,Lemma-difinition-p2}, we have $\frac{1}{\alpha}\left(\frac{d}{r}+1 \right)<1$ and $\sigma+\frac{d}{\beta}\left(\frac{1}{p_2}-\frac{1}{r}\right)<1$. Then, we obtain
\begin{equation*}
    \begin{split}
        \int_{0}^{t} (t-s)^{-\frac{1}{\alpha}\left(\frac{d}{r}+1 \right)} s^{-\sigma-\frac{d}{\beta}\left(\frac{1}{p_2}-\frac{1}{r}\right) } e^{-\frac{\gamma}{\tau}s}\mathrm{d} s&=\int_{0}^{t} (t-s)^{-\frac{1}{\alpha}\left(\frac{d}{r}+1 \right)} s^{-\sigma-\frac{d}{\beta}\left(\frac{1}{p_2}-\frac{1}{r}\right) } e^{-\frac{\gamma}{\tau}s}\mathrm{d} s\\
        &\leq \int_{0}^{t} (t-s)^{-\frac{1}{\alpha}\left(\frac{d}{r}+1 \right)} s^{-\sigma-\frac{d}{\beta}\left(\frac{1}{p_2}-\frac{1}{r}\right) } \mathrm{d} s\\
        & \leq C t^{-\sigma},
    \end{split}
\end{equation*}
as $p_{2}=\frac{dr\alpha}{\beta \left(r(\alpha-1)-d\right)+d\alpha}$ implies that $\frac{1}{\alpha}\left(\frac{d}{r}+1 \right)+\frac{d}{\beta}\left(\frac{1}{p_2}-\frac{1}{r}\right)=1$.
Thus, \eqref{Estimate-L-global} follows from the definition of $\|\cdot\|_{\mathbf{X}}$.

Note that, from the definition of $\sigma$, the estimate for $C_\mathcal{A}$ is time-independent. Similarly, the estimates for $\|u_1\|_{\mathbf{X}}$ and $C_\mathcal{L}$, from the definition of $\sigma$ and $p_1$, and of $\sigma$ and $p_2$, respectively, do not depend on time. These time independencies are crucial to apply \Cref{lemma-contraction} to obtain global well-posedness.
Beside this observation, the calculations required to obtain these estimates closely resemble the one for \Cref{Theorem:local-solutions}. We chose to include these calculations in the global result while omitting them for the local case (\Cref{subsec:Local-existence}) to emphasize this main difference between the proof of \cref{Theorem:local-solutions} and \cref{Theorem:Global-solutions}, which is the time independence of $C_\mathcal{A}$ and $C_\mathcal{L}$. 

Next, in view of inequalities \eqref{Estimate-u0-global}, \eqref{Estimate-B-global}, and \eqref{Estimate-L-global}, we set $\delta$ and the constants $C_{\mathcal{A}}$ and $C_{\mathcal{L}}$ from \Cref{lemma-contraction}  as $\delta=C_{1}\left\|\rho_{0}\right\|_{L^{p_1}}$, $C_{\mathcal{A}}=C_2$, and $C_{\mathcal{L}}=C_{3} \left\|\nabla c_{0}\right\|_{L^{p_{2}}}$. 
Then, choosing $\epsilon=\frac{1}{2 \max \left\{4C_1 C_2, 2C_{3}\right\}}$, condition \eqref{condition-global} implies that
\begin{equation*}
    C_1 \left\|\rho_{0}\right\|_{L^{p_1}}=\delta < \frac{1-2C_\mathcal{L}}{4C_\mathcal{A}}=\frac{1-2C_3\left\|\nabla c_{1}\right\|_{L^{p_2}}}{4C_2}.
\end{equation*}
Consequently, as a result of \Cref{lemma-contraction}, we prove the global existence of solutions and establish estimate \eqref{global-estimate-rho}. Therefore, $\rho \in C((0, \infty), L^{p}(\mathbb{R}^d))\cap \mathbf{X}$. 

To prove estimate \eqref{global-estimate-c},  we can apply \cref{Properties-k}\ref{norm-k} and \Cref{Lemma-Estimativa-Integral} to
\begin{equation*}
    \|\nabla c(\cdot, t)\|_{L^{r}} \leq  e^{-\frac{\gamma}{\tau} t}\left\|K^{\beta}_{\frac{t}{\tau}} * \nabla c_{0}\right\|_{L^{r}}+\int_{0}^{t} \frac{1}{\tau} e^{\frac{\gamma}{\tau} (s-t)}\left\|\nabla K^{\beta}_{\frac{t-s}{\tau}} * \rho(s)\right\|_{L^{r}} \mathrm{d} s, 
\end{equation*}
which follows from \eqref{def-eq-c}, since, from \Cref{lemma-conditions-geral-global}, we have $\sigma<1$ and $\frac{d}{\beta}\left(\frac{1}{p}-\frac{1}{r}\right)+\frac{1}{\beta}<1$. Then,  we obtain 
\begin{equation*}
    \begin{split}
        \|\nabla c(\cdot, t)&\|_{L^{r}} \leq C e^{-\frac{\gamma}{\tau}t}\left(\frac{t}{\tau}\right)^{-\frac{d}{\beta}\left(\frac{1}{p_2}-\frac{1}{r}\right)}\left\|\nabla c_{0}\right\|_{L^{p_2}}\\
        &\qquad\qquad+C \int_{0}^{t} \frac{1}{\tau} e^{\frac{\gamma}{\tau} (s-t)}\left(\frac{t-s}{\tau}\right)^{-\frac{d}{\beta}\left(\frac{1}{p}-\frac{1}{r}\right)-\frac{1}{\beta}} \|\rho(s)\|_{L^{p}} \mathrm{d} s \\ 
        &\leq C \left(\frac{t}{\tau}\right)^{-\frac{d}{\beta}\left(\frac{1}{p_2}-\frac{1}{r}\right)}\left\|\nabla c_{0}\right\|_{L^{p_2}}+C \int_{0}^{t} \frac{1}{\tau} \left(\frac{t-s}{\tau}\right)^{-\frac{d}{\beta}\left(\frac{1}{p}-\frac{1}{r}\right)-\frac{1}{\beta}} s^{-\sigma} \mathrm{d} s \left\|\rho \right\|_{\mathbf{X}} \\
        & \leq C t^{-\frac{d}{\beta}\left(\frac{1}{p_2}-\frac{1}{r}\right)}\left\|\nabla c_{0}\right\|_{L^{p_2}}+ C t^{1-\sigma-\frac{d}{\beta}\left(\frac{1}{p}-\frac{1}{r}\right)-\frac{1}{\beta}}  \left\|\rho \right\|_{\mathbf{X}}.
    \end{split}
\end{equation*}

Therefore, by multiplying both sides of the inequality above by $t^{1-\frac{1}{\alpha}\left(\frac{d}{r}+1 \right)}$ and using the definitions of $\sigma$ and $p_2$, as well as estimate \eqref{global-estimate-rho}, inequality \eqref{global-estimate-c} is obtained.

To prove that $\rho \in L^{\infty}\left((0, \infty), L^{p_1}(\mathbb{R}^{d})\right) $, consider $p$ satisfying \eqref{condition-2-global-p-1} or $p$ and $r$ satisfying \eqref{condition-2-global-p-2} and \eqref{condition-2-global-r-1}.  
Note that, since from \Cref{lemma-condition-extra-p-r}  we have $\frac{\alpha}{d}\sigma \leq  \frac{1}{r}$, $q$ defined as $1/q=1/p+1/r$ is such that $q \leq p_1$. Indeed, $\frac{1}{p_1}=\frac{\alpha}{d}\sigma+\frac{1}{p} \leq \frac{1}{q}=\frac{1}{p}+\frac{1}{r}$. 
Then, for $t>0$, we can apply \Cref{Properties-k}\ref{norm-k} to obtain
\begin{equation*}
    \begin{split}
        \|\rho(t)\|_{L^{p_1}}& \leq \|K^{\alpha}_{t}* \rho_{0}\|_{L^{p_1}} + \left\| \int_{0}^{t}  \nabla K^{\alpha}_{t-s} *\left[\rho(s) \nabla c(s) \right] \mathrm{d} s \right\|_{L^{p_1}}\\ 
        & \leq C \|\rho_{0}\|_{L^{p_1}} +  \int_{0}^{t}  \left\|\nabla K^{\alpha}_{t-s} *\left[\rho(s) \nabla c(s) \right] \right\|_{L^{p_1}} \mathrm{d} s\\ 
        & \leq C \|\rho_{0}\|_{L^{p_1}} +  C \int_{0}^{t}(t-s)^{-\frac{d}{\alpha}\left(\frac{1}{q}-\frac{1}{p_1}\right)-\frac{1}{\alpha}} \left\|\rho(s) \nabla c(s) \right\|_{L^{q}} \mathrm{d} s. 
    \end{split}
\end{equation*}

Moreover, we can apply \eqref{global-estimate-rho}, \eqref{global-estimate-c}, and \Cref{Lemma-Estimativa-Integral} to estimate the integral above, since, from \Cref{lemma-conditions-geral-global},  we have $\frac{1}{\alpha}\left(\frac{d}{r}+1\right)-1<\sigma<\frac{1}{\alpha}\left(\frac{d}{r}+1\right)$. Therefore,
\begin{align*}
        \|&\rho(t)\|_{L^{p_1}}  \leq C \|\rho_{0}\|_{L^{p_1}} +   C \int_{0}^{t}(t-s)^{-\frac{d}{\alpha}\left(\frac{1}{p}+\frac{1}{r} -\frac{\alpha}{d}\sigma-\frac{1}{p}\right)-\frac{1}{\alpha}} \left\|\rho(s) \right\|_{L^{p}}  \left\| \nabla c(s) \right\|_{L^{r}} \mathrm{d} s \\ 
        & \leq C \|\rho_{0}\|_{L^{p_1}} \\
        &\quad+  C \int_{0}^{t}(t-s)^{-\frac{1}{\alpha}\left(\frac{d}{r}+1\right)+\sigma} s^{-\sigma-1+\frac{1}{\alpha}\left(\frac{d}{r}+1\right)}s^{\sigma} \left\|\rho(s) \right\|_{L^{p}} s^{1-\frac{1}{\alpha}\left(\frac{d}{r}+1\right)}   \left\| \nabla c(s) \right\|_{L^{r}} \mathrm{d} s \\ 
        & \leq C \|\rho_{0}\|_{L^{p_1}}  \\
        &\quad + C \left(\int_{0}^{t}(t-s)^{-\frac{1}{\alpha}\left(\frac{d}{r}+1\right)+\sigma} s^{-\sigma-1+\frac{1}{\alpha}\left(\frac{d}{r}+1\right)}\mathrm{d} s\right)   \left\|\rho_0 \right\|_{L^{p_1}} \left( \left\|\rho_0 \right\|_{L^{p_1}}+\left\| \nabla c_0 \right\|_{L^{p_2}}  \right)    \\ 
        &\leq   C \left\|\rho_0 \right\|_{L^{p_1}} \left(1+\left\|\rho_0 \right\|_{L^{p_1}}+\left\| \nabla c_0 \right\|_{L^{p_2}}  \right).   
\end{align*}
\end{proof}

\begin{remark}
If we assume $\alpha=\beta=2$, then system \eqref{eq-work} is the classical one. In that case, for $d=2$, \Cref{Theorem:Global-solutions} becomes: 
\begin{center}
\begin{minipage}{0.95\textwidth}
Let $\rho_0 \in L^{1}(\mathbb{R}^{d})$ and $\nabla c_0 \in L^{2}(\mathbb{R}^{d})$. Then, there exists $\delta>0$ such that, if $ \|\rho_{0}\|_{L^{1}} +\left\|\nabla c_{0}\right\|_{L^{2}}< \delta$, there is a unique global mild solution $\rho \in C\hspace{-0.1cm}\left((0, \infty), L^{p}(\mathbb{R}^{d})\right) \cap \mathbf{X}$,  $\nabla c \in C\hspace{-0.1cm}\left((0, \infty), L^{r}(\mathbb{R}^{d})\right)$, where $\frac{4}{3}<p<2$ and $r=\frac{p}{p-1}$ (\ie $2<r<4$). 
Moreover, $\rho \in L^{\infty}\left((0, \infty), L^{1}(\mathbb{R}^{d})\right) $ conserves mass, $\sup_{t \geq 0} t^{1-\frac{1}{p}}\|\rho(\cdot, t)\|_{L^{p}} \leq C\|\rho_{0}\|_{L^{1}}$, and $\sup_{t \geq 0} t^{\frac{1}{p}-\frac{1}{2}}\left\|\nabla c(\cdot, t)\right\|_{L^{r}} <C \left\|\nabla c_{0}\right\|_{L^{2}}+ C  \left\|\rho_0 \right\|_{L^{1}}$.       
In addition, if $\int_{\mathbb{R}^{d}} c_{0}(x) \mathrm{d}x < \infty$, the concentration of the chemical, $c$, grows as in \eqref{c_decay1} and \eqref{c_decay2}.     
\end{minipage}
\end{center}
\end{remark}

\begin{proposition}
\label{Global-solutions-corollary}
    Let $d=2$, and consider system \eqref{eq-work} with $\alpha=\beta \; \in \left(5/3, \; 2 \right]$. Assume that \ref{(LA2)} is in force and, in addition, 
    \begin{equation}
        \label{Global-solutions-corollary-r-1}
        r\leq \min \left\{ \frac{p}{p(\alpha-1)-1}, \; \frac{1}{2-\alpha} \right\}.
    \end{equation}
    Consider the Banach space $\mathbf{X}$ defined as 
    \begin{equation*}
        \mathbf{X}=\{u \in C\hspace{-0.1cm}\left([0, \infty), L^{p}(\mathbb{R}^{d})\right): \|u\|_{\mathbf{X}} \equiv \underset{t>0}{sup} \; t^{\sigma} \left\|u (t) \right\|_{L^{p}} < \infty \},
    \end{equation*}
    where $\displaystyle \sigma=2\left(\frac{\alpha-1}{\alpha}\right)-\frac{2}{\alpha p}$. 
    Then, there exists $\epsilon>0$ such that, if 
    \begin{equation}
        \label{condition-corollary}
        \|\rho_{0}\|_{L^{\frac{1}{\alpha-1}}} +\left\|\nabla c_{0}\right\|_{L^{\frac{2}{\alpha-1}}}< \epsilon,
    \end{equation}
    there exists a unique global mild solution $\rho \in C((0, \infty), L^{p}(\mathbb{R}^{d})) \cap \mathbf{X}$,  $\nabla c \in C((0, \infty),$ $ L^{r}(\mathbb{R}^{d}))$, with initial condition $\rho_0 \in L^{\frac{1}{\alpha-1}}(\mathbb{R}^{d})$, $\nabla c_0 \in L^{\frac{2}{\alpha-1}}(\mathbb{R}^{d})$. 
    Moreover, $\rho \in L^{\infty}((0, \infty), L^{\frac{1}{\alpha-1}}(\mathbb{R}^{d})) $,  
    \begin{equation}
        \label{global-estimate-rho-2}
        \sup_{t \geq 0} t^{\sigma}\|\rho(\cdot, t)\|_{L^{p}} \leq C\|\rho_{0}\|_{L^{\frac{1}{\alpha-1}}}, 
    \end{equation}
    and    
    \begin{equation}
        \label{global-estimate-c-2}
        \sup_{t \geq 0} t^{1-\frac{1}{\alpha}\left(\frac{2}{r}+1 \right)}\left\|\nabla c(\cdot, t)\right\|_{L^{r}} <C \left(\left\|\nabla c_{0}\right\|_{L^{\frac{2}{\alpha-1}}}+ \left\|\rho_0 \right\|_{L^{\frac{1}{\alpha-1}}} \right).
    \end{equation}
    
    Furthermore, if $\rho_0 \in L^1(\mathbb{R}^{d})$, then 
    $\rho(\cdot,t) \in L^1(\mathbb{R}^{d})$ 
    for $t>0$, the corresponding solution conserves mass and, for $\displaystyle \int_{\mathbb{R}^{d}} c_{0} \mathrm{~d}x < \infty$, the chemical concentration, $c$, grows as 
    \begin{equation*}
        \int_{\mathbb{R}^{d}} c\left(x,t\right) \mathrm{d}x =\left(\int_{\mathbb{R}^{d}}\rho_0(x) \mathrm{d}x\right) \left(\frac{1-e^{-\frac{\gamma}{\tau} t}}{\gamma}\right)+\left(\int_{\mathbb{R}^{d}} c_0 (x) \mathrm{d}x\right)  e^{-\frac{\gamma}{\tau} t}
    \end{equation*}
    and  $\displaystyle \int_{\mathbb{R}^{d}} c\left(x,t\right) \mathrm{d}x \xrightarrow[\; t\rightarrow \infty \;]{}\frac{\int_{\mathbb{R}^{d}}\rho_0(x) \mathrm{d}x}{\gamma}$, which includes $\gamma=0$ by taking $\gamma \rightarrow 0$.  
\end{proposition}
\begin{remark}
For $d=2$ and and $\alpha=\beta \; \in \left(5/3, \; 2 \right]$, assumption \hyperlink{(LA2)(a)}{(A2a)} reduces to 
\begin{equation*} 
    \frac{4}{\alpha+1}<p<\frac{2}{\alpha-1} \qquad \text{ and } \qquad \max \left\{ \frac{2}{\alpha-1}, \; \frac{p}{p-1}\right\}<r<\frac{2p}{2-p(\alpha-1)},
\end{equation*}
where $r= \frac{p}{p-1}\; $ is also allowed if $\; \frac{p}{p-1}>\frac{d}{\alpha-1}$.
%
Moreover, the conditions in \eqref{condition-2-global-p-1}, \eqref{condition-2-global-p-2}, and  \eqref{condition-2-global-r-1} from \Cref{Theorem:Global-solutions} simplify to  
\begin{equation*}
    p \leq \frac{2}{3}\frac{2}{\alpha-1}, \qquad \frac{2}{3}\frac{2}{\alpha-1}<p< \frac{2}{\alpha-1}, \qquad \text{ and } \qquad r \leq \frac{p}{p(\alpha-1)-1},
\end{equation*}
respectively.

Therefore, in more detail, the ranges for $p$ and $r$ defined in \Cref{Global-solutions-corollary} can be specified as
\begin{equation*}
    \left\{\begin{aligned}
        \frac{4}{\alpha+1}<p\leq \frac{4}{3\alpha-3}  & \text{ and } \max \left\{ \frac{2}{\alpha-1}, \frac{p}{p-1}\right\}<r<\min \left\{ \frac{2p}{2-p(\alpha-1)},   \frac{1}{2-\alpha} \right\}, \\
         \text{or} & \\ 
        \frac{4}{3\alpha-3}<p<\frac{2}{\alpha-1}  &  \text{ and }  \max \left\{ \frac{2}{\alpha-1}, \frac{p}{p-1}\right\}<r\leq \min \left\{ \frac{p}{p(\alpha-1)-1}, \frac{1}{2-\alpha} \right\}.
    \end{aligned}\right.
\end{equation*}
In the first line, $r$ can be equal to $\frac{1}{2-\alpha}$ if $\frac{1}{2-\alpha}<\frac{2p}{2-p(\alpha-1)}$ and, in both cases, it can be equal to $\frac{p}{p-1}$ if $\; \frac{p}{p-1}>\frac{d}{\alpha-1}$.
\end{remark}
\begin{proof}
Note that the hypotheses of \Cref{Global-solutions-corollary} fall under the hypotheses of \Cref{Theorem:Global-solutions}. Consequently, we can apply \Cref{Theorem:Global-solutions} to establish the existence, uniqueness, and asymptotic behavior described in \eqref{global-estimate-rho-2} and \eqref{global-estimate-c-2} of the solution.

To prove that $\rho(\cdot,t) \in L^1(\mathbb{R}^{d})$, for $t>0$, and verify mass conservation and chemical concentration behavior, it suffices to show that $(\rho \nabla c)(\cdot, s) \in  L^{1}(\mathbb{R}^{d})$ for all $s\geq 0$. 
Indeed, by applying \hyperlink{Properties-norm-k}{\Cref{Properties-k}\ref{norm-k}} to \eqref{mild-solution-rho} and $\alpha>1$, we obtain 
\begin{equation*}
    \begin{split}
        \|\rho(\cdot, t)\|_{L^{1}} 
        & \leq \|\rho_{0}\|_{L^{1}} +  \int_{0}^{t}  \left\|\nabla K^{\alpha}_{t-s} *\left[\rho(s) \nabla c(s) \right] \right\|_{L^{1}} \mathrm{~d}s\\ 
        & \leq \|\rho_{0}\|_{L^{1}} +  C \int_{0}^{t}(t-s)^{-\frac{1}{\alpha}} \left\|\rho(s) \nabla c(s)\right\|_{L^{1}}\mathrm{~d}s \\  
        & \leq \|\rho_{0}\|_{L^{1}} +  Ct^{1-\frac{1}{\alpha}} \sup_{s\geq 0} \left\|(\rho\nabla c) (\cdot, s)\right\|_{L^{1}},
\end{split}
\end{equation*}
considering $(\rho \nabla c)(\cdot, s) \in  L^{1}(\mathbb{R}^{d})$ for all $s\geq 0$. 
Thus, $\rho(\cdot,t) \in L^1(\mathbb{R}^d)$ for every $t>0$. 

Next, integrating both sides of equation \eqref{mild-solution-rho}, we can apply \Cref{integral-k2}, for $\rho_0 \in L^1(\mathbb{R}^d)$ and $(\rho \nabla c)(\cdot, s) \in  L^{1}(\mathbb{R}^{d})$ for all $s\geq 0$. Thus, it is straightforward that $\displaystyle \int_{\mathbb{R}^d} \rho(x, t) \mathrm{~d} x=\int_{\mathbb{R}^d} \rho_0(x) \mathrm{~d} x$. 
Similarly, since $\rho(\cdot,t) \in L^1(\mathbb{R}^d)$ for every $t>0$, we integrate both sides of \eqref{def-eq-c} and apply \Cref{integral-k2} to deduce the chemical concentration behavior.

Finally, to prove that $(\rho \nabla c)(\cdot, t) \in  L^{1}(\mathbb{R}^{d})$ for all $t> 0$, observe that  $\rho \in L^{\infty}((0, \infty), L^{\frac{1}{\alpha-1}}(\mathbb{R}^{d}) \cap L^{p}(\mathbb{R}^{d}))$. This implies $\rho \in L^{\varsigma}(\mathbb{R}^{d}) $, where $\varsigma \in \left[\frac{1}{\alpha-1},p \right]$, for all $t\geq 0$. From \eqref{Global-solutions-corollary-r-1}, we have $\frac{1}{r}+\alpha-1\geq 1$. 
As a result, there is $q \in \left[\frac{1}{\alpha-1},p \right]$ such that $\frac{1}{q}+\frac{1}{r}=1$.
Therefore, $\left\|(\rho \nabla c) (\cdot,t) \right\|_{L^{1}} \leq C \left\|\rho(\cdot,t)\right\|_{L^{q}} \left\|\nabla c (\cdot,t)\right\|_{L^{r}} <\infty$.
\end{proof}

\section{Nonnegative Solutions} \label{sec:Nonnegative-Solutions}

As mentioned before, system \eqref{eq-work} describes the density of cells $\rho$ and the concentration of the chemical signal $c$. Therefore, for nonnegative initial densities, $\rho_0 \geq 0$ and $ c_0 \geq 0$, it is biologically relevant to ensure that $\rho$ and $ c$ remain nonnegative over time. 
Thus, we prove that, for a nonnegative initial condition, a sufficiently regular solution remains nonnegative.  
For that, consider the following lemma (see \Cref{definition-f-f+}).

\begin{lemma} \label{nonnegative-1}
    Let $w \in C^1\left([0, T], W^{1,p}(\mathbb{R}^d)\right)$ and $2\leq p <\infty$. Then, we obtain that 
    \begin{equation}
        \label{part-1}
         \int_{\mathbb{R}^d} - (w_{-})^{p-1} \partial_t w  \mathrm{d} x=\frac{1}{p} \frac{\mathrm{~d}}{\mathrm{~d}t} \left\| w_{-}(\cdot, t)\right\|_{L^p}^p. 
    \end{equation}
    Moreover, for $v \in W^{2, \infty}(\mathbb{R}^d)$, we obtain 
    \begin{equation}
        \label{part-2}
        \int_{\mathbb{R}^d}  (w_{-})^{p-1} \nabla \cdot \left(w  \nabla v \right) \mathrm{d} x \leq \frac{p-1}{p} \left\|\Delta v \right\|_{L^{\infty}} \left\| w_{-}(\cdot, t)\right\|_{L^p}^p, 
    \end{equation}
    for every $t \in [0, T]$.
\end{lemma} 
\begin{proof}
Let $w \in C^1\left([0, T], W^{1,p}(\mathbb{R}^d)\right)$  and $p \geq 2$. Since $w=w_{+}-w_{-}$ and, in view of \eqref{derivpm}, $w_{-}\partial_t(w_{+})=0$  then
\begin{equation*}
    - (w_{-})^{p-1} \partial_t w = (w_{-})^{p-1} \partial_t (w_{-}) = \frac{1}{p} \partial_t (w_{-})^{p},  
\end{equation*}
which implies \eqref{part-1}. 

Moreover, we have in the weak sense,
\begin{eqnarray*}
     w \nabla (w_{-})^{p-1}&=&\left(p-1\right)\left((w_{-})^{p-2}  \nabla (w_{-})\right) w
    =-\left(p-1\right)(w_{-})^{p-1}  \nabla (w_{-})\\
    &=&-\frac{p-1}{p} \nabla (w_{-})^{p}.
\end{eqnarray*}

Then, assuming $v \in W^{2, \infty}(\mathbb{R}^d)$, we obtain 
\begin{align*}
    \int_{\mathbb{R}^d} (w_{-})^{p-1} \nabla \cdot \left(w  \nabla v \right) \mathrm{d} x & =-\frac{p-1}{p} \int_{\mathbb{R}^d}   w_{-}^{p} \, \nabla \cdot \nabla v \mathrm{~d} x \\
    &\leq \frac{p-1}{p} \left\|\Delta v\right\|_{L^{\infty}} \left\| w_{-}(\cdot, t)\right\|_{L^p}^p.
\end{align*} 
\end{proof}
\begin{theorem} \label{Thm:positivity-local-in-time}
  
   Let $(\rho, \nabla c)$ be the local-in-time mild solution of system \eqref{eq-work} established by \Cref{Theorem:Higher-Regularity-local}, with a nonnegative initial condition $(\rho_0, \nabla c_0)$. 
    Assume that the parameters $p$ and $r$ in \Cref{Theorem:Higher-Regularity-local}, in addition to \ref{(LA2)}, satisfy $p\geq 2$ and $r< \infty $ if $\beta>2$. Additionally, let the regularity index $N$ from \Cref{Theorem:Higher-Regularity-local} be given by $N=\max\{2, 2\lceil \beta/2 \rceil - 1 \}$.
    Then, the local-in-time solution satisfies the positivity conditions, 
    \[
        \rho(\cdot,t) \geq 0 \quad \text{and} \quad c(\cdot,t) \geq 0, 
    \]
    for all $t \in [0, T]$.  
\end{theorem}
\begin{proof}
The idea is to prove that $\|\rho_{-}\|_{L^p}=0$ and $\|\nabla (c_{-})\|_{L^p}=0$. To do so, we will use \cref{nonnegative-1}.  Thus, let us show that $\Delta c \in  L^1([0, T], L^{\infty}(\mathbb{R}^d))$. 
Applying \eqref{estimativa-1} and \eqref{estimativa-2} to \eqref{def-eq-c}, we obtain the following estimate for $0<t\leq T$ 
\begin{align*}
    \| \Delta c(\cdot, t)\|_{L^{\infty}} &\leq  e^{-\frac{\gamma}{\tau} t}\left\| K^{\beta}_{\frac{t}{\tau}} *  \Delta c_{0}\right\|_{L^{\infty}}+\int_{0}^{t} \frac{1}{\tau} e^{\frac{\gamma}{\tau} (s-t)}\left\| K^{\beta}_{\frac{t-s}{\tau}} * \Delta \rho(\cdot,s)\right\|_{L^{\infty}} \mathrm{d} s \\
    &\leq C e^{-\frac{\gamma}{\tau} t} t^{-\frac{1}{\beta}\frac{d}{q_0}} \left\| \Delta c_{0}\right\|_{L^{q_0}}+C \int_{0}^{t} \frac{1}{\tau} e^{\frac{\gamma}{\tau} (s-t)} \left(\frac{s-t}{\tau}\right)^{-\frac{d}{\beta}\frac{1}{p}} \|\Delta \rho(\cdot,s)\|_{L^{p}} \mathrm{d} s,  \\
    &\leq C e^{-\frac{\gamma}{\tau} t} t^{-\frac{1}{\beta}\frac{d}{q_0}}  \left\| \Delta c_{0}\right\|_{L^{q_0}} \\
    & \hspace{2cm} +C  \left(\int_{0}^{t} \frac{1}{\tau} e^{\frac{\gamma}{\tau} (s-t)} \left(\frac{s-t}{\tau}\right)^{-\frac{d}{\beta}\frac{1}{p}} \mathrm{d} s\right) \sup_{t \in[0, T]} \left\|\rho(\cdot, t)\right\|_{W^{2,p}}, 
\end{align*}
since, as  $N \geq 2$, the assumption $\nabla c_0 \in W^{N,q_0}(\mathbb{R}^d)$ implies $\Delta c_0 \in L^{q_0}(\mathbb{R}^d)$, and $\rho \in C([0,T], W^{2,p}(\mathbb{R}^d))$. 
Moreover, from \Cref{Theorem:Higher-Regularity-local}, the parameter $p$ satisfies \ref{(LA2)}, guaranteeing $p>d/\beta$. This ensures the inner integral above converges, yielding, for $\gamma \neq 0$, 
\begin{equation*}
    \begin{split}
        \int_{0}^{t} \frac{1}{\tau} e^{\frac{\gamma}{\tau} (s-t)}\left(\frac{t-s}{\tau}\right)^{-\frac{d}{\beta}\frac{1}{p}}  \mathrm{d} s &=\gamma^{-\left(1-\frac{d}{\beta}\frac{1}{p}\right)} \int_{0}^{\gamma t/\tau} e^{-u} u^{\left(1-\frac{d}{\beta}\frac{1}{p}\right)-1} \mathrm{d} u \\
        & \leq C \Gamma \left(1-\frac{d}{\beta}\frac{1}{p}\right) < \infty,
    \end{split}
\end{equation*}
 and, for $\gamma =0$,  
\begin{equation*}
    \int_{0}^{t} \frac{1}{\tau} \left(t-s\right)^{-\frac{d}{\beta}\frac{1}{p}}  \mathrm{d} s \leq C t^{1-\frac{d}{\beta}\frac{1}{p}}.
\end{equation*}

Furthermore, from \Cref{Theorem:Higher-Regularity-local}, the parameter $q_0$ for the initial condition must be either $q_0=r$ or $q_0= \wp$, where $r$ satisfies hypotheses \ref{(LA2)} and $\wp$ satisfies \eqref{codition-r-0-Theorem}. In both cases, $q_0>d/\beta$, since $\beta>1$, $r>\frac{d}{\alpha-1}$ and $\wp\geq \frac{d}{\alpha-1}$. Consequently, the first term in the estimate above is integrable over $[0,T]$, leading to $\Delta c \in L^1((0,T], L^\infty(\mathbb{R}^d))$.  

Now, we multiply the first equation of  \eqref{eq-work} by $(\rho_{-})^{p-1}$ and integrate over $\mathbb{R}^d$. Since, from  the hypotheses $p\geq 2$ and, as $p$ satisfies \ref{(LA2)}, $p <\infty$, we can use \eqref{part-1}, \eqref{part-2}, and \Cref{Positivity-Lemma-2} to obtain 
\begin{equation*}
    \frac{1}{p} \frac{\mathrm{~d}}{\mathrm{~d}t} \left\|  \rho_{-}(\cdot, t)\right\|_{L^p}^p \leq \frac{p-1}{p} \left\|\Delta c (\cdot, t)\right\|_{L^\infty} \left\|  \rho_{-}(\cdot, t)\right\|_{L^p}^p.
\end{equation*}
Setting $\eta(t)=\left\|  \rho_{-}(\cdot, t)\right\|_{L^p}^p$ and $\phi(t)=\left(p-1\right)\left\|\Delta c (\cdot, t)\right\|_{L^\infty}$, by Gronwall's inequality, we get
\begin{equation*}
    \eta(t) \leq e^{\int_0^t \phi(s) d s}\eta(0)
\end{equation*}
for all $0 \leq t \leq T$, as $\phi$ is nonnegative and summable on $[0, T]$ (as proved previously). Since $\eta(0)=0$, as $\rho_0\geq 0$, it follows that $\eta \equiv 0$ on  $[0, T]$. 
Hence, $\rho_{-}= 0$ and, consequently, $\rho \geq 0$. 

Now, from \Cref{Theorem:Higher-Regularity-local}, 
we have that, if $\beta \in (1,2]$, then $\nabla c \in C((0, T], W^{N,q}(\mathbb{R}^{d}))$ for $q>2$. Therefore, since $r$ satisfies \ref{(LA2)} and, by hypothesis,  $r< \infty $ if $\beta>2$, it follows from \Cref{Theorem:Higher-Regularity-local} that there exists $2<q<\infty$, such that $\nabla c \in C([0, T], W^{N,q}(\mathbb{R}^{d}))$ for any $\beta \in (1,d]$ (with $q=r$ if $\beta>2$). 
Then, by multiplying the second equation of \eqref{eq-work} by $(c_{-})^{q-1}$ and integrate over $\mathbb{R}^d$, using \eqref{part-1} and \Cref{Positivity-Lemma-2}, we find 
\begin{equation*}
    \frac{1}{q} \frac{\mathrm{~d}}{\mathrm{~d}t} \left\|  c_{-}(\cdot, t)\right\|_{L^q}^q \leq -\left[\int_{\mathbb{R}^d} \left(c_{-}(x, t) \right)^{q-1}  \rho(x, t) \mathrm{d}x +\gamma \left\|  c_{-}(\cdot, t)\right\|_{L^q}^q \right]\leq 0,       
\end{equation*}
since $\rho\geq 0$ and $c_{-} \geq 0$ implies that the first term inside the brackets is nonnegative. 
Then, as $\left\|  c_{-}(\cdot, 0)\right\|_{L^q}=0$ (because $c_0 \geq 0$), it follows that $c_{-}= 0$, hence $c\geq 0$. 
This completes the proof.
\end{proof}

\begin{theorem}
   Let $(\rho, \nabla c)$ be the global-in-time mild solution of system \eqref{eq-work} guaranteed by  \Cref{Theorem:Higher-Regularity-global}, with a nonnegative initial condition $(\rho_0, \nabla c_0)$. 
    Assume that $p$ and $r$ in \Cref{Theorem:Higher-Regularity-global}, along with conditions \ref{(LA2)} and \ref{(GA5)}, satisfy $p\geq 2$ and $r< \infty $. Furthermore, set the regularity index $N$ to $N=\max\{2, 2\lceil \beta/2 \rceil - 1 \}$.
    Then, the global-in-time solution satisfies the positivity conditions for all $t \geq 0$. 
\end{theorem}

The proof of the above theorem follows by extending the arguments of \Cref{Thm:positivity-local-in-time} to the global case, with estimates extended over $(0,\infty)$ instead of $(0,T]$. 

\appendix

\section{Parameters for Lebesgue Spaces} 
\label{Parameters-for-Lebesgue-Spaces}

In \cref{subsec:Local-existence,subsec:Global-existence}, suitable restrictions are imposed on the parameters of the system and on the parameters related to the space of initial data and solution, as described in \hyperlink{Comments-about-the-Proof}{  \textbf{Comments about the Proof. (1)}}.

In this Appendix, we present the proof of the existence of $p$ and $r$ satisfying those restrictions and show that they imply \eqref{condition-geral} and  \eqref{condicao-L-unica} for the local results, and \eqref{condition-geral}, \eqref{condition-global-3}, \eqref{definition-p2-apendice}, $1 \leq p_1 < p$ and $1 < p_2 < r $ for the global results. 
Such conditions constitute the hypotheses in  H\"older's inequality and in \hyperlink{Properties-norm-k}{Lemmas \ref{Properties-k}\ref{norm-k}} and \ref{Lemma-Estimativa-Integral}.
\begin{remark}
    Throughout this appendix, consider $\alpha \in (1,2]$, $\beta \in (1,d]$, and $d \geq 2$.
\end{remark}
\begin{definition} \label{Definition-sigma-appendix}
Let $1 < p\leq r \leq \infty$. We define 
\begin{equation}
    \label{definition-sigma-apendice}
     \sigma=2-\frac{1}{\alpha}\left(\frac{d}{r}+1 \right)-\frac{d}{\beta}\left(\frac{1}{p}-\frac{1}{r}\right)-\frac{1}{\beta}.
\end{equation}
\end{definition}
\begin{lemma}\label{lemma-conditions-geral-global} 
Consider $p$ and $r$ satisfying \ref{(LA2)}, that is, 
\begin{subequations}
    \label{definition-p-q-appendix}
\begin{multline}
    \label{definition-p-q-appendix-A} 
    \hspace{0.8cm}  \max \left\{ \frac{2d}{d+\beta-1}, \; \frac{d}{\alpha+\beta-2}\right\}< p \leq \frac{d}{\beta -1} \quad \text{ and } \\
     \quad   \max  \left\{p, \; \frac{p}{p-1}, \; \frac{d}{\alpha-1}\right\}< r  <\frac{pd}{d-p(\beta-1)},  \quad \text{ or }\hspace{0.5cm} 
\end{multline} 
\begin{equation}
    \label{definition-p-q-appendix-B}
    p > \frac{d}{\beta -1}  \qquad \text{ and } \qquad  r>\max  \left\{p, \; \frac{p}{p-1}, \; \frac{d}{\alpha-1}\right\},  
\end{equation}
\end{subequations}
where, in both cases, the equality $r=\max  \left\{p, \; \frac{p}{p-1}\right\}$ is possible if $\max  \left\{p, \; \frac{p}{p-1}\right\}>\frac{d}{\alpha-1}$. 

\hypertarget{Part 1}{\textbf{Part 1.}} Assume $2\beta \left(\alpha-1\right)-\alpha\geq 0$. Then, there exist $p$ and $r$ satisfying \eqref{definition-p-q-appendix} with 
\begin{equation}
   \label{condition-global-p-1}
   p<\frac{d\alpha}{2\beta \left(\alpha-1\right)-\alpha}. 
\end{equation} 

\hypertarget{Part 2}{\textbf{Part 2:}} Consider $p$ and $r$ satisfying \eqref{definition-p-q-appendix}. Then $p>1$, and the following conditions hold

\begin{subequations}
    \label{condition-geral}
    \noindent\begin{minipage}{.3\linewidth}
    \begin{equation}
        \label{condition-geral(a)}
        \frac{1}{\alpha}\left(\frac{d}{r}+1 \right)<1,  
    \end{equation}
    \end{minipage}
    \begin{minipage}{.4\linewidth}
    \begin{equation}
        \label{condition-geral(b)}
        \frac{d}{\beta}\left(\frac{1}{p}-\frac{1}{r}\right)+\frac{1}{\beta}<1,
    \end{equation}
    \end{minipage}
    \begin{minipage}{.26\linewidth}
    \begin{equation}
        \label{condition-geral(c)}
        \frac{1}{p}+\frac{1}{r}\leq 1.
    \end{equation}    
    \end{minipage}
\end{subequations}
\\

Additionally, if $2\beta \left(\alpha-1\right)-\alpha< 0$ or $p$ also satisfies \eqref{condition-global-p-1}, then  

\begin{subequations}
    \label{condition-global-3}
    \noindent\begin{minipage}{.47\linewidth}
    \begin{equation}
        \label{condition-global-3(d)}
        \sigma+\frac{1}{2} \left[\frac{d}{\beta}\left(\frac{1}{p}-\frac{1}{r}\right)+\frac{1}{\beta}\right]<1,
    \end{equation}
    \end{minipage}
    \begin{minipage}{.51\linewidth}
    \begin{equation}
        \label{condition-global-3(e)}
        \frac{1}{\alpha}\left(\frac{d}{r}+1\right)-1<\sigma<\frac{1}{\alpha}\left(\frac{d}{r}+1\right). 
    \end{equation}
    \end{minipage}
\end{subequations}
\end{lemma}
\begin{proof}
\hyperlink{Part 1}{\textbf{Part 1.}} Note that, since by assumption $2\beta \left(\alpha-1\right)-\alpha\geq 0$, we have 
\begin{equation}
    \label{eq-lemma-condition-extra-p-r-eq-0}
    \frac{d\alpha}{2\beta \left(\alpha-1\right)-\alpha}\geq \frac{d}{\beta-1},
\end{equation}
as $\alpha\leq 2$. 
Therefore, for $\alpha<2$, there exists $p$ in the range defined by \eqref{definition-p-q-appendix-B} and \eqref{condition-global-p-1}, \ie such that $\frac{d}{\beta -1}<p<\frac{d\alpha}{2\beta \left(\alpha-1\right)-\alpha}$. Then, for $\alpha<2$, there exist numbers $p$ and $r$ in the ranges defined by \eqref{definition-p-q-appendix-B} where $p$ satisfies \eqref{condition-global-p-1}.

Note that (regardless of whether $2\beta \left(\alpha-1\right)-\alpha\geq 0$) there also exist numbers $p$ and $r$ in the ranges defined in \eqref{definition-p-q-appendix-A}.  Indeed, for $\displaystyle p \leq \frac{d}{\beta -1} $, we have 
\begin{align*}
     & \frac{d}{\beta -1}>\frac{2d}{d+\beta-1}  & \text{ as } &  & \beta&<d+1, \hspace{3cm} \\
    & \frac{d}{\beta -1}>\frac{d}{\alpha+\beta-2}  & \text{ as } &  &  \alpha&>1.
\end{align*}
Moreover, since $d-p(\beta-1)\geq 0$,
\begin{align*}
     & \frac{pd}{d-p(\beta-1)}>\frac{p}{p-1}  & \text{ as } &  &  p&> \frac{2d}{d+\beta-1}, \hspace{2.67cm}\\
     & \frac{pd}{d-p(\beta-1)}> \frac{d}{\alpha-1}   & \text{ as } &  &  p&>\frac{d}{\alpha+\beta-2}, \hspace{2.67cm} \\
     & \frac{pd}{d-p(\beta-1)}>p   & \text{ as } & & \beta&>1. 
\end{align*}

Then, for $\alpha<2$, there are  $p$ and $r$, with $p$ satisfying \eqref{condition-global-p-1}, in the ranges defined by either cases in \eqref{definition-p-q-appendix}. 
On the other hand, for $\alpha=2$, such numbers exist only in the ranges defined by \eqref{definition-p-q-appendix-A}, and we must have $p< \frac{d}{\beta -1}$ so that \eqref{condition-global-p-1} also holds. 
Therefore, it is always possible to find numbers $p$ and $r$  in the ranges defined in \eqref{definition-p-q-appendix-A} or \eqref{definition-p-q-appendix-B}  where $p$ also satisfies \eqref{condition-global-p-1}.

\hyperlink{Part 2}{\textbf{Part 2.}} 
Note that $p>1$ follows from the fact that $\frac{2d}{d+\beta-1}>1$, since $1<\beta<d+1$, and $p>\frac{2d}{d+\beta-1}$.

Now, from either restrictions in \eqref{definition-p-q-appendix}, we obtain 
\begin{equation}
    \label{condition-geral-r}
     \frac{1}{r}<\frac{\alpha-1}{d} \quad \text{ and } \quad \frac{1}{r}\leq 1-\frac{1}{p}, 
\end{equation}
and conditions \eqref{condition-geral(a)} and \eqref{condition-geral(c)} follow immediately.

Next, assuming \eqref{definition-p-q-appendix-A}, we see that $r  <\frac{pd}{d-p(\beta-1)}$ and $d-p(\beta-1)\leq 0$, and assuming \eqref{definition-p-q-appendix-B}, we have $\frac{1}{p}-\frac{\beta-1}{d}<0$ and $\frac{1}{r}\geq 0$. Then, from either cases, it follows that
\begin{equation}
    \label{condition-geral-r-2}
    \frac{1}{p}-\frac{\beta-1}{d}<\frac{1}{r},
\end{equation}
which implies \eqref{condition-geral(b)}. 

To prove \eqref{condition-global-3(d)},  note that if $2\beta \left(\alpha-1\right)-\alpha < 0$, $\frac{\beta-1}{d}+\frac{\beta(\alpha-2)}{d\alpha}<0$. Otherwise,  $p$ satisfies \eqref{condition-global-p-1}.
Then, it follows from both scenarios that $\frac{1}{p}>\frac{\beta-1}{d}+\frac{\beta(\alpha-2)}{d\alpha}$, implying $\alpha \left(\frac{\beta-1}{d}-\frac{1}{p}\right)+\frac{\beta(\alpha-2)}{d}<0$. Hence, as $\frac{1}{r}\geq 0$ and $2\beta-\alpha>0$, we have 
\begin{equation*}
    \frac{1}{r}(2\beta-\alpha)>\alpha \left(\frac{\beta-1}{d}-\frac{1}{p}\right)+\frac{\beta(\alpha-2)}{d}, 
\end{equation*}
which leads to
\begin{equation*}
    1-\frac{1}{\alpha}\left(\frac{d}{r}+1 \right)-\frac{1}{2} \left[\frac{d}{\beta}\left(\frac{1}{p}-\frac{1}{r}\right)+\frac{1}{\beta}\right]<0,
\end{equation*}
proving \eqref{condition-global-3(d)}. 

Finally, from \eqref{condition-geral(a)} and \eqref{condition-geral(b)}, we obtain 
\begin{equation*}
    \begin{split}
        \frac{1}{\alpha}\left(\frac{d}{r}+1\right)-1&<\sigma-3+\frac{2}{\alpha}\left(\frac{d}{r}+1 \right)+\frac{d}{\beta}\left(\frac{1}{p}-\frac{1}{r}\right)+\frac{1}{\beta} \\
        &<\sigma+2\left[\frac{1}{\alpha}\left(\frac{d}{r}+1 \right)-1\right]+\left[\frac{d}{\beta}\left(\frac{1}{p}-\frac{1}{r}\right)+\frac{1}{\beta}-1\right] <\sigma, 
    \end{split}
\end{equation*}
and from \eqref{condition-global-3(d)},
\begin{equation*}
    \begin{split}
        \sigma-\frac{1}{\alpha}\left(\frac{d}{r}+1 \right) &=2\left\{1-\frac{1}{\alpha}\left(\frac{d}{r}+1 \right)-\frac{1}{2}\left[\frac{d}{\beta}\left(\frac{1}{p}-\frac{1}{r}\right)+\frac{1}{\beta}\right]\right\}\\ 
        &=2\left\{\sigma-1+\frac{1}{2}\left[\frac{d}{\beta}\left(\frac{1}{p}-\frac{1}{r}\right)+\frac{1}{\beta}\right]\right\}\\
        &=2\left\{\sigma+\frac{1}{2}\left[\frac{d}{\beta}\left(\frac{1}{p}-\frac{1}{r}\right)+\frac{1}{\beta}\right]\right\}-2<0, 
    \end{split}
\end{equation*}
which proves \eqref{condition-global-3(e)}.
\end{proof}
\begin{remark} \label{Biler-fractional-9-parabolic-elliptic}
    The lower bound on $p$ in \Cref{lemma-conditions-geral-global} is the same as that in  \cite[Theorem 2.1]{Biler-fractional-9}  used to establish the local and global existence of mild solutions for the parabolic-elliptic form of \eqref{eq-work}. 
\end{remark} 
\begin{remark} \label{remark-lemma-conditions-geral-global-2}
Consider $\sigma$ defined in \eqref{definition-sigma-apendice} and $p$ and $r$ in the ranges defined in  \Cref{lemma-conditions-geral-global}. 
\begin{enumerate}[topsep=1pt,partopsep=0pt, left=-1pt,label=\textbf{\arabic*.}]
    \item 
    \begin{enumerate}[label=\textbf{(\alph*)}] 
        \item \label{remark-1(a)} We have $p<2$ iff $\frac{p}{p-1}>2$ (since this is true for any number greater than zero).
        \item From the fact that $\alpha \in (1,2]$, $\beta \in (1,d]$, $d \geq 2$, the following relations are established:
        \begin{itemize}[itemsep=5pt, topsep=3pt]
            \item $\displaystyle \frac{d}{\alpha-1} \geq 2$,  thus $\displaystyle r>2$; 
            \item $\displaystyle \frac{p}{p-1}>\frac{d}{\alpha-1}$,   iff $\displaystyle 1<p< \frac{d}{d+1-\alpha}$;
            \item $\displaystyle p<2$ if $\displaystyle p< \frac{d}{d+1-\alpha}$,  as $\displaystyle \frac{d}{d+1-\alpha} \leq 2$ for any $\displaystyle \alpha$ and $\displaystyle d$; 
            \item $\displaystyle \frac{d}{d+1-\alpha}>\frac{2d}{d+\beta-1}$ iff $\displaystyle 2\alpha+\beta>d+3$;
            \item $\displaystyle \frac{2d}{d+\beta-1}>\frac{d}{\alpha+\beta-2}$ iff $\displaystyle 2\alpha+\beta>d+3$.  
        \end{itemize}
        \item If $\displaystyle r=\frac{p}{p-1}$, we can state that $\displaystyle \frac{d}{\alpha-1}<r<\frac{2d}{d-\beta+1}< \frac{pd}{d-p(\beta-1)}$.  
    \end{enumerate}
    \item In more detail, from \cref{remark-1(a)}, \eqref{definition-p-q-appendix} can be rewritten as \label{remark-2}
    \begin{enumerate}[label=\textbf{(\alph*)}] 
        \item for $2\alpha+\beta\leq d+3$,
    \end{enumerate}
    \begingroup
        \addtolength{\jot}{0.6em}
    \begin{align}
        \max  \left\{\frac{d}{\alpha-1}, \; p\right\} < r& < \frac{pd}{d-p(\beta-1)} & \; \text{ if } \; & \; \frac{d}{\alpha+\beta-2}  < p  \leq \frac{d}{\beta-1}, \label{condition-geral-p-and-r-(a)-1}\\
        r &> \max  \left\{\frac{d}{\alpha-1},\; p\right\}  & \; \text{ if } \;&   \textcolor{white}{\frac{d}{\alpha+\beta-2}  <} p  > \frac{d}{\beta-1}, \label{condition-geral-p-and-r-(a)-2}
    \end{align}
    \endgroup
    \begin{enumerate}[label=\textbf{(\alph*)}, resume] 
        \item otherwise, for \; $ 2\alpha+\beta>d+3$, 
    \end{enumerate}
    \begingroup
        \addtolength{\jot}{0.6em}
    \begin{align}
        \frac{p}{p-1} \leq r &< \frac{pd}{d-p(\beta-1)} & \;  \text{ if }  \; &  \frac{2d}{d+\beta-1}<p< \frac{d}{d+1-\alpha}, \label{condition-geral-p-and-r-(b)-1}\\
         \max \left\{\frac{d}{\alpha-1},\; p\right\} < r& < \frac{pd}{d-p(\beta-1)} & \text{ if } \; &  \frac{d}{d+1-\alpha}  \leq p  \leq \frac{d}{\beta-1}, \; \; \label{condition-geral-p-and-r-(b)-2}\\
        r &>  \max  \left\{\frac{d}{\alpha-1},\; p\right\}  & \; \text{ if } \; &  \textcolor{white}{\frac{d}{\alpha+\beta-2}  <} p  > \frac{d}{\beta-1}, \label{condition-geral-p-and-r-(b)-3}
    \end{align}
    \endgroup
    \item [] with, in both cases, the equality $r=p$ being possible if $\displaystyle p>d/(\alpha-1)$.
    \item Adding \eqref{condition-global-p-1}, for $2\beta \left(\alpha-1\right)-\alpha\geq 0$, \cref{remark-2} incorporates the following changes: 
    \item [] $\bullet$ \textbf{Case A.} $\alpha<2$ 
    \item [] \hspace{0.20cm} addition of condition  \eqref{condition-global-p-1} to lines \eqref{condition-geral-p-and-r-(a)-2} and \eqref{condition-geral-p-and-r-(b)-3}, \ie
    \begin{equation*}
        r > \max  \left\{\frac{d}{\alpha-1},\; p\right\}     \quad \text{ if } \qquad  \frac{d}{\beta-1}  <  p  <\frac{d\alpha}{2\beta \left(\alpha-1\right)-\alpha}. 
    \end{equation*}
    \item [] $\bullet$ \textbf{Case B.}  $\alpha=2$
    \item [] \hspace{0.20cm} exclusion of \eqref{condition-geral-p-and-r-(a)-2} and \eqref{condition-geral-p-and-r-(b)-3}; replacement of $p\leq \frac{d}{\beta -1} $ with $p< \frac{d}{\beta -1} $ in \eqref{condition-geral-p-and-r-(a)-1}  and \eqref{condition-geral-p-and-r-(b)-2}.
    \item From \eqref{definition-sigma-apendice}, \eqref{condition-geral(a)} and \eqref{condition-geral(c)}, we see that the parameter $\sigma$ is such that $\sigma>0$, which is a better lower bound than the one given by \eqref{condition-global-3(e)}. 
    Moreover, \eqref{condition-global-3(e)} ensures that $\sigma<1$. 
    Then, seeking a better upper bound for $\sigma$ that has the advantage over the one in \eqref{condition-global-3(e)} of depending only on $\alpha$ and $\beta$, we find, when $\alpha, \beta \in (1,2)$, that
    \begin{equation*}
        \sigma<\frac{\max{\left\{\alpha,\beta \right\}}-2}{\min{\left\{\alpha,\beta \right\}}}+1.
    \end{equation*}    
    Indeed, we can rewrite $\sigma$ defined in \eqref{definition-sigma-apendice} as 
    \[\sigma=2-\frac{1}{\alpha\beta}\left[d\left(\frac{1}{r}\left(\beta-\alpha\right)+\frac{\alpha}{p}\right)+\beta+\alpha\right].\] 
    Then, due to relation \eqref{condition-geral-r}, if $\alpha \geq \beta$, we obtain 
    \begin{equation*}
        \sigma<2-\frac{1}{\alpha\beta}\left[d\left(\frac{\alpha-1}{d}\left(\beta-\alpha\right)+\frac{\alpha}{p}\right)+\beta+\alpha\right]
        =\frac{\alpha-1}{\beta}+\frac{\beta-1}{\beta}-\frac{d}{\beta p}<\frac{\alpha-2}{\beta}+1; 
    \end{equation*}
    and, due to relation \eqref{condition-geral-r-2}, if $\alpha<\beta$, we obtain
    \begin{eqnarray*}
        \sigma &<&2-\frac{1}{\alpha\beta}\left\{d\left[\left(\frac{1}{p}-\frac{\beta-1}{d}\right)\left(\beta-\alpha\right)+\frac{\alpha}{p}\right]+\beta+\alpha\right\} \\
        &=&\frac{\alpha-1}{\alpha}+\frac{\beta-1}{\alpha} -\frac{d}{\alpha p} 
        <\frac{\beta-2}{\alpha}+1.
    \end{eqnarray*}
    Therefore, the estimate follows.
\end{enumerate}
\end{remark}
\begin{lemma}\label{lemma-condition-initial-local}
Consider $p$ and $r$ in the ranges defined in \eqref{definition-p-q-appendix} and let $\wp $ be given by 
\begin{equation}
    \label{eq-lemma-condition-initial-local}
    \begin{aligned}
        & \wp \in \left[\frac{d}{\alpha-1}, r \right] & \quad \text{ if } \quad \alpha\leq \beta, \\ 
        & \wp =r & \quad \text{ if } \quad \alpha>\beta.
    \end{aligned}  
\end{equation}
Then, $2 \leq \wp  \leq r$ and  
\begin{equation}
    \label{condicao-L-unica}
    \frac{1}{\alpha}\left(\frac{d}{r}+1\right)+\frac{d}{\beta}\left(\frac{1}{\wp }-\frac{1}{r}\right) \leq 1. 
\end{equation}
\end{lemma}
\begin{proof}
It is easy to see that $2 \leq \wp  \leq r$ from the definition of $\wp$ and the fact that $r>\frac{d}{\alpha-1}\geq 2$.  
Next, if $\alpha>\beta$, condition \eqref{condicao-L-unica} falls into \eqref{condition-geral(a)}. Otherwise, if $\alpha\leq \beta$, then $\frac{1}{\wp }\leq \frac{\alpha-1}{d}$, and with this
\begin{equation*}
    \begin{split}
        \frac{1}{\alpha}\left(\frac{d}{r}+1\right)+\frac{d}{\beta}\left(\frac{1}{\wp}-\frac{1}{r}\right) &\leq \frac{1}{\alpha}\left(\frac{d}{r}+1\right)+\frac{d}{\beta}\left(\frac{\alpha-1}{d}-\frac{1}{r}\right) \\&
        =\frac{1}{\alpha}\left(\frac{d}{r}+1\right)\left(1-\frac{\alpha}{\beta}\right)+\frac{\alpha}{\beta}. 
    \end{split}
\end{equation*}
Therefore, if $\alpha< \beta$, from condition \eqref{condition-geral(a)}, we have
\begin{equation*}
    \frac{1}{\alpha}\left(\frac{d}{r}+1\right)\left(1-\frac{\alpha}{\beta}\right)+\frac{\alpha}{\beta}<1-\frac{\alpha}{\beta}+\frac{\alpha}{\beta}=1,
\end{equation*}
and if $\alpha= \beta$,
\begin{equation*}
    \frac{1}{\alpha}\left(\frac{d}{r}+1\right)\left(1-\frac{\alpha}{\beta}\right)+\frac{\alpha}{\beta}=\frac{\alpha}{\beta}=1.
\end{equation*}
Note that the equality in \eqref{condicao-L-unica} holds only for $\alpha= \beta$ and in that case $\wp = \frac{d}{\alpha-1}$. 
\end{proof}

The following lemmas are pertinent to \cref{subsec:Global-existence}. They address the proof of some inequalities that must be satisfied by $p_1$ and $p_2$, the parameters defining the space of the initial data: $\rho_0 \in L^{p_1}(\mathbb{R}^d)$ and $\nabla c_0 \in L^{p_2}(\mathbb{R}^d)$. 

\begin{lemma} \label{Lemma-difinition-p1}
    Consider $p$ and $r$ as in \eqref{definition-p-q-appendix}, with $p$ satisfying \eqref{condition-global-p-1} if $\; 2\beta \left(\alpha-1\right)-\alpha\geq 0$, and
    $p_1=\frac{pd}{\alpha \sigma p+d}$. 
    Then, $1 \leq p_1 < p$. 
\end{lemma}
\begin{proof}
Note that $p_1 \geq 1 $ if and only if $\left[\frac{d}{\beta}\left(\frac{1}{p}-\frac{1}{r}\right)+\frac{1}{\beta}\right](\beta-\alpha) \leq d+2(1-\alpha)$. Moreover, as $\alpha\leq 2$ and $d\geq 2$, we see that $d+2(1-\alpha)\geq 0$. Therefore, if $\beta \leq \alpha$, the inequality is automatically verified. 
Otherwise, from \eqref{condition-geral(b)} we obtain $\left[\frac{d}{\beta}\left(\frac{1}{p}-\frac{1}{r}\right)+\frac{1}{\beta}\right](\beta-\alpha)<\beta-\alpha$, 
and, as $\beta \leq d$ and $\alpha \leq 2$, it follows that $\beta-\alpha\leq d-\alpha+2-\alpha =d+2(1-\alpha)$; thus, $p_1\geq 1$.  
Furthermore, $p_1 < p$, as $\alpha \sigma p>0$ implies $\frac{d}{\alpha \sigma p+d}<1$. 
\end{proof}

\begin{lemma} \label{Lemma-difinition-p2}
    Consider $p$ and $r$ as in \eqref{definition-p-q-appendix}, with $p$ satisfying \eqref{condition-global-p-1} if $\; 2\beta \left(\alpha-1\right)-\alpha\geq 0$, 
    and 
    \begin{equation}
        \label{definition-p2-apendice}
        p_{2}=\frac{dr\alpha}{\beta \left(r(\alpha-1)-d\right)+d\alpha}.
    \end{equation}
Then, $1 < p_2 < r $ and the following condition is satisfied
    \begin{equation}
        \label{condicao-L-global-apendice}
        \sigma+\frac{d}{\beta}\left(\frac{1}{p_2}-\frac{1}{r}\right)<1. 
    \end{equation}
\end{lemma}
\begin{proof}
Note first that $\beta \left(r(\alpha-1)-d\right)+d\alpha>0$ as $r>\frac{d}{\alpha-1}$. 
Then, \eqref{definition-p2-apendice} can be rewritten as 
\begin{equation}
    \label{definition-p2-apendice-2}
    \frac{1}{p_{2}}=\frac{\beta}{d}\left(1-\frac{1}{\alpha}\left(\frac{d}{r}+1 \right)\right)+\frac{1}{r},
\end{equation}
and it follows from condition \eqref{condition-geral(a)} that $\frac{1}{p_{2}}> \frac{1}{r}$. Thus, $r>p_2$. 
Moreover, $p_2>1$. Indeed, from \eqref{definition-p2-apendice-2}, conditions \eqref{condition-global-3(d)}  and \eqref{condition-geral(c)} and the fact that $d \geq 2$, we obtain
\begin{equation*}
    \begin{split}
        \frac{1}{p_{2}} 
        &=\frac{\beta}{d}\left[\sigma+\frac{1}{2}\left(\frac{d}{\beta}\left(\frac{1}{p}-\frac{1}{r}\right)+\frac{1}{\beta}\right)-1+\frac{1}{2}\left(\frac{d}{\beta}\left(\frac{1}{p}-\frac{1}{r}\right)+\frac{1}{\beta}\right)\right]+\frac{1}{r}\\ 
        &\leq\frac{1}{2}\left(\frac{1}{p}+\frac{1}{r}+\frac{1}{d}\right) \\
        &\leq \frac{3}{4}<1.
    \end{split} 
\end{equation*}

Next, note that \eqref{definition-p2-apendice-2} leads to $\frac{d}{\beta}\left(\frac{1}{p_{2}}-\frac{1}{r}\right)=1-\frac{1}{\alpha}\left(\frac{d}{r}+1 \right)$. Therefore,  condition \eqref{condicao-L-global-apendice} is equivalent to  $\sigma<\frac{1}{\alpha}\left(\frac{d}{r}+1 \right)$, which is true from estimate \eqref{condition-global-3(e)}.
\end{proof}

In the following lemma, new ranges for the parameters $p$ and $r$ are defined to satisfy the supplementary condition: $r \sigma\leq d/\alpha$. 
This condition is essential for proving that the global mild solution $\rho$ belongs to the space $L^{p_1}(\mathbb{R}^{d})$, where the parameter $p_1$ (defined in \Cref{Lemma-difinition-p1}) specifies the space of the initial data $\rho_0 \in L^{p_1}(\mathbb{R}^d)$. 

\begin{lemma} \label{lemma-condition-extra-p-r}
Consider $p$ and $r$ satisfying 
\begin{subequations}
\label{definition-p-q-appendix-2}
\begin{multline}
    \label{condition-analise-p-r-2-def}
    \hspace{1cm}  \max  \left\{ \frac{2d}{d+\beta-1}, \; \frac{d}{\alpha+\beta-2}\right\}<p \leq \frac{2d}{(\alpha-1)+2(\beta-1)} \quad \text{ and } \\
     \quad   \max  \left\{p, \; \frac{p}{p-1}, \; \frac{d}{\alpha-1}\right\}< r  <\frac{pd}{d-p(\beta-1)},  \quad \text{ or }\hspace{0.5cm} 
\end{multline}
\begin{multline}
    \label{condition-analise-p-r-3-def}
    \hspace{1cm}  \frac{2d}{(\alpha-1)+2(\beta-1)} < p < \frac{\alpha d }{\max\left\{2\beta(\alpha-1)-\alpha,  \; \alpha(\alpha-2)+\beta \right\}} \quad \text{ and } \\
     \quad   \max  \left\{p, \; \frac{p}{p-1}, \; \frac{d}{\alpha-1}\right\}<r \leq \frac{(2\beta-\alpha)pd}{[\beta(\alpha-1)+\alpha(\beta-1)]p-\alpha d}, \hspace{0.5cm} 
\end{multline}
\end{subequations} 

\noindent where, in both cases, the equality $r=\max  \left\{p, \; \frac{p}{p-1}\right\}$ is possible if $\max  \left\{p, \; \frac{p}{p-1}\right\}>\frac{d}{\alpha-1}$. 

\hypertarget{Part 1-2}{\textbf{Part 1:}} There exist $p$ and $r$ such that \eqref{definition-p-q-appendix-2} holds. 

\hypertarget{Part 2-2}{\textbf{Part 2:}} Consider $p$ and $r$ satisfying constraint \eqref{definition-p-q-appendix-2}. Then $p>1$, and the following conditions are satisfied: 

\begin{subequations}
    \label{condition-geral-2}
    \noindent\begin{minipage}{.47\linewidth}
    \begin{equation}
        \label{condition-geral-2(a)}
        \frac{1}{\alpha}\left(\frac{d}{r}+1 \right)<1,
    \end{equation}
    \end{minipage}
    \begin{minipage}{.47\linewidth}
    \begin{equation}
        \label{condition-geral-2(b)}
        \frac{d}{\beta}\left(\frac{1}{p}-\frac{1}{r}\right)+\frac{1}{\beta}<1,
    \end{equation}
    \end{minipage}

    \noindent\begin{minipage}{.47\linewidth}
    \begin{equation}
        \label{condition-geral-2(c)}
        \frac{1}{p}+\frac{1}{r}\leq 1, 
    \end{equation}
    \end{minipage}
    \begin{minipage}{.47\linewidth}
    \begin{equation}
        \label{condition-geral-2(d)}
        \sigma+\frac{1}{2} \left[\frac{d}{\beta}\left(\frac{1}{p}-\frac{1}{r}\right)+\frac{1}{\beta}\right]<1,
    \end{equation}
    \end{minipage}

    \noindent\begin{minipage}{.47\linewidth}
    \begin{equation}
        \label{condition-geral-2(e)}
        -1<\sigma-\frac{1}{\alpha}\left(\frac{d}{r}+1\right)<0, 
    \end{equation}
    \end{minipage}
    \begin{minipage}{.47\linewidth}
    \begin{equation}
        \label{condition-geral-2(f)}
        r \sigma\leq d/\alpha.
    \end{equation}
    \end{minipage}
\end{subequations}
\end{lemma} 

\begin{remark}
Consider $p$ and $r$ satisfying \eqref{condition-analise-p-r-3-def}. Note that, since $\beta>\alpha(2-\alpha)$ for any $\alpha \in (1,2]$, $\max\left\{2\beta(\alpha-1)-\alpha,  \; \alpha(\alpha-2)+\beta \right\}>0$. 
Moreover,  as $\beta>\alpha/2$ implies $\frac{2d}{(\alpha-1)+2(\beta-1)}>\frac{\alpha d}{\beta(\alpha-1)+\alpha(\beta-1)}$, we have $p>\frac{\alpha d}{\beta(\alpha-1)+\alpha(\beta-1)}$, and hence $[\beta(\alpha-1)+\alpha(\beta-1)]p-\alpha d>0$.  
Therefore, 
\begin{equation*}
    \frac{\alpha d }{\max\left\{2\beta(\alpha-1)-\alpha,  \; \alpha(\alpha-2)+\beta \right\}}>0  \quad \text{ and } \quad \frac{(2\beta-\alpha)pd}{[\beta(\alpha-1)+\alpha(\beta-1)]p-\alpha d}>0.
\end{equation*}
\end{remark}

\begin{proof}
\hyperlink{Part 1-2}{\textbf{Part 1:}} 
Consider first the restriction given by \eqref{condition-analise-p-r-2-def}.  
For $(\alpha,\beta) \neq (2,d)$, we obtain 
\begin{align*}
     & \frac{2d}{(\alpha-1)+2(\beta-1)}>\frac{2d}{d+\beta-1}  & \text{ as } &  & \alpha+\beta&<d+2, \hspace{3cm} \\
    & \frac{2d}{(\alpha-1)+2(\beta-1)}>\frac{d}{\alpha+\beta-2}  & \text{ as } &  &  \alpha&>1.
\end{align*}
Moreover, note that 
\begin{equation}
    \label{eq-lemma-condition-extra-p-r-eq-1}
    \frac{2d}{(\alpha-1)+2(\beta-1)}<\frac{d}{\beta -1}
\end{equation}
as $\alpha>1$. Then, $p \leq \frac{2d}{(\alpha-1)+2(\beta-1)} $ is also less then $\frac{d}{\beta -1} $. Thus, $d-p(\beta-1)\geq 0$ and $\frac{pd}{d-p(\beta-1)}>\max\left\{\frac{p}{p-1},  \; \frac{d}{\alpha-1},  \; p\right\}$, as we have already shown in \Cref{lemma-conditions-geral-global}. 
Therefore, there exist numbers $p$ and $r$  in the ranges defined in \eqref{condition-analise-p-r-2-def} for $(\alpha,\beta) \neq (2,d)$. 

Now, consider the restriction given by \eqref{condition-analise-p-r-3-def} in the case $(\alpha,\beta) = (2,d)$. In this scenario, \eqref{condition-analise-p-r-3-def} becomes
\begin{equation}
    \label{condition-analise-p-r-3-def-2}
    \frac{2d}{2d-1} < p < \frac{2d }{2d-2} \quad \text{ and } 
     \quad   \frac{p}{p-1} \leq r \leq \frac{2(d-1)pd}{2(d-1)p+dp-2 d}, 
\end{equation}
since  $\frac{2d }{2d-2} \leq d < \frac{p}{p-1} < 2d$.
Notice that there exists number $r$ in the range defined by \eqref{condition-analise-p-r-3-def-2}, since  $\frac{2d}{2d-1}<p<2$ implies $-\frac{1}{2}<\frac{d}{2(d-1)}-\frac{2 d}{2(d-1)p}<0$, which leads to
\begin{equation*}
    \frac{p}{p-1} \leq \frac{2(d-1)pd}{2(d-1)p+dp-2 d}=\frac{d}{1+\left(\frac{d}{2(d-1)}-\frac{2 d}{2(d-1)p}\right)}.
\end{equation*}
Moreover, it is trivial to see that there exists a number $p$ in the range defined by \eqref{condition-analise-p-r-3-def-2}.
Thus, we obtain that there exist numbers $p$ and $r$ in the ranges defined by \eqref{condition-analise-p-r-3-def} for $(\alpha,\beta) = (2,d)$.

\hyperlink{Part 2-2}{\textbf{Part 2:}} To prove \eqref{condition-geral-2}, first we assume \eqref{condition-analise-p-r-2-def}.
In that case, note that $p$ also satisfies \eqref{condition-global-p-1} if $2\beta \left(\alpha-1\right)-\alpha\geq0$, since  from \eqref{eq-lemma-condition-extra-p-r-eq-1} and \eqref{eq-lemma-condition-extra-p-r-eq-0} it follows that 
\begin{equation*} 
    \frac{2d}{(\alpha-1)+2(\beta-1)}<\frac{\alpha d}{2\beta \left(\alpha-1\right)-\alpha}.
\end{equation*}

Moreover, from \eqref{eq-lemma-condition-extra-p-r-eq-1}, we see that $p$ and $r$ also satisfy \eqref{definition-p-q-appendix-A}. Therefore, from \Cref{lemma-conditions-geral-global}, \eqref{condition-geral-2} follows.

On the other hand, if we assume \eqref{condition-analise-p-r-3-def}, we have, for  $d-p(\beta-1)\geq0$, 
\begin{equation*}
    \frac{(2\beta-\alpha)pd}{[\beta(\alpha-1)+\alpha(\beta-1)]p-\alpha d}<\frac{pd}{d-p(\beta-1)},
\end{equation*}
since $p>\frac{2d}{(\alpha-1)+2(\beta-1)}$. Moreover, $p$ satisfies \eqref{condition-global-p-1} if $2\beta \left(\alpha-1\right)-\alpha\geq 0$, since in that case
\begin{equation*}
    \frac{\alpha d}{\max\left\{2\beta(\alpha-1)-\alpha,  \; \alpha(\alpha-2)+\beta \right\}} \leq  \min\left\{\frac{\alpha d}{2\beta(\alpha-1)-\alpha},  \; \frac{\alpha d}{\alpha(\alpha-2)+\beta }\right\}.
\end{equation*}
Then, again from \Cref{lemma-conditions-geral-global}, \eqref{condition-geral-2} follows.

To prove \eqref{condition-geral-2(f)}, note that this is equivalent to 
\begin{equation*}
    r\leq \frac{(2\beta-\alpha)pd}{[\beta(\alpha-1)+\alpha(\beta-1)]p-\alpha d}.
\end{equation*}
Then, assuming \eqref{condition-analise-p-r-2-def}, we have 
\begin{equation}
    \label{eq-lemma-condition-extra-p-r-eq-3}
    \frac{pd}{d-p(\beta-1)}\leq \frac{(2\beta-\alpha)pd}{[\beta(\alpha-1)+\alpha(\beta-1)]p-\alpha d},
\end{equation}
since $p\leq \frac{2d}{(\alpha-1)+2(\beta-1)} $, and \eqref{condition-geral-2(f)} follows. 
On the other hand, assuming \eqref{condition-analise-p-r-3-def}, \eqref{condition-geral-2(f)} follows directly from \eqref{eq-lemma-condition-extra-p-r-eq-3}. 
\end{proof}

\section{Fractional Laplacian}
\label{Fractional-Laplacian}

In this appendix, we present and prove essential properties of the Fractional Laplacian for completeness. 
We begin by introducing a definition of the nonlocal operator $(-\Delta)^{s}$ for any $s \in \mathbb{R}_{+}\setminus\mathbb{N}$. 
\begin{definition} \label{Definition-fractional-powers-Laplacian}
    Let $s=m+\sigma$, where $\sigma \in(0,1)$ and $m \in \mathbb{N}_0$. Then, the operator $(-\Delta)^{s}$ can be defined as \citep{Positive-powers-of-Laplacian,loss-of-maximum-principles,Fractional-integrals-and-derivatives} 
    \begin{equation}
        \label{eq-definition-fractional-Laplacian-integral-2}
        (-\Delta)^{s} u(x):=\frac{c_{d, 2s}}{2} P.V.\int_{\mathbb{R}^d} \frac{\delta_{m+1} u(x, y)}{|y|^{d+2s}} d y, \quad x \in \mathbb{R}^d,
    \end{equation} 
    where $P.V.$ stands for the Cauchy principal value, $c_{d, s}$ is the positive normalization constant  
    \begin{equation}
        \label{eq-definition-fractional-Laplacian-constant-2}
        c_{d, 2s}:=\frac{4^s \Gamma\left(\frac{d}{2}+s\right)}{\pi^{\frac{d}{2}}\Gamma\left(-s\right)}\left(\sum_{k=1}^{m+1} (-1)^k  \begin{pmatrix}
         2(m+1) \\
         m+1-k
        \end{pmatrix} k^{2s}\right)^{-1}, 
    \end{equation} 
    and $\delta_{m+1} u$ is the finite difference of order $2(m+1)$ of $u$
    \begin{equation}
        \label{eq-definition-fractional-Laplacian-integral-3}
        \delta_{m+1} u(x, y):=\sum_{k=-m-1}^{m+1}(-1)^k\left(\begin{array}{c}
        2(m+1) \\
        m+1-k
        \end{array}\right) u(x+k y) \quad \text { for } x, y \in \mathbb{R}^d.
    \end{equation}    
\end{definition}

Note that, by a change of variables, it is easy to see that  \eqref{eq-definition-fractional-Laplacian-integral-2} can be rewritten as 
\begin{equation}
    \label{eq-definition-fractional-Laplacian-integral}
    (-\Delta)^{s}  u(x)=c_{d, 2s} P.V. \int_{\mathbb{R}^d} \frac{u(x)-u(y)}{|x-y|^{d+2s}} \mathrm{d} y, \quad x \in \mathbb{R}^d,
\end{equation} 
where the normalization constant is given by
\begin{equation}
    \label{eq-definition-fractional-Laplacian-constant}
    c_{d, 2s}:=\frac{4^{s} s \Gamma\left(d/2+s\right)}{\pi^{d / 2} \left|\Gamma\left(1-s\right)\right|}=\frac{4^s \Gamma\left(d/2+s\right)}{\pi^{d / 2}\left|\Gamma\left(-s\right)\right|}. 
\end{equation}   

\begin{remark}
    In literature, the definition of the fractional Laplacian as $\Lambda^{2s}=(-\Delta)^{s}$ generally applies for $0<s<1$ and can be characterized in multiple equivalent ways (see \citep[]{What-is-the-fractional-Laplacian-2020, Ten-Equivalent-Definitions-Fractional-Laplace-2017}). 
    One such characterization is as the singular integral operator \eqref{eq-definition-fractional-Laplacian-integral} \cite{New-developments-nonlocal-operators-I,Guo-2015-Partial-Differential-Equations}.
\end{remark}

We start by showing that $\Lambda^{\alpha}=(-\Delta)^{\alpha/2}$ maps $W^{k,p}(\mathbb{R}^{d})$ to $L^p(\mathbb{R}^{d})$ for certain values of $p$ and $k$. 
For that purpose, consider the following lemma.

\begin{lemma}{\citep[Theorem 1.2]{Taylor-expansions}} \label{Lemma-Taylor-polynomial}
    Let $u \in W^{k, p}(\mathbb{R}^{d})$, for some $1\leq p \leq \infty$, and $k$ be a nonnegative integer. Then, the following estimate holds 
    \begin{equation}
        \label{Taylor-polynomial}
        \left\|u(\cdot+y)-P_ku(\cdot, y)\right\|_{L^p} \leq C \| u\|_{W^{k,p}}|y|^{k},
    \end{equation}
where $C>0$ depends only on $d$ and $k$, and $P_k u(x,\cdot)$ denotes the Taylor polynomial of order $k$, centered at $x$, of the function $u$, \ie using the standard multi-index notation,
\begin{equation*}
    P_k u(x, y):=\sum_{|\zeta | \leq k} \frac{D^\zeta u(x)}{\zeta  !} y^\zeta.
\end{equation*}
\end{lemma}

\begin{proposition} \label{Prop-Lambda-lp}
    Let $\alpha>0$. If $u \in W^{ 2\left \lceil \alpha/2 \right \rceil, p}(\mathbb{R}^{d})$, then $\Lambda^{\alpha} u \in L^p(\mathbb{R}^{d})$ for $1 < p < \infty$. Specifically, the following inequality holds
    \begin{equation*}
        \|\Lambda^{\alpha} u\|_{L^p} \leq C \|u\|_{W^{2\left \lceil \alpha/2 \right \rceil,p}},
    \end{equation*}
    where $C>0$ depends only on $d$, $\alpha$, and $p$. 
\end{proposition}
\begin{proof}
    The cases where $\alpha/2 \in \mathbb{N}$ can be checked directly, as $\Lambda^{\alpha}$ is a local differential operator, representing a power of the standard Laplacian. 
    For $\alpha/2 \in \mathbb{R}_{+}\setminus\mathbb{N}$, from \eqref{eq-definition-fractional-Laplacian-integral}, we write
    \begin{equation*}
        \begin{split}
            \left\| \int_{\mathbb{R}^d} \frac{u(y)-u(x)}{|x-y|^{d+\alpha}} \mathrm{~d} y\right\|_{L^{p}} &= \left\| \int_{|x-y| < 1} \frac{u(y)-u(x)}{|x-y|^{d+\alpha}} \mathrm{~d} y\right\|_{L^{p}} \\&\hspace{3cm} +\left\| \int_{|x-y| \geq 1} \frac{u(y)-u(x)}{|x-y|^{d+\alpha}} \mathrm{~d} y\right\|_{L^{p}} \\
            &= I_1+I_2.
        \end{split}        
    \end{equation*}
    
    Set $k=2\left \lceil \alpha/2 \right \rceil$. Since by hypothesis $u \in W^{k, p}(\mathbb{R}^{d})$, to compute $I_1$, we use the Taylor expansion centered at $x$, \ie $u(x+y)=P_k u(x, y)+r(x,y)$, where $P_k u(x, y)$ is defined in \Cref{Lemma-Taylor-polynomial}. 
    From Minkowski's inequality for integrals, since $\chi_{[-1,1]}(y)\left(\frac{u(\cdot+y)-P_k(\cdot, y)}{|y|^{d+\alpha}}\right)$ $  \in L^{ p}(\mathbb{R}^{d})$ for a.e. $y$, and $\chi_{[-1,1]}(y) \left\|\frac{u(\cdot+y)-P_k(\cdot, y)}{|y|^{d+\alpha}}\right\|_{L^p}  \in L^{ 1}(\mathbb{R}^{d})$, 
    we have 
    \begin{equation*}
        \begin{split}
            \left\| \int_{|y| < 1} \frac{r(x,y)}{|y|^{d+\alpha}} \mathrm{~d} y \right\|_{L^p} & \leq C \int_{|y| < 1}  \frac{\left\|u(\cdot+y)-P_k(\cdot, y)\right\|_{L^p}}{|y|^{d+\alpha}} \mathrm{~d} y \\
            & \leq C \int_{|y| < 1}  \frac{\| u\|_{W^{k,p}}|y|^{k}}{|y|^{d+\alpha}} \mathrm{~d} y \\
            & \leq  \left(\int_{|y| < 1}  \frac{1}{|y|^{d+\alpha-k}} \mathrm{~d} y\right) \| u\|_{W^{k,p}} \\
            & \leq  C \| u\|_{W^{k,p}},  
        \end{split}
    \end{equation*}
    where we used \eqref{Taylor-polynomial} and the fact that $d+\alpha-k<d$.   
    Next, note that the finite difference of order $k$ of $u$ can be expressed as 
    \begin{equation*}
        \delta_{k/2} u(x, y)=-\sum_{|\zeta|=k} y^{\zeta} D^\zeta u +\sum_{l=-k/2}^{k/2}(-1)^l\left(\begin{array}{c}
        k \\
        k/2-l
        \end{array}\right) r(x+l y) \quad \text { for } x, y \in \mathbb{R}^d.
    \end{equation*}   
    Therefore, from the above, we obtain 
    \begin{equation*}
        \begin{split}
             I_1
             & = C\left\| \int_{|y| < 1} \frac{\delta_{k/2} u(x, y) }{|y|^{d+\alpha}} \mathrm{~d} y \right\|_{L^{p}} \\
             & \leq C\left\| \int_{|y| < 1} \sum_{|\zeta|=k} \frac{y^{\zeta} D^\zeta u  }{|y|^{d+\alpha}} +\sum_{l=-k/2}^{k/2}(-1)^l\left(\begin{array}{c}
            k \\
            k/2-l
            \end{array}\right) \frac{r(x+l y)  }{|y|^{d+\alpha}} \mathrm{~d} y \right\|_{L^{p}}  \\
            &\leq  C \left\|  D^k u\right\|_{L^p}\left( \int_{|y| < 1} \frac{1 }{|y|^{d+\alpha-k}} \mathrm{~d} y \right) +  C \| u\|_{w^{k,p}}  \\
            &\leq C \| u\|_{w^{k,p}},
        \end{split}
    \end{equation*}
    where $\left|D^k u\right|=\left(\sum_{|\zeta|=k}\left|D^\zeta u\right|^2\right)^{1 / 2}$, and we used the fact that $1 /|y|^{d+\alpha-k}$ is integrable as $k-\alpha>0$.

    Now applying again Minkowski's inequality, we obtain 
     \begin{equation*}
        \begin{split}
            I_2
            &=\left\| \int_{|z| \geq 1} \frac{u(x+z)-u(x)}{|z|^{d+\alpha}} \mathrm{~d}z\right\|_{L^{p}} \\
            & \leq \int_{|z| \geq 1} \frac{\left\|u(\cdot+z)-u(\cdot)\right\|_{L^{p}}}{|z|^{d+\alpha}} \mathrm{~d} z \\
           & \leq 2 \left(\int_{|z| \geq  1} \frac{1}{|z|^{d+\alpha}} \mathrm{~d} z\right) \left\|u\right\|_{L^{p}} \\
            & \leq C \left\|  u\right\|_{L^p},
        \end{split}
    \end{equation*}
    where we used the fact that $1 /|z|^{d+\alpha}$ is integrable, as $d+\alpha>d$ and $|z|\geq 1$. 
\end{proof}

Now consider the following definition:
\begin{definition} \label{definition-f-f+}
    For a real Lebesgue-measurable function $u$ in $\mathbb{R}^d, d \in \mathbb{N}$, we define $u \mapsto u_{+}$ and $u \mapsto u_{-}$, bounded linear operators in $L^{p}(\mathbb{R}^d)$ spaces, with $1\leq p \leq \infty$, as
\begin{equation}
    u_{+}(x)=\max \{ u(x), 0\}, \quad \text{ and } \quad  u_{-}(x)=\max \{-u(x), 0\},\quad \mbox{ a.e. } x \in \mathbb{R}^d.
\end{equation}
Thus, $u(x)=u_{+}(x)-u_{-}(x)$ a.e. $x \in \mathbb R^d$. 
\end{definition}

For $u \in W^{1,p}(\mathbb{R}^d)$, the weak derivatives of $\partial_i u_{+}$ and $\partial_i u_{-}$ are given by 
\begin{equation}\label{derivpm}
    \partial_i u_{+}(x)= \begin{cases}
        \partial_i u(x), & \text{ if } u(x)>0 \\
        0, & \text{ if } u(x) \leq 0
        \end{cases} \;\text{ and } \; \partial_i u_{-}(x)= \begin{cases}
       - \partial_i u(x),& \text{ if } u(x)\leq 0 \\
        0, & \text{ if }  u(x) >  0.
        \end{cases}
\end{equation}

Then, we can prove for $u$, $u_-$ and $u_+$ the following: 
\begin{lemma} \label{nonnegative-2}
    Let $\alpha \geq 0$ and $u \in \mathcal{S}(\mathbb{R}^d)$. Then,
    \begin{equation}
        \label{part-3}
       - u_{-} \Lambda^{\alpha} u \geq u_{-} \Lambda^{\alpha} u_{-} \quad \text{ and } \quad u_{+} \Lambda^{\alpha} u \geq u_{+} \Lambda^{\alpha} u_{+}.
    \end{equation}
\end{lemma}
\begin{proof}
The cases where $\alpha/2 \in \mathbb{N}$ can be verified directly, as $\Lambda^{\alpha}$ acts as a local differential operator, specifically a power of the Laplacian. 
For $\alpha/2 \in \mathbb{R}_{+}\setminus\mathbb{N}$, from \eqref{eq-definition-fractional-Laplacian-integral}, we obtain 
\begin{equation*}
    \begin{split}
        - u_{-} \Lambda^{\alpha} u(x) & = c_{d, \alpha} P . V . \int _{\mathbb{R}^d} \frac{- u_{-}(x)u(x) -(- u_{-}(x))u(y)} {|x-y|^{d+\alpha}} \mathrm{~d} y \\
        &=c_{d, \alpha} P . V . \int _{\mathbb{R}^d} \frac{( u_{-}(x))^2+ u_{-}(x)\left(u_{+}(y) - u_{-}(y) \right)}{|x-y|^{d+\alpha}} \mathrm{~d} y \\
        & \geq c_{d, \alpha} P . V . \int _{\mathbb{R}^d} \frac{(u_{-}(x))^2-u_{-}(x)u_{-}(y) }{|x-y|^{d+\alpha}} \mathrm{~d} y \\
        &=u_{-} \Lambda^\alpha u_{-}(x).
    \end{split}
\end{equation*}
where we used that $u_{-}(x)u_{+}(x)=0$ and $u_{-}(x)u_{+}(y) \geq 0$. 
\end{proof}

\begin{lemma} \label{Positivity-Lemma-2}
    Let $\alpha\geq 0$, and suppose that $u \in W^{2\left \lceil \alpha/2 \right \rceil,p}(\mathbb{R}^d)$, with $1<p<\infty$ . Then, 
    \begin{equation}
        \label{Positivity-Lemma}
        \int_{\mathbb{R}^d}|u|^{p-2} u \Lambda^\alpha u \mathrm{~d} x \geq 0.
    \end{equation}
    and
    \begin{equation}
        \label{part-4}
\int_{\mathbb{R}^d} \pm  u_{ \pm }^{p-1} \Lambda^\alpha u \mathrm{~d} x \geq 0. 
    \end{equation}
\end{lemma}
\begin{proof}
    The proof of \eqref{Positivity-Lemma} with $\alpha \in [0,2]$ is available in \citet[Lemma 2.5]{A-maximum-principle-Cordoba} for $x \in \mathbb{R}^2$ or $\mathbb{T}^2$. Interestingly, the same proof extends seamlessly to $x \in \mathbb{R}^d$ and $\alpha>2$. 
    This extension holds, particularly for $\alpha>2$, 
    owing to the equivalence between the fractional Laplacian expressions \eqref{eq-definition-fractional-Laplacian-integral-2} and \eqref{eq-definition-fractional-Laplacian-integral}. 
    Note that \eqref{part-4} follows from \eqref{part-3} and \eqref{Positivity-Lemma}. Indeed,  
    \begin{align*}
       \int_{\mathbb{R}^d} \pm  u_{\pm}^{p-1} \Lambda^\alpha u \mathrm{~d} x  \geq \int_{\mathbb{R}^d} u_{\pm}^{p-1}  \Lambda^\alpha u_{\pm} \mathrm{~d} x \geq 0. 
    \end{align*}    
\end{proof}
\section*{Acknowledgments}
The first author received support under the grant 2019/02512-5, S\~ao Paulo Research Foundation (FAPESP). 
The second author was supported by Coordenação de Aperfeiçoamento de Pessoal de Nível Superior - Brasil (CAPES) - Finance Code 001, S\~ao Paulo Research Foundation (FAPESP), Brasil -- Process 2019/02512-5, and by  FAEPEX/FUNCAMP - 519.292.
%

\printbibliography
\end{document}